\documentclass[12pt,reqno]{amsart}
\usepackage{amsmath, amsthm, amssymb, stmaryrd}
\usepackage{hyperref}
\usepackage{enumerate}
\usepackage{url}
\usepackage{mathrsfs}

\usepackage{mathtools}

\usepackage{multirow, array}
\usepackage{placeins}

\usepackage{caption} 
\advance \topmargin by -\headheight
\advance \topmargin by -\headsep
     
\evensidemargin \oddsidemargin
\newtheorem{thm}{Theorem}[section]
\newtheorem{lemma}[thm]{Lemma}

\theoremstyle{definition}
\newtheorem{defn}[thm]{Definition}

\theoremstyle{remark}
\newtheorem{remark}[thm]{Remark}

\numberwithin{equation}{section}

\newcommand*\wrapletters[1]{\wr@pletters#1\@nil}
\def\wr@pletters#1#2\@nil{#1\allowbreak\if&#2&\else\wr@pletters#2\@nil\fi}

\def\eps{\varepsilon}

\def\le{\leqslant} \def\ge{\geqslant}

\def\d{{\,{\bRm d}}}

\def \d{\mathrm{d}}

\def \bA {\mathbf A}

\def \bC {\mathbb C}

\def \bbF {\mathbb F}

\def \bI {\mathbb I}

\def \bN {\mathbb N}
\def \bP {\mathbb P}
\def \bQ {\mathbb Q}
\def \bR {\mathbb R}
\def \bZ {\mathbb Z}

\def \bS {\mathbb S}
\def \bV {\mathbb V}

\def \ba {\mathbf a}
\def \bb {\mathbf b}
\def \bc {\mathbf c}
\def \bd {\mathbf d}

\def \bh {\mathbf h}

\def \bj {\mathbf j}

\def \br {\mathbf r}

\def \bt {\mathbf t}
\def \bu {\mathbf u}
\def \bv {\mathbf v}
\def \bx {\mathbf x}
\def \bw {\mathbf w}
\def \by {\mathbf y}
\def \bz {\mathbf z}

\def \fc {\mathfrak c}
\def \fd {\mathfrak d}

\def \fC {\mathfrak C}

\def \fF {\mathfrak F}

\def \fI {\mathfrak I}
\def \fJ {\mathfrak J}

\def \fR {\mathfrak R}
\def \fS {\mathfrak S}

\def \dim {\mathrm{dim}}

\def \det {\mathrm{det}}
\def \Span {\mathrm{Span}}

\def \vol {\mathrm{vol}}

\def \rad {{\mathrm{rad}}}

\begin{document}
\title[Random Diophantine equations in the primes II]{Random Diophantine equations in the primes II}
\author[Philippa Holdridge]{Philippa Holdridge}
\address{}
\email{holdridge.philippa@renyi.hu}
\thanks{}
\date{}

\begin{abstract}
Let $d\ge 2$ and $n\ge d$ with $(d,n)\notin \{(2,2),(3,3)\}$. We consider homogeneous Diophantine equations of degree $d$ in $n+1$ variables and whether they have solutions in the primes. In particular, we show that a certain local-global principle holds for almost all such equations, following on from previous work of the author \cite{h2023}. We do this by adapting the methods of Browning, Le Boudec and Sawin \cite{bbs2023}, with the main input coming from some results on counting points with prime coordinates in lattices.
\end{abstract}
\maketitle
\setcounter{tocdepth}{1}
\tableofcontents

\section{Introduction}
Let $\bV_{d,n}=\bP^{N_{d,n}-1}(\bQ)$ where $N_{d,n}=\binom{n+d}{d}$. We see that $N_{d,n}$ is the number of monomials of degree $d$ in $n+1$ variables and so we can think of $\bV_{d,n}$ as the space of homogeneous polynomials of degree $d$ in $n+1$ variables $x_{0},\dots,x_{n}$, up to scaling. We let $\nu_{d,n}:\bR^{n+1}\rightarrow \bR^{N_{d,n}}$ be the Veronese embedding, which is defined by listing all the monomials in lexicographical order and let $\bZ_{\mathrm{prim}}^{N}$ denote the set of primitive vectors in $\bZ^{N}$ (by which we mean the integer vectors $\ba$ such that $\gcd(\ba)=1$). To each $V \in \bV_{d,n}$ we can associate a coefficient vector $\ba_{V}\in \bZ_{\mathrm{prim}}^{N_{d,n}}$, which is unique up to sign, and, writing $\langle \cdot, \cdot \rangle$ for the standard Euclidean inner product, identify $V$ with the vanishing locus of the equation
\[\langle\ba_{V}, \nu_{d,n}(\bx)\rangle = 0.\]
We will also frequently use the notation $f_{\ba}(\bx)=\langle\ba, \nu_{d,n}(\bx)\rangle$, where $\ba \in \bR^{N_{d,n}}$. Let
\[\bV_{d,n}(A)=\left\{V\in\bV_{d,n}:\Vert \ba_{V}\Vert\le A\right\},\]
where we always use $\Vert\cdot\Vert$ to denote Euclidean norm unless specified otherwise.
We write $\bA_{\bQ}$ for the ring of rational adeles and $V(\bA_{\bQ})$ for the set of $\bA_{\bQ}$-points of $V$ and define
\[\bV_{d,n}^{\mathrm{loc}}=\left\{V\in\bV_{d,n}:V(\bA_{\bQ})\ne \emptyset\right\}\]
and
\[\bV_{d,n}^{\mathrm{loc}}(A)=\left\{V\in\bV_{d,n}^{\mathrm{loc}}:\Vert \ba_{V}\Vert\le A\right\}.\]
Note that, since $V$ is projective, having $V(\bA_{\bQ})\ne \emptyset$ is equivalent to having a point in $\bR$ and in $\bQ_{p}$ for all $p$.

Browning, Le Boudec and Sawin showed \cite{bbs2023} that when $n\ge d\ge 2$ and $(n,d)\notin \{(2,2),(3,3)\}$ then for almost all $V\in\bV_{d,n}^{\mathrm{loc}}$, we have $V(\bQ)\ne \emptyset$. By ``almost all'', we mean that it holds for a proportion $1-o(1)$ of the $V\in\bV_{d,n}^{\mathrm{loc}}(A)$ as $A\rightarrow\infty$. The main goal of this paper is to prove an analogous result, but instead of looking for rational points on $V$, we look for points with prime coordinates. We note, as in \cite{bbs2023}, that the assumption $n\ge d$ implies that a generic $V\in \bV_{d,n}$ will be a smooth Fano hypersurface.

To this end, let $\mathcal{P}$ be the set of all prime numbers. We embed the set $\mathcal{P}^{n+1}\setminus \Span\{(1,\dots,1)\}$ into $\bP^{n}(\bQ)$ and let $V(\mathcal{P})$ be the set of points on $V$ which lie in the image of this embedding. (The reason for excluding $\Span\{(1,\dots,1)\})$ is to avoid diagonal solutions $(p,\dots,p)$ so that our set maps injectively.) We are interested in when $V(\mathcal{P})\ne \emptyset$. If we let $\bR^{+}=\{x\in\bR:x>0\}$, $\bZ_{p}^{\times}=\{x\in \bQ_{p}: |x|_{p}=1\}$ and $\bA'_{\bQ}=\bR^{+}\times\prod_{p}\bZ_{p}^{\times}\subset \bA_{\bQ}$, then $(\bA'_{\bQ})^{n+1}$ embeds into $\bP^{n}(\bA_{\bQ})$ and we let $V(\bA'_{\bQ})$ denote the points of $V$ which lie in this image. Then $V(\bA'_{\bQ})\ne\emptyset$ if and only if $V$ has a point over $\bR^{+}$ and over $\bZ_{p}^{\times}$ for all $p$.

We now let
\[\bV_{d,n}^{\mathrm{ploc}}=\left\{V\in\bV_{d,n}:V(\bA'_{\bQ})\ne \emptyset\right\}\]
and
\[\bV_{d,n}^{\mathrm{ploc}}(A)=\bV_{d,n}^{\mathrm{ploc}}\cap\bV_{d,n}(A).\]

We use the notation $\log_{(k)} x$ to mean the function defined inductively by $\log_{(0)} x = x$ and
\[\log_{(k+1)} x = \max\{ 1, \log (\log_{(k)} x)\}.\]
So $\log_{(k)} x$ is, for sufficiently large $x$, equal to $\log$ applied $k$ times.

\begin{thm}\label{positivelocaldensity}
Let $d,n\ge 2$. Then for some $\rho_{d,n}^{\mathrm{ploc}}\in \bR$,
\[\frac{\#\bV_{d,n}^{\mathrm{ploc}}(A)}{\#\bV_{d,n}(A)}=\rho_{d,n}^{\mathrm{ploc}}+o(1).\]
Furthermore, we have $\rho_{d,n}^{\mathrm{ploc}}>0$ if and only if $(d,n)\ne (2,2)$.

In particular, when $(d,n)\ne (2,2)$,
\[\#\bV_{d,n}^{\mathrm{ploc}}(A)\gg A^{N_{d,n}}.\]
\end{thm}
Throughout this paper, $d,n$ will be considered fixed, so that any $o(\cdot)$ terms and any implied constants may depend on them.

We remark that $\rho_{2,2}^{\mathrm{ploc}}=0$ follows from \cite[Exemple 4]{s1990}, which states that $\rho_{2,2}^{\mathrm{loc}}=0$, where
\[\rho_{d,n}^{\mathrm{loc}}=\lim_{A\rightarrow\infty}\frac{\#\bV_{d,n}^{\mathrm{loc}}(A)}{\#\bV_{d,n}(A)}.\]

The following is our main result.
\begin{thm}\label{main}
Suppose $n\ge d\ge 2$ with $(n,d)\notin\{(2,2),(3,3)\}$. Then for all
\[0<\kappa<\frac{n-2}{9n+2},\]
we have
\[\frac{1}{\#\bV_{d,n}^{\mathrm{ploc}}(A)}\#\left\{V\in\bV_{d,n}^{\mathrm{ploc}}(A):V(\mathcal{P})=\emptyset\right\}\ll_{\kappa} \frac{1}{(\log_{(2)} A)^{\kappa}}.\]
\end{thm}
In other words, when $n\ge d\ge 2$ and $(d,n)\notin\{(2,2),(3,3)\}$, for almost all homogeneous equations of degree $d$ in $n+1$ variables, if there is a solution in $\bR^{+}$ and in $\bZ_{p}^{\times}$ for all $p$, then the equation has a solution in the primes. In some previous work of the author \cite{h2023} they showed an analogous result for equations of the form $a_{1}x_{1}^{d}+\cdots+a_{s}x_{s}^{d}=0$, under the assumption that $s\ge 3d+1$ when $d$ is even and $s\ge 3d+2$ when $d$ is odd. Both of these results only apply to almost all equations. In contrast, Cook and Magyar showed \cite{cm2014} that for any system of integral polynomial equations, a similar local-global principle holds for solubility in the primes, provided that the number of variables is sufficiently large (in terms of the degree and number of equations) compared to the dimension of a certain variety defined in terms of the Jacobian matrix. The number of variables required here is very large in general, growing like an exponential tower in the degree $d$ even with a single equation. Under the assumption that the system is nonsingular, Liu and Zhao reduced the number of variables to $4^{d+2}d^{2}r^{5}$, where $d$ is the degree and $r$ is the number of equations \cite{lz2023}. In the case of quadratic forms, there has been work on showing a local-global principle for equations $f(\bx)=N$ for a fixed $N\in\bZ$ and $f$ a quadratic form. Liu \cite{l2011} showed this for a wide family of quadratic forms in any number of variables that is even and at least 10. Zhao \cite{z2016} improved this to any indefinite form in at least 9 variables, and finally Green \cite{g2021} showed it for forms in 8 variables, subject to mild non-degeneracy conditions.

Though we do not tackle the degree 1 case, we note that Green and Tao showed \cite[Theorem 1.8]{gt2010} that a wide range of systems of inhomogeneous linear equations satisfy a local-global principle for prime solubility. This includes the case of a single equation $a_{1}x_{1}+\dots+a_{s}x_{s}=N$ when $s\ge3$ and the $a_{i}$ are all nonzero, but does not include, for example, the case $x_{1}-x_{2}=2$, which is equivalent to the twin prime conjecture. 

Following the notation of \cite{bbs2023}, we define, for $N\ge 1$, $\bv\in\bR^{N}$ and $\bc\in\bZ^{N}$,
\begin{equation}\label{defcalC}\mathcal{C}_{\bv}^{(\gamma)}=\left\{\bt\in\bR^{N}:|\langle\bv,\bt\rangle|\le\frac{\Vert \bv\Vert\cdot\Vert\bt\Vert}{2\gamma}\right\},\end{equation}
\[\Lambda_{\bc}=\left\{\by\in\bZ^{N}:\langle\bc,\by\rangle=0\right\},\]
and
\[\Lambda_{\bc}^{(Q)}=\left\{\by\in\bZ^{N}:\langle\bc,\by\rangle\equiv 0\:\text{(mod $Q$)}\right\}.\]
As in \cite[Section 3.1]{bbs2023}, we define a lattice to be a discrete subgroup of $\bR^{N}$. Any lattice $\Lambda$ will have a basis $\bb_{1},\dots,\bb_{r}$ and we say that $r$ is the rank of $\Lambda$ and define the determinant $\det \Lambda$ to be the area of the fundamental parallelepiped spanned by $\bb_{1},\dots,\bb_{r}$. These are both independent of the choice of basis. Note that $\Lambda_{\bc}$ and $\Lambda_{\bc}^{(Q)}$ are both lattices of rank $N-1$ and $N$ respectively. If $\Lambda\subset\bZ^{N}$ then we say that $\Lambda$ is integral, and we say that $\Lambda$ is primitive if it is integral and there is no integral lattice $\Lambda'$ of the same rank such that $\Lambda \subsetneq \Lambda'$. 

We define, analogously to the quantities $\Xi_{d,n}(B)$ and $N_{V}(B)$ from \cite{bbs2023},
\[\Xi'_{d,n}(B)=\left\{\bx\in\mathcal{P}^{n+1}\setminus\Span\{(1,\dots,1)\}:\Vert \bx\Vert \le B^{1/(n+1-d)}\right\},\]
and, for $V\in \bV_{d,n}$,
\[N'_{V}(B)=(\log B)^{n+1}\sum_{\substack{\bx\in\Xi'_{d,n}(B)\\\ba_{V}\in\Lambda_{\nu_{d,n}(\bx)}}} 1.\]
This function counts the number of non-diagonal prime solutions to the equation $f_{\ba_{V}}(\bx)=\langle \ba_{V}, \nu_{d,n}(\bx)\rangle=0$ with a weight of $(\log B)^{n+1}$ to compensate for the sparsity of primes.

We also let
\begin{equation}\label{Wdef} w=\frac{\log_{(2)} B}{\log_{(3)} B},\quad \alpha= \log B, \quad W=\prod_{p\le w} p^{\lceil \log w/\log p\rceil+1},\end{equation}
and we define
\[N^{\mathrm{ploc}}_{V}(B)=(\log B)^{n+1}\frac{\alpha W}{\Vert \ba_{V}\Vert}\sum_{\substack{\bx\in\Xi'_{d,n}(B)\\\ba_{V}\in \Lambda_{\nu_{d,n}(\bx)}^{(W)}\cap \mathcal{C}_{\nu_{d,n}(\bx)}^{(\alpha)}}} \frac{1}{\Vert \nu_{d,n}(\bx)\Vert}.\]

As in \cite{bbs2023}, the general idea of the proof is to to establish an upper bound on the average difference between $N'_{V}$ and $N_{V}^{\mathrm{ploc}}$ and then show that $N_{V}^{\mathrm{ploc}}$ is usually not very small when there are local solutions. This is the object of the next three theorems, which are analogues of \cite[Propositions 2.4, 2.3 and 4.1]{bbs2023} respectively.
\begin{thm}\label{localcountnotsmall}
Let $d\ge 2$ and $n\ge d$ with $(d,n)\ne (2,2)$. Let $\psi,\xi:\bR_{>0}\rightarrow \bR_{>1}$ such that $\psi(A)\le A$, and $\xi(A)\le (\log A)^{9n+1}$. Then we have, for all $\eps>0$
\[\frac{1}{\#\bV_{d,n}^{\mathrm{ploc}}(A)}\cdot\#\left\{V\in\bV_{d,n}^{\mathrm{ploc}}(A):N^{\mathrm{ploc}}_{V}(A\psi(A))\le \frac{\psi(A)}{\xi(A)^{1/2}}\right\}\ll_{\eps} \frac{1}{\xi(A)^{1/(9n+1)-\eps}}.\]
\end{thm}
\begin{thm}\label{countapproxlocalcount}
Let $d\ge 2$ and $n\ge d$ with $(d,n)\notin \{(2,2),(3,3)\}$. Suppose that $m>n+1$, $\eps>0$, $\psi(A)=(\log A)^{m}$ and $0<c<n-2-\eps$. Then we have
\begin{multline*}\frac{1}{\#\bV_{d,n}(A)}\cdot\#\left\{V\in\bV_{d,n}(A):\left|N'_{V}(A\psi(A))-N^{\mathrm{ploc}}_{V}(A\psi(A))\right|>\frac{\psi(A)}{(\log_{(2)} A)^{c/2}}\right\}\\ \ll_{c,\eps}\frac{1}{(\log_{(2)} A)^{n-2-\eps-c}}.\end{multline*}
\end{thm}
\begin{thm}\label{mainvar}
Let $d\ge 2$ and $n\ge d$ with $(d,n)\notin \{(2,2),(3,3)\}$, and let $B=A(\log A)^{m}$ for some $m>n+1$. Then for all $\eps>0$, we have,
\[\frac{1}{\#\bV_{d,n}(A)}\sum_{V\in \bV_{d,n}(A)}\left(N'_{V}(B)-N^{\mathrm{ploc}}_{V}(B)\right)^{2}\ll_{m,\eps} \frac{B^{2}}{A^{2}(\log_{(2)} A)^{n-2-\eps}}.\]
\end{thm}
\begin{proof}[Proof of Theorem \ref{countapproxlocalcount} assuming Theorem \ref{mainvar}]
The quantity we wish to bound is at most
\[\frac{1}{\#\bV_{d,n}(A)}\sum_{V\in \bV_{d,n}(A)}\left(\frac{N'_{V}(A\psi(A))-N^{\mathrm{ploc}}_{V}(A\psi(A))}{\psi(A)/(\log_{(2)} A)^{c/2}}\right)^{2},\]
which by Theorem \ref{mainvar} is at most
\[O\left(\frac{1}{(\log_{(2)} A)^{n-2-\eps-c}}\right).\]
\end{proof}

We now show how Theorem \ref{main} can be deduced from the other theorems, so that it just remains to prove Theorems \ref{positivelocaldensity}, \ref{localcountnotsmall}, and \ref{mainvar}.
\begin{proof}[Proof of Theorem \ref{main} assuming Theorems \ref{positivelocaldensity}, \ref{localcountnotsmall}, and \ref{countapproxlocalcount}]
Clearly if $V\in\bV_{d,n}(A)$ and $N'_{V}(B)\ne 0$ for some $B$ then $V(\mathcal{P})\ne \emptyset$. If $N'_{V}(B)=0$ for some $B$ then by the triangle inequality, for all $t>0$, at least one of $N_{V}^{\mathrm{ploc}}(B)\le t$ or $|N'_{V}(B)-N_{V}^{\mathrm{ploc}}(B)|>t$ must hold. Taking $B=A(\log A)^{n+1+\eps}$ and $t=(\log A)^{n+1+\eps}(\log_{(2)} A)^{-c/2}$ where $\eps>0$ and $0<c<n-2-\eps$, we see that $V$ must lie in one of the two sets whose sizes are bounded in Theorem \ref{localcountnotsmall} and Theorem \ref{countapproxlocalcount} (where we choose $\xi(A)=(\log_{(2)} A)^{c}$ in the former and $m=n+1+\eps$ in the latter). It follows that
\begin{align*}\frac{1}{\#\bV_{d,n}(A)}\#\left\{V\in\bV_{d,n}^{\mathrm{ploc}}(A):V(\mathcal{P})=\emptyset\right\}\ll_{\eps}& \frac{1}{(\log_{(2)} A)^{c/(9n+1)-\eps}}\\+&\frac{1}{(\log_{(2)} A)^{n-2-c-\eps}}.\end{align*}
By Theorem \ref{positivelocaldensity}, we can replace the factor of $1/\#\bV_{d,n}(A)$ with $1/\#\bV_{d,n}^{\mathrm{ploc}}(A)$. If we now choose $c=(n-2)(9n+1)/(9n+2)$ then we get the result.
\end{proof}

\section*{Acknowledgements}
The author would like to thank their PhD supervisor Sam Chow for his help and advice in writing this paper.

Philippa Holdridge is supported by the Warwick Mathematics Institute Centre for Doctoral Training, and gratefully acknowledges funding from the University of Warwick.
\section*{Outline of the proof}
In Section \ref{localsolubledensitysection} we give a proof of Theorem \ref{positivelocaldensity}. This section is largely independent of the later sections.

For Theorems \ref{mainvar} and \ref{localcountnotsmall}, the method of proof will follow \cite{bbs2023} where possible. This relies heavily on the geometry of numbers and at many points requires  the counting of points in lattices. The proof of Theorem \ref{mainvar} follows \cite[Section 4]{bbs2023}. By interchanging the order of summation we will mostly only be counting integral points in lattices, and so fortunately we will be able to use many of the lemmas from \cite{bbs2023} without modification. However, we will in some places need an upper bound on certain counts of prime points, which we will show in the first part of Section \ref{latticeboundsection}. The latter parts of this section are devoted to the application of these bounds to producing replacements for some of the key lemmas of \cite{bbs2023}, and finally deducing Theorem \ref{mainvar} from them.

Section \ref{localcountlowerboundsection} is dedicated to the proof of Theorem \ref{localcountnotsmall}. We first establish a lower bound on a count of prime points in a certain lattice and in a certain region in terms of Archimedean and non-Archimedean factors. We do this by cutting up the region into boxes and applying a version of the Siegel-Walfisz theorem in short intervals. These factors are then bounded below almost always in a similar way to \cite[Section 5]{bbs2023}, except the bounding of the Archimedean factor is more complicated because we have to deal with the cases where the equation has most, but not all, of its real solutions lying outside of the region $(\bR^{+})^{n+1}$.
\section*{Notation}
We use the notation $O(f(x))$ to denote a function $g(x)$ such that, for some constant $C>0$, $|g(x)| \le C f(x)$ for all $x$ for which $f(x)$ is defined. We also write $f(x)\ll g(x)$ if $f(x)=O(g(x))$ and $f(x)\asymp g(x)$ if both $f(x)\ll g(x)$ and $f(x)\gg g(x)$. We write $o(f(x))$ to denote a function $g(x)$ such that $g(x)/f(x)\rightarrow 0$ as $x\rightarrow \infty$ or as $x\rightarrow c$ for some specified $c$.

Given $\bv \in \bR^{N}$, $N \ge 1$, we let $\bv^{\perp}$ be the space of all vectors perpendicular to $\bv$. For $S\subset \bR^{n}$, we write $S^{\perp}$ for the vectors perpendicular to $\bv$ for all $\bv\in S$.

For $r\in \bR^{+}$, $N\ge 1$ we write $\mathcal{B}_{N}(r)$ for the closed Euclidean ball of radius $r$, centre $\mathbf{0}$, in $\bR^{N}$. For a subset $S\subset \bR^{N}$, we let $\mathrm{vol}(S)$ denote its $N$-dimensional volume unless specified otherwise.

Unless specified otherwise, $\tau(n)$ will mean the divisor function
\[\tau(n)=\#\{k\in \bZ^{+}: k \mid n\}\]
and $\omega(n)$ will be the prime divisor function
\[\omega(n)=\#\{p\: \text{prime}: p\mid n\}.\]
These are defined for $n\in \bZ$, where we say that $\tau(0)=\omega(0)=\infty$.
\section{The Density of Locally Soluble Equations}\label{localsolubledensitysection}
In this section we will prove Theorem \ref{positivelocaldensity}. We do this by exhibiting a product formula for the density of $\bV_{d,n}^{\mathrm{ploc}}$. In \cite{bbs2023}, the authors cite a result of Poonen and Voloch \cite[Theorem 3.6]{pv2004} which gives a product formula for the density of $\bV_{d,n}^{\mathrm{loc}}$, using a more general result of Poonen and Stoll \cite[Lemma 20]{ps1999}. However, in \cite{ps1999} and \cite{pv2004} the density is defined using the sup norm, whereas in this paper, and in \cite{bbs2023}, it is defined using the Euclidean norm. Fortunately, \cite[Lemma 20]{ps1999} can easily be generalised to any norm.

Throughout this section, we will use $\mu_{\infty}$ to mean the Lebesgue Measure on $\bR^{N}$ and $\mu_{p}$ to be the Haar measure on $\bQ_{p}^{N}$, normalised so that $\mu_{p}(\bZ_{p}^{N})=1$. For $\bv = (v_{1},\dots,v_{n})\in \bQ_{p}^{N}$, we also define $|\bv|_{p}=\max_{1\le i \le N} |v_{i}|_{p}$. For $\lambda\in \bR$ and $S,T\subset \bR^{N}$, we define $\lambda\cdot S=\{\lambda s: s\in S\}$ and $S\cdot T=\{st:s\in S,\: t\in T\}$.

Let $N\ge 1$, $\mathcal{R}\subset \bR^{N}$ bounded with nonempty interior. Define, for sets $S\subset \bZ^{N}$,
\[\overline{\rho}(S;\mathcal{R})=\limsup_{A\rightarrow \infty} \frac{\#(S\cap A\cdot \mathcal{R})}{\#(\bZ^{N}\cap A\cdot \mathcal{R})}, \quad \underline{\rho}(S;\mathcal{R})=\liminf_{A\rightarrow \infty} \frac{\#(S\cap A\cdot \mathcal{R})}{\#(\bZ^{N}\cap A\cdot \mathcal{R})},\]
and if they are equal, then we define $\rho(S;\mathcal{R})$ to be either. We are now ready to state our generalisation of \cite[Lemma 20]{ps1999}.
\begin{lemma}\label{productformula}
Let $N\ge 1$, $\mathcal{R}\subset [-1,1]^{N}$, $U_{\infty}\subset \bR^{N}$ and for each prime $p$, $U_{p}\subset \bZ_{p}^{N}$. Suppose that $\mathcal{R}$ has nonempty interior and $\mu_{\infty}(\partial \mathcal{R})=0$. Also suppose that $\bR^{+}\cdot U_{\infty}=U_{\infty}$ and $\mu_{\infty}(\partial U_{\infty})=0$ and that for each prime $p$, $\mu_{p}(\partial U_{p})=0$. Let $U_{\infty}^{1}=U_{\infty}\cap \mathcal{R}$ and $s_{\infty}=s_{\infty}(\mathcal{R})=\mu_{\infty}(U_{\infty}^{1})$, $s_{p}=\mu_{p}(U_{p})$ for each prime $p$. Finally, suppose that we have
\begin{equation}\label{Updensitylimit} \lim_{M\rightarrow\infty} \overline{\rho}\left(\left\{ \ba \in \bZ^{N}: \ba \in U_{p} \: \text{for some prime $p>M$}\right\};\mathcal{R}\right)=0.
\end{equation}
Let $S\subset \bZ^{N}$ be the set of all $\ba\in \bZ^{N}$ such that $\ba \notin U_{p}$ for all primes $p$ and also for $p=\infty$. Then $\sum_{p} s_{p}$ converges, $\rho(S;\mathcal{R})$ exists and
\[\rho(S;\mathcal{R})=(1-s_{\infty})\prod_{p\in \mathcal{P}} (1-s_{p}).\]
Furthermore, if the interior of $\mathcal{R}$ contains the origin, then the hypothesis \eqref{Updensitylimit} is independent of the choice of $\mathcal{R}$.
\begin{proof}
When $\mathcal{R}=[-1,1]^{N}$, this is a special case of \cite[Lemma 20]{ps1999}. Looking at the proof of this lemma, we see that the only properties of the set $[-1,1]^{N}$ which are needed are that it is compact and has boundary of measure zero. Hence the result follows when $\mathcal{R}$ is closed. When $\mathcal{R}$ is not closed, the fact that its boundary has measure zero means that we may replace $\mathcal{R}$ with its closure.

For the last part, we note that if the interior of $\mathcal{R}$ contains the origin, then $\mathcal{R}$ contains $[-r,r]^{N}$ for some $r>0$. Then for any set $S$ and any $A>0$, we have $(Ar)\cdot[-1,1]^{N}\subset A\cdot \mathcal{R}\subset A\cdot [-1,1]^{N}$, so
\[\#(S \cap (Ar)\cdot [-1,1]^{N})\le \#(S\cap A\cdot \mathcal{R})\le \#(S\cap A \cdot [1,1]^{N})\]
and it follows that
\[r^{N}\overline{\rho}(S;[-1,1]^{N})\le \overline{\rho}(S;\mathcal{R})\le \overline{\rho}(S;[-1,1]^{N}).\]
Hence, $\lim_{M\rightarrow\infty}\overline{\rho}(\mathcal{S}_{M};\mathcal{R})=0$ if and only if $\lim_{M\rightarrow\infty}\overline{\rho}(\mathcal{S}_{M};[-1,1]^{N})=0$.
\end{proof}
\end{lemma}
\begin{remark}If $\|\cdot\|^{*}$ is some norm on $\bR^{N}$ then we may define
\[\rho(S;\|\ba\|^{*})=\lim_{A\rightarrow\infty} \frac{\#\{\ba \in S: \|\ba\|^{*}\le A\}}{\#\{\ba \in \bZ^{N}: \|\ba\|^{*}\le A\}}\]
and similarly define $\overline{\rho}(S;\|\ba\|^{*})$ and $\underline{\rho}(S;\|\ba\|^{*})$. Then if $\mathcal{R}$ is a closed ball centred on the origin with respect to $\|\cdot\|^{*}$, we see that
\[\rho(S;\mathcal{R})=\rho(S;\|\cdot\|^{*}),\]
and similarly for $\overline{\rho}$, $\underline{\rho}$. The interior of $\mathcal{R}$ must contain the origin, so we find that Lemma \ref{productformula} holds for $\rho(S;\|\cdot\|^{*})$ with the hypothesis \eqref{Updensitylimit} independent of the choice of norm.
\end{remark}
\begin{proof}[Proof of Theorem \ref{positivelocaldensity}]
We use a similar method to the proof of \cite[Lemma 3.6]{pv2004}. We have already dealt with the case $(n,d)=(2,2)$, so suppose $(n,d)\ne (2,2)$. For $\ba\in \bZ^{N_{d,n}}\setminus \{\mathbf{0}\}$, write $V_{\ba}$ for the corresponding element of $\bV_{d,n}$. Recall that this gives a two to one correspondence between $\bZ_{\mathrm{prim}}^{N_{d,n}}$ and $\bV_{d,n}$, so the limit
\[\lim_{A\rightarrow\infty}\frac{\#\bV_{d,n}^{\mathrm{ploc}}(A)}{\#\bV_{d,n}(A)}=\rho_{d,n}^{\mathrm{ploc}}\]
exists if and only if
\begin{equation}\label{densitylimit1}\lim_{A\rightarrow\infty}\frac{\#\left\{\ba\in\bZ_{\mathrm{prim}}^{N_{d,n}}:V_{\ba}\in \bV_{d,n}^{\mathrm{ploc}}(A)\right\}}{\#\left\{\ba\in\bZ_{\mathrm{prim}}^{N_{d,n}}:\Vert\ba\Vert\le A\right\}}\end{equation}
exists, and if they exist then they are equal. It can be shown by a standard argument involving M\"obius inversion (see for example \cite{n1972}), that
\[\#\left\{\ba\in\bZ_{\mathrm{prim}}^{N_{d,n}}:\Vert\ba\Vert\le A\right\}\sim \zeta(N_{d,n})^{-1}\#\left\{\ba\in\bZ^{N_{d,n}}:\Vert\ba\Vert\le A\right\}\]
as $A\rightarrow \infty$. By a similar argument, it can be shown that if there exists $c\in [0,1]$ such that
\[\#\left\{\ba\in\bZ^{N_{d,n}}\setminus \{\mathbf{0}\}:V_{\ba}\in \bV_{d,n}^{\mathrm{ploc}}(A)\right\}\sim cA^{N_{d,n}}\]
then
\[\#\left\{\ba\in\bZ_{\mathrm{prim}}^{N_{d,n}}:V_{\ba}\in \bV_{d,n}^{\mathrm{ploc}}(A)\right\}\sim \frac{c}{\zeta(N^{d,n})}A^{N_{d,n}}.\]
Hence if the limit
\begin{equation}\label{densitylimit2}\lim_{A\rightarrow\infty}\frac{\#\left\{\ba\in\bZ^{N_{d,n}}\setminus \{\mathbf{0}\}:V_{\ba}\in \bV_{d,n}^{\mathrm{ploc}}(A)\right\}}{\#\left\{\ba\in\bZ^{N_{d,n}}:\Vert\ba\Vert\le A\right\}}\end{equation}
exists, then so does \eqref{densitylimit1} and they are equal.

We apply Lemma \ref{productformula} with $\mathcal{R}=\mathcal{B}_{N_{d,n}}(1)$ the unit ball, $U_{\infty}$ the set of $\ba\in\bR^{N_{d,n}}$ such that $f_{\ba}(\bx)=0$ does \emph{not} have a solution in $(\bR^{+})^{n+1}$ and $U_{p}$ the set of $\ba\in\bZ_{p}^{N_{d,n}}$ such that $f_{\ba}(\bx)=0$ does not have a solution in $(\bZ_{p}^{\times})^{n+1}$. Then the set $S$ in Lemma \ref{productformula} is exactly $\{\mathbf{0}\}\cup\{\ba\in\bZ^{N_{d,n}}\setminus \{\mathbf{0}\}:V_{\ba}\in \bV_{d,n}^{\mathrm{ploc}}(A)\}$ so if the hypotheses of the lemma are satisfied, then the limit \eqref{densitylimit2} exists and
\[\rho_{d,n}^{\mathrm{ploc}}=\lim_{A\rightarrow \infty}\frac{\#\bV_{d,n}^{\mathrm{ploc}}(A)}{\#\bV_{d,n}(A)}=(1-s_{\infty})\prod_{p}(1-s_{p}),\]
where $s_{\infty}$ and $s_{p}$ are as in Lemma \ref{productformula}. Since $\sum_{p} s_{p}$ converges, the above product is nonzero if and only if $1-s_{\infty}>0$ and $1-s_{p}>0$ for all $p$.

First, we show that the hypotheses of Lemma \ref{productformula} are satisfied. We begin by showing that $\mu_{p}(\partial U_{p})=0$ for all primes $p$.

We first claim that $U_{p}$ is open. By the compactness of the set $\bZ_{p}^{\times}$ it follows that if $f_{\ba}$ has no zeros in $(\bZ_{p}^{\times})^{n+1}$, then $|f_{\ba}(\bx)|_{p}$ must have a minimum value $m>0$ on this set. For $\bb\in \bZ_{p}^{N_{d,n}}$ with $|\bb|_{p}<m$ and for all $\bx \in (\bZ_{p}^{\times})^{n+1}$, we have $|f_{\bb}(\bx)|_{p}\le |\bb|_{p}<m\le |f_{\ba}(\bx)|_{p}$, so $|f_{\ba+\bb}(\bx)|_{p} = |f_{\ba}(\bx) + f_{\bb}(\bx)|_{p} = |f_{\ba}(\bx)|_{p} >0$ by standard properties of $|\cdot|_{p}$. Hence $U_{p}$ is open. Now suppose that $\ba \in \bZ_{p}^{N_{d,n}}\setminus U_{p}$. Then there is some $\by\in(\bZ_{p}^{\times})^{n+1}$ such that $f_{\ba}(\by)=0$. We want to lift this $\by$ to a solution of $f_{\ba+\bb}(\bx)=0$ when $\bb$ is sufficiently small. When at least one of the partial derivatives $\partial_{i}f_{\ba}(\by)$ is nonzero and $\bb$ is sufficiently small relative to it, Hensel's lemma allows us to lift $\by$ to a solution $\bz$ and also ensure that $\bz\in (\bZ_{p}^{\times})^{n+1}$. This means that $\ba$ has a neighbourhood contained in $\bZ_{p}^{N_{d,n}} \setminus U_{p}$ and so does not lie in the closure of $U_{p}$. Hence, we must have $\partial U_{p}\subset \{\ba \in \bQ_{p}^{N_{d,n}}: \exists\bx \in \bQ_{p}^{n+1}\setminus\{0\},\:\nabla f_{\ba}(\bx)=\mathbf{0}\}$. By a result in elimination theory, \cite[Theorem 5.7A]{hartshorne}, we have that this is an algebraic set (i.e. it is given by the set of common zeros of some polynomials). It can be shown by Fubini's theorem and induction on $N$, that an algebraic set in $\bQ_{p}^{N}$ is either the whole space or has measure 0, so it just remains to show that there is at least one $\ba$ where there is no nonzero solution to $\nabla f_{\ba}(\bx)=\mathbf{0}$. One such example is $f_{\ba}(\bx)=x_{0}^{d}+...+x_{n}^{d}$.

Now we show that $\mu_{\infty}(\partial U_{\infty})=0$. If $\ba$ is such that there is no solution to $\bx \in \bS^{n}\cap \bR_{\ge 0}^{n+1}$, then by compactness, $f_{\ba}$ must have a minimum value $m>0$ on this set. Hence, as in the $U_{p}$ case, there is a neighbourhood of $\ba$ which is contained in $U_{\infty}$. This means that, denoting the interior of $U_{\infty}$ by $U_{\infty}^{\mathrm{o}}$,
\begin{equation}\label{interiorUinf}U_{\infty}\setminus U_{\infty}^{\mathrm{o}} \subset \left\{\ba\in\bR^{N_{d,n}}:\begin{array}{l l}\exists \bx\in\bS^{n} \cap (\bR_{\ge 0})^{n+1}\: f_{\ba}(\bx)=0\\ \nexists \bx\in(\bR^{+})^{n+1}\: f_{\ba}(\bx)=0\end{array}\right\}.\end{equation}
This set is similar to the set $\mathscr{B}_{d,n,0}$ which we will define later in \eqref{scrBdef}. In Lemma \ref{badvolbound}, it will be proven that $\mathscr{B}_{d,n,0}$ has measure zero, and it can easily be deduced from this that the right hand side of \eqref{interiorUinf} has measure zero.

For the closure, if $\ba \in \bR^{N_{d,n}}$ such that there exists $\bx \in (\bR^{+})^{n+1}$ with $f_{\ba}(\bx)=0$ and $\nabla f_{\ba}(\bx) \ne \mathbf{0}$, then we claim that for $\bb$ sufficiently small (in terms of $\ba$), there will be a solution $\by$ to $f_{\ba+\bb}(\by)=0$ with $\by \in (\bR^{+})^{n+1}$. From this, it follows that $\ba$ is not in the closure of $U_{\infty}$. To prove the claim, let $\bd=\nabla f_{\ba}(\bx)$ and $\by=\by(t)=\bx+t\bd$ for some $|t|\ll 1$ to be chosen. We have that $\nabla f_{\ba+\bb}(\bx)=\bd+O(\|\bb\|)$, and if $\|\bb\|\ll 1$ then all higher order derivatives are $O(1)$. So, Taylor expanding in the variable $t$ and recalling that $f_{\ba}(\bx)=0$, we have
\[f_{\ba+\bb}(\by)=f_{\bb}(\bx)+t\langle \bd,\bd\rangle+O(|t|\|\bb\|)+O(|t|^{2}).\]
If we consider some small $\delta>0$ and suppose $\|\bb\|,|t|\le \delta$, then the error terms are $O(\delta^{2})$. Let $t_{0}=-f_{\bb}(\bx)/\langle \bd,\bd\rangle$ and for $C>0$ some constant, let $t_{\pm}=t_{0}\pm C\delta^{2}$ and suppose that $\|\bb\|$ and $\delta$ are small enough in terms of $C$ that $|t_{\pm}|\le \delta$. Then we have $f_{\ba+\bb}(\by(t_{\pm}))=\pm C \delta^{2}+O(\delta^{2})$, so for $C$ large enough, $f_{\ba+\bb}(\by(t_{+}))>0$ and $f_{\ba+\bb}(\by(t_{-}))<0$. By the intermediate value theorem, there exists $t$ such that $f_{\ba+\bb}(\by)=0$. Also, since, $\|\by-\bx\|\le \delta$, if $\delta$ is small enough then $\by\in (\bR^{+})^{n+1}$. This proves the claim.

Hence, writing $\overline{U_{\infty}}$ as the closure of $U_{\infty}$:
\[\overline{U_{\infty}} \setminus U_{\infty}\subset \left\{ \ba\in\bR^{N_{d,n}}:\exists \bx\in \bR^{n+1}\setminus \{\mathbf{0}\}\: \begin{array}{l l}f_{\ba}(\bx)=0,\\ \nabla f_{\ba}(\bx)=\mathbf{0}\end{array} \right\}.\]
By \cite[Theorem 5.7A]{hartshorne} again, this can be shown to have measure 0. Hence $\partial U_{\infty}$ has measure 0.

To show \eqref{Updensitylimit}, we need to show that the density of the set
\[\mathcal{S}_{M}=\{\ba\in\bZ^{N_{d,n}}:\ba\in U_{p}\text{ for some finite $p>M$}\}\]
tends to $0$ as $M\rightarrow\infty$. Since \eqref{Updensitylimit} is independent of $\mathcal{R}$ as long as the interior of $\mathcal{R}$ contains the origin, we may consider $\overline{\rho}(\mathcal{S}_{M};[-1,1]^{N})$, which means we can use results from \cite{ps1999}.
If the (mod $p$) reduction of the hypersurface $f_{\ba}=0$ has a smooth point in $(\bbF_{p}^{\times})^{n+1}$, then by Hensel's lemma this lifts to a point in $(\bZ_{p}^{\times})^{n+1}$, which means $f_{\ba}$ is solvable over $\bZ_{p}^{\times}$. If $f_{\ba}$ is absolutely irreducible and $p$ is sufficiently large, then the Lang-Weil estimate \cite[Theorem 1]{lw1954} implies that the number of points over $\bbF_{p}$ is
\[p^{n}+O(p^{n-1/2}),\]
with the implied constant depending at most on $n$ and $d$. Then \cite[Lemma 1]{lw1954} tells us that a variety (not necessarily irreducible), of dimension $r$ which is defined by polynomials over $\bbF_{p}$ of total degree at most $m$ will have at most $O_{r,m}(p^{r})$ points over $\bbF_{p}$. When $p>d$, the singular points on $f_{\ba}=0$ define a proper subvariety, which will have dimension at most $n-1$ and so contains $O(p^{n-1})$ points over $\bbF_{p}$. Similarly, since $d>1$, the intersection of $f_{\ba}=0$ with any of the hyperplanes $x_{i}=0$ is a subvariety of dimension $n-1$ and so each of these will contain $O(p^{n-1})$ points over $\bbF_{p}$.

It follows that the number of smooth points of $f_{\ba}=0$ over $\bbF_{p}^{\times}$ is
\[p^{n}+O(p^{n-1/2}),\]
which is strictly positive when $p$ is sufficiently large in terms of $d$, $n$.
It follows from our earlier discussion that when $p$ is sufficiently large, $\ba\in U_{p}$ implies that $f_{\ba}$ is not absolutely irreducible (mod $p$). We claim that there exist polynomials $g_{1},\dots,g_{m}\in \bZ[A_{0},\dots,A_{N_{d,n}-1}]$ such that for any field $K$, $f_{\ba}$ is reducible over $\overline{K}$ if and only if $g_{1}(\ba)=\cdots=g_{m}(\ba)=0$ in $K$. This follows from a result of Noether (see \cite[Chapter V, Theorem 2A]{s1976}), which gives a similar result for polynomials that are not assumed to be homogeneous. In particular, this result gives us integer polynomials $\tilde{g}_{1},\dots,\tilde{g}_{m}$ in variables $B_{i_{0},\dots,i_{n}}$, for $i_{0}+\cdots+i_{n}\le d$ which vanish if and only if the polynomial $\sum B_{i_{0},\dots,i_{n}}x_{0}^{i_{0}}\cdots x_{n}^{i_{n}}$ either is reducible over $\overline{K}$ or has degree at most $d-1$. The claim follows by setting the variables $B_{i_{0},\dots,i_{n}}=0$ for $i_{0}+\cdots+i_{n}<d$.

We may suppose that $g_{1},\dots,g_{m}$ are pairwise relatively prime. In the proof of \cite[Theorem 3.6]{pv2004} it is shown that the variety consisting of $\ba$ such that $f_{\ba}$ is reducible over $\overline{\bQ}$ has codimension at least $2$, which implies that $m\ge 2$. Hence, \cite[Lemma 21]{ps1999} implies that $\overline{\rho}(\mathcal{S}_{M};[-1,1]^{N})$ tends to $0$ as required.

It just remains to show that $1-s_{\infty}>0$ and $1-s_{p}>0$ for all primes $p$. Consider the vector $\ba$ corresponding to $f_{\ba}(\bx)=x_{0}^{d}-x_{1}^{d}$. This has a solution $(1,\dots,1)$ and if $|\bb|_{p}$ is sufficiently small then this may be lifted by Hensel's lemma to a solution to $f_{\ba+\bb}(\bx)=0$ in $(\bZ_{p}^{\times})^{n+1}$. Hence the complement of $U_{p}$ contains an open ball and so has positive measure, so $1-s_{p}>0$. For $U_{\infty}$ we use $f_{\ba}=(x_{0}^{d}-x_{1}^{d})/2$ and the argument is similar, but we use the mean value theorem and intermediate value theorem in place of Hensel's lemma.
\end{proof}
\section{Bounds on Prime Vectors in Lattices}\label{latticeboundsection}
In this section we will prove Theorem \ref{mainvar}. As mentioned earlier, we use the method of \cite[Section 4]{bbs2023}, but first we need an upper bound for counting prime points in lattices. For this, we use a modified form of a result of Tulyaganova and Fainleib \cite{tf1993} (which uses Selberg's sieve). In the rest of the paper, $p$ will always mean a prime variable unless stated otherwise.

Given an integer $m\times n$ matrix $A$, and integer vector $\bb\in \bZ^{m}$ we let $A'$ be the $m\times (n+1)$ matrix made by attaching $\bb$ to $A$ as an additional column. When the $j$th row of $A$ is nonzero, we let $d_{j}$ be the gcd of its entries and otherwise we let $d_{j}=0$. Similarly, we let $d_{i,j}$ be the gcd of the $2\times 2$ minors of the $i$th and $j$th rows of $A'$ when these numbers are not all $0$, and $d_{i,j}=0$ otherwise (note that $d_{i,j}\ne 0$ if and only if the $i$th and $j$th rows are linearly independent). We then define the quantity $g(A,\bb)$ to be $\prod_{j=1}^{m} d_{j}\prod_{1\le i<j\le m} d_{i,j}$. If $g(A,\bb)\ne 0$ then we define
\begin{equation}\label{fcdef}\fc(A,\bb) = \exp \left(m\sum_{p\mid g(A,\bb)} p^{-1}\right),\end{equation}
and otherwise we let $\fc(A,\bb)=\infty$.

If we let $\omega_{t}(k)=\sum_{p\mid k} p^{t}$ for $t \in \bR$, $k\in \bZ \setminus{\{0\}}$ and $\omega_{t}(0)=\sum_{p} p^{t}$ (which may be infinite), then we have
\[\fc(A,\bb)=\exp\left(m\omega_{-1}(g(A,\bb))\right).\]
We have $g(A,\bb)\ll (\Vert A\Vert+\Vert \bb\Vert)^{O(1)}$. We also have, letting $p_{j}$ be the $j$th prime and $\omega(k)=\omega_{0}(k)$, if $k\ne 0$ then:
\[\omega_{-1}(k)\le \sum_{j=1}^{\omega(k)} \frac{1}{p_{j}} \le \log_{(2)}\omega(k)+O(1)\le \log_{(3)} k+O(1).\]
Therefore, if $\fc(A,\bb)\ne \infty$ then
\begin{equation}\label{fcgeneralupperbound}\fc(A,\bb)\ll (\log_{(2)} (\Vert A\Vert +\Vert \bb\Vert))^{m}.\end{equation}
However, $\omega(k)$ is usually of size $O(\log_{(2)} k)$, so we would expect that for most $A$ and $\bb$, this bound can be significantly improved to $\fc(A,\bb)\ll (\log_{(3)} (\Vert A\Vert +\Vert \bb\Vert))^{m}$. We will make this more precise in Lemma \ref{badlatticecount}.

The following lemma is a modified form of \cite[Theorem 1]{tf1993}.
\begin{lemma}\label{matrixprimecount}
Let $A$ be an $m\times n$ integer matrix, $\by=(y_{1},\dots,y_{n})$, $\bb=(b_{1},\dots,b_{m})$ be integer vectors, $N_{1},\dots,N_{n}\ge 2$ and $N=\min_{j} N_{j}$. 

Define $\pi(N_{1},\dots,N_{n},\by,A,\bb)$ to be the number of $\bx=(x_{1},\dots,x_{n})\in\bZ^{n}$ such that $|x_{j}-y_{j}|\le N_{j}$ for each $j$ and $A\bx+\bb\in\mathcal{P}^{m}$.

Then
\[\pi(N_{1},\dots,N_{n},\by,A,\bb)\ll \frac{N_{1}\cdots N_{n}}{(\log N)^{m}}\fc(A,\bb),\]
where the implicit constant depends only on $m$ and $n$.
\begin{proof}
Without loss of generality, assume that $N_{1}\le N_{2}\le \cdots \le N_{n}$ so that $N=N_{1}$. We may also assume that $\fc(A,\bb)<\infty$, which implies that the rows of $A$ are nonzero and the rows of $A'$ are pairwise linearly independent.

Write $a_{i,j}$ for the $i,j$-entry of $A$ and $\br_{i}$ for the $i$th row. Then $\pi(N_{1},\dots,N_{n},A,\bb)$ can be thought of as counting the values of $\bx=(x_{1},\dots,x_{n})$ such that $a_{i,1}x_{1}+\cdots+a_{i,n}x_{n}+b_{i}=\langle \bx, \br_{i}\rangle+b_{i}$ is simultaneously prime for all $i$.

Suppose that one of the rows of $A'$ is not a primitive vector. Then for some $i$, $\gcd(a_{i,1},\dots,a_{i,n},b_{i})=p>1$, where we may assume that $p$ is prime because $p$ always divides $\langle\bx, \br_{i}\rangle+b_{i}$, and so if $p$ isn't prime then $\pi$ is identically zero. Then $\langle\bx, \br_{i}\rangle+b_{i}$ is prime if and only if $\langle\bx, \br_{i}\rangle+b_{i}=p$, which means $\bx$ lies in an affine hyperplane, and it is not hard to show that the number of $\bx$ with $|x_{j}-y_{j}|\le N_{j}$ which lie on a given affine hyperplane is at most $O(N_{2}\cdots N_{n})$, which is small enough. Hence, we may suppose that all of the rows of $A'$ are primitive.

We are now in the situation of what the authors of \cite{tf1993} call the class $P$, and so we would like to use the main result of that paper. We let $\lambda_{j}(\bv)$ denote the $j$th coordinate of a vector $\bv$ and define $\delta(p)$ to be the number of $\bx\in(\bZ/p\bZ)^{n}$ such that $\lambda_{j}(A\bx + \bb) \equiv 0$ (mod $p$) for some $j$. Then the following holds when $N_{1}=\cdots=N_{n}$:
\[\pi(N_{1},\dots,N_{n},\by,A,\bb)\ll \frac{N_{1} \cdots N_{n}}{(\log N_{1})^{m}}\exp \left(\sum_{p\mid g(A,\bb)}\frac{mp^{n-1}-\delta(p)}{p^{n}}\right).\]
What we actually need is for the above to hold for all $N_{1}\le \cdots \le N_{n}$ not necessarily equal, but this can be achieved by essentially the same argument as in \cite{tf1993}. Most of the changes are merely notational, such as replacing $\omega(p)$ with $\delta(p)$, $x_{0}$ with $\by$, $n$ with $m$, $k$ with $n$ and transposing the matrix $A$. The main lemma requires no modification. In the proof of Theorem 1, the only other changes necessary are to replace occurrences of the condition $\| x - x_{0} \| \le N/2$ (or $\|\bx -\by\| \le N/2$ in our notation) with the condition that for each $j$, the $j$th coordinate of $x -x_{0}$ lies in the interval $[-N_{j}/2,N_{j}/2]$, and to replace occurrences of $N^{k}$ with $N_{1}\cdots N_{n}$ and $N^{k-1}$ with $N_{2} \cdots N_{n}$. To complete the proof, note that $\delta(p)\ge 0$ and so
\[\exp\left(\sum_{p\mid g(A,\bb)} \frac{mp^{n-1}-\delta(p)}{p^{n}}\right)\le \fc(A,\bb).\]
\end{proof}
\end{lemma}

Given a lattice $\Lambda$ of rank $r$ and $1\le i\le r$, we define:
\[\lambda_{i}(\Lambda)=\inf\{u\in\bR^{+}:\dim(\Span_{\bR}(\Lambda\cap\mathcal{B}_{N}(u)))\ge i\}.\]
We call this the $i$th successive minimum.

\begin{lemma}\label{basiscoeffbound}
Suppose $\Lambda\subset \bZ^{m}$ is a primitive lattice of rank $r$. Then there exists a basis $\bv_{1},\dots,\bv_{r}$ for $\Lambda$ such that
\begin{equation}\label{basissize}\| \bv_{i}\| \asymp \lambda_{i}(\Lambda) \:\text{for $1\le i \le r$}
\end{equation}
with implied constants depending only on $m$. We may also ensure that the distance from $\bv_{i}$ to the space $(\Span_{\bR}(\{\bv_{1},\dots,\bv_{i-1}\}))^{\perp}$ is at most
\begin{equation}\label{basisperpbound}\frac{1}{2}\left(\|\bv_{1}\|+\dots+\|\bv_{i-1}\|\right).
\end{equation}

Now let $\bv_{1},\dots,\bv_{r}$ be any basis for $\Lambda$ that satisfies \eqref{basissize}. For $\bx\in \Lambda$ with $\Vert \bx\Vert \le X$ we may write it as $\bx=x_{1}\bv_{1}+\cdots+x_{r}\bv_{r}$ (with $x_{j}\in\bZ$ for all $j$). Then
\begin{equation}\label{latticecoordinatebound}|x_{j}|\ll \frac{X}{\lambda_{j}(\Lambda)}\end{equation}
for all $1\le j\le r$. The implied constants depend only on $m$ and on the implied constants in \eqref{basissize}.
\begin{proof}
Let $\bv_{1}$ be a nonzero vector of minimal length in $\Lambda$. We say that a set of vectors $\bc_{1},\dots,\bc_{k}\in\Lambda$ is primitive if for any $a_{1},\dots,a_{k}\in\bR$ such that $a_{1}\bc_{1}+\cdots+a_{k}\bc_{k}\in \Lambda$, we have $a_{1},\dots,a_{k}\in \bZ$. Following the construction given at the beginning of \cite[Lecture X, section 5]{sc1989} we may inductively define a ``reduced basis''. Indeed, given $\bv_{1},\dots,\bv_{i-1}$, we consider the (nonempty) set of $\bv\in \Lambda$ such that $\bv_{1},\dots,\bv_{i-1},\bv$ is primitive, and then let $\bv_{i}$ be of minimal length in this set. It follows from \cite[Lecture X, Lemma 2]{sc1989} that $\Vert \bv_{i}\Vert\le (3/2)^{i-1}\lambda_{i}(\Lambda)$ for $1\le i\le r$ (though the lemma is only stated for $\Lambda$ of full rank, the proof still works in the more general case). We also have the trivial lower bound $\Vert \bv_{i}\Vert\ge \lambda_{i}(\Lambda)$, so $\Vert\bv_{i}\Vert\asymp \lambda_{i}(\Lambda)$.

To obtain \eqref{basisperpbound}, we proceed by induction and suppose that the bound holds for $i$ up to some $j<r$. Then we may write $\bv_{j+1}=\bu+t_{1}\bv_{1}+\dots+t_{j}\bv_{j}$ for some $\bu\in (\Span\{\bv_{1},\dots,\bv_{j}\})^{\perp}$ and $t_{i}\in \bR$. We then let $n_{i}$ be the closest integer to $t_{i}$ and replace $\bv_{j+1}$ with $\bw=\bv_{j+1}-n_{1}\bv_{1}-\dots-n_{j}\bv_{j}$. This doesn't change the span and we still have $\|\bw\| \asymp \lambda_{j+1}(\Lambda)$ after possibly modifying the implied constants. Clearly $\|\bw-\bu\| \le (\|\bv_{1}\|+\dots+\|\bv_{j}\|)/2$ which gives the bound \eqref{basisperpbound}.

For the bound on $x_{i}$ we follow a similar argument to the proof of \cite[Lemma 5, part (iii)]{d1963}. We use Minkowski's second theorem (see for example \cite[Chapter VIII, Theorem V]{c1997}), which states that, for any lattice $\Lambda$ of rank $r$,
\begin{equation}\label{minkowski2}\det(\Lambda)\le \lambda_{1}(\Lambda)\cdots \lambda_{r}(\Lambda)\ll_{r}\det(\Lambda).\end{equation}
Let $\Lambda_{j}$ be the integer span of the vectors $\{\bv_{i}: 1\le i \le r, \: i\ne j\}$. Recalling the definition of $\det \Lambda$ as the volume of a parallelepiped, we see that the perpendicular distance of $\bv_{j}$ to the vector space spanned by $\Lambda_{j}$ is exactly $\det \Lambda / \det \Lambda_{j}$. For all $j$, the perpendicular distance of $\bx$ to this vector space is therefore $|x_{j}| \det \Lambda / \det \Lambda_{j}$. The perpendicular distance from $\bx$ to any vector space is at most $\bx$, so $|x_{j}| \det \Lambda / \det \Lambda_{j} \le \|\bx\|$. Since $\{\bv_{i}: 1\le i\le r, \: i\ne j\}$ is a basis for $\Lambda_{j}$ it follows that
\[\det \Lambda_{j}\le \prod_{\substack{1\le i \le r\\ i\ne j}} \|\bv_{i}\| \ll \prod_{\substack{1\le i \le r\\ i\ne j}} \lambda_{i}(\Lambda).\]
Then by Minkowski's second theorem,
\[\frac{\det \Lambda}{\det \Lambda_{j}} \gg \lambda_{j}(\Lambda),\]
so
\[X \ge \| \bx \| \gg |x_{j}| \lambda_{j}(\Lambda).\]
\end{proof}
\end{lemma}

Now suppose we have a primitive lattice $\Lambda\subset \bZ^{m}$ of rank $r$. We define $\pi(\Lambda,X)=\#\left(\mathcal{B}_{m}(X)\cap \Lambda\cap \mathcal{P}^{m}\right)$. For $\bv_{1},\dots,\bv_{r}$ given by Lemma \ref{basiscoeffbound}, we may form a matrix $A$ by taking its columns to be $\bv_{1},\dots,\bv_{r}$. We then define $\fc(\Lambda)$ to be the largest value of $\fc(A,\mathbf{0})$ as $\bv_{1},\dots,\bv_{r}$ ranges over all the possible choices of basis which satisfy \eqref{basissize} and \eqref{basisperpbound}. Whenever we choose a basis $\bv_{1},\dots,\bv_{r}$ coming from Lemma \ref{basiscoeffbound}, we may suppose that $\fc(A,\mathbf{0})=\fc(\Lambda)$.

\begin{lemma}\label{latticeprimecount}
Suppose $\Lambda$ is a primitive lattice of rank $r$ lying in $\bZ^{m}$. Suppose $\eps>0$, $1\le i_{0}\le r$ and $\lambda_{i_{0}}(\Lambda)\le X^{1-\eps}$. Then
\[\pi(\Lambda,X)\ll_{\eps} \frac{X^{r}}{(\log X)^{m}\lambda_{1}(\Lambda)\cdots \lambda_{i_{0}-1}(\Lambda)\lambda_{i_{0}}(\Lambda)^{r-i_{0}+1}}\fc(\Lambda),\]
where the implied constant depends only on $m$, $r$ and $\eps$.
\begin{proof}
Pick a basis $\bv_{1},\dots,\bv_{r}$ provided by Lemma \ref{basiscoeffbound}. If $A$ is the $m\times r$ matrix whose columns are the $\bv_{j}$ then every $\bz\in \Lambda$ can be written (uniquely) as $A\bx$ where $\bx=(x_{1},\dots,x_{r}) \in \bZ^{r}$ and $\bz=x_{1}\bv_{1}+\cdots+x_{r}\bv_{r}$. If $\Vert \bz \Vert \le X$ then by Lemma \ref{basiscoeffbound}, there exists a constant $C$ such that whenever $\lambda_{j}(\Lambda)\le X$, we have $|x_{j}|\le CX/\lambda_{j}(\Lambda)$. We then apply Lemma \ref{matrixprimecount} with
\[N_{j}=\frac{CX}{\lambda_{j}(\Lambda)}\]
for $j\le i_{0}$, and $N_{j}=N_{i_{0}}$ for all larger $j$. The reason for this choice is so that $N_{j}\ge X^{\eps}$ for all $j$. We have
\[\pi(\Lambda,X)\le \pi(N_{1},\dots,N_{r},\mathbf{0},A,\mathbf{0})\ll \frac{N_{1}\cdots N_{r}}{(\eps \log X)^{m}}\fc(A,\mathbf{0}).\]
We also have that $\fc(A,\mathbf{0})\le\fc(\Lambda)$. The result now follows by the definition of $N_{j}$.
\end{proof}
\end{lemma}

Let $S_{r,n}(s_{1},..,s_{r})$ be the set of all primitive lattices $\Lambda\subset \bZ^{n+1}$ of rank $r$ such that $\lambda_{j}(\Lambda)\in(s_{j}/2,s_{j}]$ for all $j\in\{1,\dots,r\}$ and define $\widehat{S}_{r,n,c,T}(s_{1},\dots,s_{r})$ to be the set of all $\Lambda\in S_{r,n}(s_{1},\dots,s_{r})$ such that $\fc(\Lambda)\ge c(\log_{(3)} T)^{n+1}$. We wish to bound the size of $\widehat{S}$ and for our purposes, it suffices to do so in the rank $2$ case, which we do in Lemma \ref{badlatticecount}. First we need some bounds for sums involving the divisor function.
\begin{lemma}\label{divisorsumlemma}
 Let $Z\ge 1$ and $t\in \bN$. Then
 \begin{equation}\label{powerdivisorsum}\sum_{1\le z\le Z} \tau(z)^{t}\ll_{t} Z(\log Z)^{2^{t}-1}.\end{equation}
Furthermore, if $U_{1},V_{1},U_{2},V_{2}\ge 1$ such that $U_{1}V_{2}=U_{2}V_{1}=Z$. Then
\begin{equation}\label{divisordetsum}\sum_{\substack{1\le |x|\le U_{1},\: 1\le |y|\le U_{2}\\1\le |z|\le V_{1},\: 1\le |w|\le V_{2}\\xw-yz\ne 0}} \tau(xw-yz)\ll Z^{2}(\log Z)^{4}.\end{equation}
\begin{proof}
The first part is proved in, for example, \cite[Lemma 2.5]{hua}.

For \eqref{divisordetsum}, we have,
\begin{align*}\sum_{\substack{1\le |x|\le U_{1},\: 1\le |y|\le U_{2}\\1\le |z|\le V_{1},\: 1\le |w|\le V_{2}\\xz-yw\ne 0}} \tau(xw-yz)&\le 4\sum_{\substack{1\le |u|,|v|\le Z\\u\ne v}} \tau(u)\tau(v)\tau(u-v)\\&= 4\sum_{1\le |u|\le Z}\tau(u)\sum_{\substack{1\le |v|\le Z\\v\ne u}}\tau(v)\tau(u-v).\end{align*}
We now finish by applying the Cauchy-Schwarz inequality to the inner sum and using \eqref{powerdivisorsum} with $t=2$ and $t=1$.
\end{proof}
\end{lemma}

\begin{lemma}\label{badlatticecount}
Let $n\ge 2$ and $K>0$. Then there exists $c>0$ such that for all $1\le s_{1}\le s_{2}\le T$,
\[\#\widehat{S}_{2,n,c,T}(s_{1},s_{2})\ll_{K} \frac{s_{1}^{n+2}s_{2}^{n}}{(\log T)^{K}}+s_{1}^{n+1}s_{2}^{n}.\]
\begin{proof}
Given some $\Lambda\in \widehat{S}_{2,n,c,T}(s_{1},s_{2})$, we get a basis $\bv_{1},\bv_{2}$ from Lemma \ref{basiscoeffbound} such that $\| \bv_{i}\| \asymp \lambda_{i}(\Lambda) \asymp s_{i}$ for $i \in \{1,2\}$ and $\bv_{2}$ lies within a distance $\| \bv_{1}\|$ of the hyperplane $V_{1}=\bv_{1}^{\perp}$. Let $0<\alpha<\beta$ such that $\alpha s_{j}\le \Vert \bv_{j}\Vert\le \beta s_{j}$ for $j\in\{1,2\}$. Let $\fc(\bv_{1},\bv_{2})=\fc(A,\mathbf{0})$ where $A$ is the $(n+1)\times 2$ matrix whose columns are $\bv_{1},\bv_{2}$. The quantity $\#\widehat{S}_{2,n,c,T}(s_{1},s_{2})$ is therefore bounded above by the cardinality of the set $\mathcal{N}$ of linearly independent pairs of vectors $\bv_{1},\bv_{2}$ such that $\alpha s_{j}\le \Vert \bv_{j}\Vert \le \beta s_{j}$ for $j\in \{1,2\}$ and $\bv_{2}$ is within a distance $\beta s_{1}$ of the hyperplane $V_{1}$ and $\fc(\bv_{1},\bv_{2})\ge c(\log_{(3)} T)^{n+1}$.

Recalling the definition of $g=g(A,\mathbf{0})$ as a product of numbers $d_{j}$, $d_{i,j}$, we see that if $\fc(\bv_{1},\bv_{2})$ is large then at least one of these $d_{j}$ or $d_{i,j}$ must have a lot of distinct prime factors. More precisely, we claim that if $c$ and $T$ are sufficiently large and $\fc(\bv_{1},\bv_{2})\ge c(\log_{(3)} T)^{n+1}$, then
\begin{align}\label{frakcsizeclaim} &\tau(d_{j})\ge (\log T)^{c\log 2}\:\text{for some $1\le j \le n+1$,}\\
\nonumber\text{or}\:&\tau(d_{i,j})\ge (\log T)^{c\log 2}\:\text{for some $1\le i<j \le n+1$.}
\end{align}
To see this, first 
\[\omega_{-1}(g)\ge \log_{(4)} T+\frac{1}{n+1}\log c.\]
Note that $\omega_{-1}(0)=+\infty$. Now recalling that $\log_{(2)}\omega(k)\ge \omega_{-1}(k)-O(1)$, we have, for $c$ sufficiently large:
\[\omega(g)\ge (\log_{(2)} T)^{c^{1/(n+2)}}.\]
If we choose $c$ and $T$ sufficiently large, we get
\begin{equation}\label{omegaglowerbound} \omega(g)\ge c\frac{(n+1)(n+2)}{2}\log_{(2)} T.\end{equation}
We also have
\[\omega(g)\le \sum_{j=1}^{n+1} \omega(d_{j})+\sum_{1\le i<j\le n+1} \omega(d_{i,j}).\]
So \eqref{omegaglowerbound} implies that $\omega(d_{j})\ge c \log_{(2)} T$ for some $j$ or $\omega(d_{i,j})\ge c \log_{(2)} T$ for some $i<j$. We also have $\tau(k)\ge 2^{\omega(k)}$ for all $k\ge 1$. We then have the claim \eqref{frakcsizeclaim}.

Let $\mathcal{N}_{j}$ be the set of $(\bv_{1},\bv_{2})\in\mathcal{N}$ such that $\tau(d_{j})\ge (\log T)^{c\log 2}$ and $\mathcal{N}_{i,j}$ the set of $(\bv_{1},\bv_{2})\in\mathcal{N}$ such that $\tau(d_{i,j})\ge (\log T)^{c\log 2}$. Then for sufficiently large $c$ and $T$:
\[\mathcal{N}=\bigcup_{j=1}^{n+1} \mathcal{N}_{j}\cup \bigcup_{1\le i<j\le n+1}\mathcal{N}_{i,j},\]
so it is enough to show that, for $c$ sufficiently large in terms of $K$,
\[\#\mathcal{N}_{j}\ll_{K}\frac{s_{1}^{n+2}s_{2}^{n}}{(\log T)^{K}}+s_{1}^{n+1}s_{2}^{n}\]
for all $1\le j\le n+1$ and
\[\#\mathcal{N}_{i,j}\ll_{K}\frac{s_{1}^{n+2}s_{2}^{n}}{(\log T)^{K}}+s_{1}^{n+1}s_{2}^{n}\]
for all $1\le i<j\le n+1$.

First, we will bound $\#\mathcal{N}_{j}$. From the definition of $d_{j}$, we see that it is the gcd of the $j$th coordinate of $\bv_{1}$ and the $j$th coordinate of $\bv_{2}$ (which we'll call $v_{1,j}$ and $v_{2,j}$ respectively). Hence, $\tau(d_{j})$ is bounded above by $\tau(v_{1,j})$. It follows from \eqref{powerdivisorsum} with $t=1$ that the number of $|v_{1,j}|\le \beta s_{1}$ such that $\tau(v_{1,j})\ge (\log T)^{c\log 2}$ is at most $O(s_{1}/(\log T)^{c\log 2-1}+1)$ (the $+1$ here is coming from $v_{1,j}=0$). There are $O(s_{1})$ choices for each of the remaining $v_{1,i}$ and for each choice of $\bv_{1}$, there are at most $O(s_{1}s_{2}^{n})$ choices of $\bv_{2}$, so we get the desired bound on $\#\mathcal{N}_{j}$.

Now consider $\#\mathcal{N}_{i,j}$. By the definition of $d_{i,j}$ we see that it is equal to $v_{1,i}v_{2,j}-v_{1,j}v_{2,i}$. The number of $(\bv_{1},\bv_{2})$ where $v_{1,i}=0$ is at most $O(s_{1}^{n+1}s_{2}^{n})$. We have a similar bound for the cases $v_{1,j}=0$, $v_{2,i}=0$, and $v_{2,j}=0$. When $v_{1,j}\ne 0$, the number of $\bv_{1},\bv_{2}$ where $v_{1,i}v_{2,j}-v_{1,j}v_{2,i}=0$ is also at most $O(s_{1}^{n+1}s_{2}^{n})$, because choosing $v_{1,i},v_{1,j},v_{2,i}$ uniquely determines $v_{2,j}$. So in all that follows we will suppose that $v_{1,i}$, $v_{1,j}$, $v_{2,i}$, $v_{2,j}$, and $v_{1,i}v_{2,j}-v_{1,j}v_{2,i}$ are all nonzero. As $n+1\ge 3$, we may choose some $1\le k\le n+1$ distinct from both $i$ and $j$.

The fact that $\bv_{2}$ lies in a region that depends on $\bv_{1}$ causes us some problems, so we aim to remove this dependence by projecting. Let $\pi_{k}:\bR^{n+1}\rightarrow \bR^{n}$ be the map which omits the $k$th coordinate. Let $\bw_{1}=\pi_{k}(\bv_{1})$ and $\bw_{2}=\pi_{k}(\bv_{2})$. Let $\theta\in [0,\pi/2]$ be the angle between $\bv_{1}$ and the hyperplane $H_{k}$ consisting of points whose $k$th coordinate is $0$. Then $\|\bw_{1}\|\le \beta s_{1}$ and $\|\bw_{2}\| \le \beta s_{2}$. For each $\bw_{1}$ there are at most $O(s_{1})$ corresponding values of $\bv_{1}$ and if $\theta>0$ then for each $\bw_{2}$, we have that the corresponding number of $\bv_{2}$ is at most $O(s_{1}/\sin \theta)$. To see this, recall that the distance of $\bv_{2}$ to $V_{1}=\bv_{1}^{\perp}$ is at most $\|\bv_{1}\|$, so it lies in the region between two planes which are parallel to $V_{1}$ and are a distance $2\beta s_{1}$ apart. By some trigonometry, any line not parallel to $V_{1}$ intersects this region in a line segment of length $2\beta s_{1}/\sin \phi$ where $\phi$ is the angle between the line and $V_{1}$. For each $\bw_{2}$, we see that $\pi^{-1}_{k}(\bw_{2})$ is a line parallel to the $k$-axis of angle $\theta$ away from $V_{1}$.

The number of $\bv_{1}$ such that $\sin \theta\le (\log T)^{-c/3}$ is $O(s_{1}^{n+1}(\log T)^{-c/3}+s_{1}^{n})$, and for each $\bv_{1}$ there are at most $O(s_{1}s_{2}^{n})$ possible values of $\bv_{2}$. Hence, we may assume that $\sin \theta> (\log T)^{-c/3}$ and so for each $(\bw_{1},\bw_{2})$, the number of corresponding $(\bv_{1},\bv_{2})$ is at most $O(s_{1}^{2}(\log T)^{c/3})$.

Now we bound the number of $(\bw_{1},\bw_{2})$. Using \eqref{divisordetsum} from Lemma \ref{divisorsumlemma} with $U_{1}=U_{2}=\beta s_{1}$, $V_{1}=V_{2}=\beta s_{2}$, and the fact that  $\log s_{1}s_{2}\le 2\log T$, it follows that when $c\log 2>4$, the number of $(\bw_{1},\bw_{2})$ is at most $O(s_{1}^{n}s_{2}^{n}/(\log T)^{c\log 2-4})$. The desired bound on $\#\mathcal{N}_{i,j}$ then follows for $c(\log 2 -1/3)>K+4$.
\end{proof}
\end{lemma}

As in \cite{bbs2023}, for $\bx,\by\in\bZ^{n+1}$ linearly independent, we define $\mathfrak{d}_{r}(\bx)$ to be the minimum determinant of a rank $r$ sublattice of $\bZ^{n+1}$ containing $\bx$ and $\mathfrak{d}_{r}(\bx,\by)$ the minimum over sublattices containing both $\bx$ and $\by$. Note that we trivially have
\begin{equation}\label{fdtrivbound} \fd_{2}(\bx,\by) \le \det\left(\bZ\bx\oplus\bZ\by\right) \le \|\bx\| \cdot \|\by\|.\end{equation}
For positive real numbers $X$, $Y$, we define the following quantities, which are analogous to $\ell_{r,n}(X;\Delta)$ and $\ell_{r,n}(X,Y;\Delta)$ from \cite{bbs2023}:
\[\ell'_{r,n}(X;\Delta)=\#\left\{ \bx\in \mathcal{P}^{n+1}:\begin{array}{l}0<\Vert \bx\Vert \le X\\ \mathfrak{d}_{r}(\bx)\le \Delta\end{array}\right\},\]
\[\ell'_{r,n}(X,Y;\Delta)=\#\left\{ (\bx,\by)\in \mathcal{P}^{n+1}\times \mathcal{P}^{n+1}:\begin{array}{l}\dim (\Span_{\bR}(\{\bx,\by\}))=2\\ \Vert \bx\Vert \le X, \Vert \by\Vert \le Y\\ \mathfrak{d}_{r}(\bx,\by)\le \Delta\end{array}\right\}.\]
 The following is an analogue of \cite[Lemma 3.20]{bbs2023} in the rank $2$ case.


\begin{lemma}\label{lXYbound}
Suppose $n\ge 2$, $\eta>0$ and $3\le Y^{\eta}\le X\le Y$. Then
\[\ell'_{2,n}(X,Y;\Delta)\ll_{\eta} \frac{X^{2}Y^{2}}{(\log X)^{n+1}(\log Y)^{n+1}}\Delta^{n-1}(\log_{(3)} Y)^{2n+2}+X^{2}Y^{2}\Delta^{n-3/2}.\]
\begin{proof}
Given a lattice $L$, let $\pi(L,X,Y)$ be the number of $\bx,\by\in\mathcal{P}^{n+1}$ with $\Vert \bx\Vert \le X$, $\Vert \by\Vert\le Y$ which are linearly independent such that $\bx,\by\in L$. Recall Minkowski's second theorem \eqref{minkowski2}. We have
\[\ell'_{2,n}(X,Y;\Delta)\ll \sum_{s_{2}\le Y}\sum_{\substack{s_{1}\le \min\{s_{2},X\}\\ s_{1}s_{2}\ll \Delta}}\sum_{L\in S_{2,n}(s_{1},s_{2})}\pi(L,X,Y),\]
where in the sums, $s_{1}$ and $s_{2}$ run through dyadic ranges.

Let $\eps>0$ to be chosen later. When $s_{2}\le X^{1-\eps}$, and $L\in S_{2,n}(s_{1},s_{2})$, we can apply Lemma \ref{latticeprimecount} to get
\[\pi(L,X,Y)\le \pi(L,X)\pi(L,Y)\ll_{\eps} \frac{X^{2}Y^{2}}{(\log X)^{n+1}(\log Y)^{n+1}s_{1}^{2}s_{2}^{2}}\fc(L)^{2}.\]
It follows from \cite[Lemma 3.19]{bbs2023} that $\#S_{2,n}(s_{1},s_{2})\ll s_{1}^{n+2}s_{2}^{n}$, so we get
\begin{align*}\sum_{L\in S_{2,n}(s_{1},s_{2})}\pi(L,X)\pi(L,Y)&\ll_{\eps} \frac{X^{2}Y^{2}}{(\log X)^{n+1}(\log Y)^{n+1}}s_{1}^{n}s_{2}^{n-2}(\log_{(3)} X)^{2n+2}\\&+\sum_{L\in \widehat{S}_{2,n,c,Y}(s_{1},s_{2})}\#\left(\mathcal{B}_{n+1}(X)\cap L\right)\#\left(\mathcal{B}_{n+1}(Y)\cap L\right).\end{align*}
Then by Lemma \ref{badlatticecount} with $K=2n+2$, $T=Y$ and \cite[Lemma 3.6]{bbs2023} we have
\begin{align*}\sum_{L\in \widehat{S}_{2,n,c,Y}(s_{1},s_{2})}\#\left(\mathcal{B}_{n+1}(X)\cap L\right)\#\left(\mathcal{B}_{n+1}(Y)\cap L\right)\ll& \frac{X^{2}Y^{2}s_{1}^{n}s_{2}^{n-2}}{(\log X)^{n+1}(\log Y)^{n+1}}\\&+X^{2}Y^{2}s_{1}^{n-1}s_{2}^{n-2}.\end{align*}
It follows that:
\begin{align}\nonumber\sum_{L\in S_{2,n}(s_{1},s_{2})}\pi(L,X)\pi(L,Y)\ll_{\eps}& \frac{X^{2}Y^{2}}{(\log X)^{n+1}(\log Y)^{n+1}}s_{1}^{n}s_{2}^{n-2}(\log_{(3)} X)^{2n+2}\\ \label{lXYboundeq3}&+X^{2}Y^{2}s_{1}^{n-1}s_{2}^{n-2}.\end{align}
We apply this when $s_{1}<X^{\varepsilon}$ and $s_{2}\le X^{1-\varepsilon}$.

We now deal with the case $s_{1}\ge X^{\eps}\ge Y^{\eta\eps}$. By Lemma \ref{basiscoeffbound}, each $L\in S_{2,n}(s_{1},s_{2})$ can be written as $L=\bZ \bu\oplus \bZ \bv$ with $\bu,\bv$ linearly independent, $\| \bu\| \asymp s_{1}$, $\| \bv\| \asymp s_{2}$, and the distance from $\bv$ to $\bu^{\perp}$ at most $O(s_{1})$. Let $\mathcal{V}$ be the set of $\bu,\bv$ satisfying these conditions. Then, by \eqref{latticecoordinatebound}, for all $R>0$, each $\bz\in L$ with $\Vert \bz\Vert\le R$ may be written as $\bz=a\bu+b\bv$ with $|a|\ll R/s_{1}$, $|b|\ll R/s_{2}$. Hence:
\begin{equation}\label{lXYboundeq2}\sum_{L\in S_{2,n}(s_{1},s_{2})}\pi(L,X,Y)\le \sum_{\substack{|a_{1}|\ll X/s_{1}\\ |a_{2}|\ll Y/s_{1}}}\sum_{\substack{|b_{1}|\ll X/s_{2}\\ |b_{2}|\ll Y/s_{2}}} \sum_{\substack{\bu,\bv\in \mathcal{V}\\ \dim \Span(\{\bx,\by\})=2}} \prod_{i,j} \mathbf{1}_{\mathcal{P}}(a_{i}u_{j}+b_{i}v_{j}),\end{equation}
(where $\bx=a_{1}\bu+b_{1}\bv$, $\by=a_{2}\bu+b_{2}\bv$ and $\mathbf{1}_{\mathcal{P}}$ is the indicator function of the set of primes). In the sum, we may assume that $a_{1}b_{2}-a_{2}b_{1}\ne 0$ because otherwise $\dim \Span(\{\bx,\by\})\le 1$ for all $\bu,\bv$. We split this sum up according to cases.

Let $\mathcal{X}$ be the set of all $(a_{1},a_{2},b_{1},b_{2})\in \bZ^{4}$ such that $|a_{1}|\ll X/s_{1}$, $|a_{2}|\ll Y/s_{1}$, $|b_{1}|\ll X/s_{2}$, $|b_{2}|\ll Y/s_{2}$, and $a_{1}b_{2}-a_{2}b_{1}\ne 0$. Let
\[\mathcal{X}_{1}=\left\{(a_{1},a_{2},b_{1},b_{2})\in \mathcal{X}:\: \begin{array}{l l}b_{1}\ne 0,\: b_{2}\ne 0,\\ \max\{\tau(b_{1}),\tau(b_{2}),\tau(a_{1}b_{2}-a_{2}b_{1})\}\ll (\log Y)^{c}\end{array}\right\}\]
where $c$ is a sufficiently large constant to be chosen,
\[\mathcal{X}_{2}=\left\{(a_{1},a_{2},b_{1},b_{2})\in \mathcal{X}:\: \begin{array}{l l}b_{1}=0,\: \tau(b_{2})\ll (\log Y)^{c}\\ \text{or}\: b_{2}=0,\: \tau(b_{1})\ll (\log Y)^{c}\end{array}\right\},\]
\[\mathcal{X}_{3}=\left\{(a_{1},a_{2},b_{1},b_{2})\in \mathcal{X}:\: \begin{array}{l l}b_{1}=0,\: \tau(b_{2})\gg (\log Y)^{c}\\ \text{or}\: b_{2}=0,\: \tau(b_{1})\gg (\log Y)^{c}\end{array}\right\},\]
and
\[\mathcal{X}_{4}=\left\{(a_{1},a_{2},b_{1},b_{2})\in \mathcal{X}:\: \begin{array}{l l}b_{1}\ne 0,\: b_{2}\ne 0,\\ \max\{\tau(b_{1}),\tau(b_{2}),\tau(a_{1}b_{2}-a_{2}b_{1})\}\gg (\log Y)^{c} \end{array}\right\}.\]
Clearly, $\mathcal{X}=\bigcup_{i=1}^{4}\mathcal{X}_{i}$, so if we further define
\[\mathcal{Q}(a_{1},a_{2},b_{1},b_{2})=\sum_{\substack{\bu,\bv\in \mathcal{V}\\ \dim \Span(\{\bx,\by\})=2}} \prod_{i,j} \mathbf{1}_{\mathcal{P}}(a_{i}u_{j}+b_{i}v_{j}),\]
and, for $1\le i\le 4$,
\[\mathcal{Z}_{i}=\sum_{(a_{1},a_{2},b_{1},b_{2})\in \mathcal{X}_{i}} \mathcal{Q}(a_{1},a_{2},b_{1},b_{2}),\]
then \eqref{lXYboundeq2} gives
\begin{equation}\label{mathcalZsumbound}\sum_{L\in S_{2,n}(s_{1},s_{2})}\pi(L,X,Y)\le \sum_{i=1}^{4} \mathcal{Z}_{i}.\end{equation}

Let us first bound $\mathcal{Z}_{1}$. Let $(a_{1},a_{2},b_{1},b_{2})\in \mathcal{X}_{1}$, and consider $\mathcal{Q}(a_{1},a_{2},b_{1},b_{2})$. We fix some $\bu$ and sum over $\bv$. The set of $\bv$ which are within a distance $O(s_{1})$ of $\bu^{\perp}$ can be partitioned into at most $O((s_{2}/s_{1})^{n})$ balls of radius $O(s_{1})$. The contribution to the sum from each of these balls is of the form $\pi(Cs_{1},\dots,Cs_{1},\tilde{\bv},A,\tilde{\bu})$ for some $C>0$, where $\pi$ is as in Lemma \ref{matrixprimecount}, $\tilde{\bu}=(a_{1}u_{0},\dots,a_{1}u_{n},a_{2}u_{0},\dots,a_{2}u_{n})$, $\tilde{\bv}$ is some integer vector, and $A$ is the block matrix
\[\begin{pmatrix} b_{1}I_{n+1}\\ b_{2}I_{n+1}\end{pmatrix}.\]
Assume that $u_{j}\ne 0$ for all $j$. Then $\fc(A,\tilde{\bu})< \infty$, so by Lemma \ref{matrixprimecount}, the contribution to $\mathcal{Q}$ from each ball is at most
\[O\left( \frac{s_{1}^{n+1}}{(\log s_{1})^{2n+2}}\fc(A,\tilde{\bu})\right).\]
Since $\tau(b_{i})\ll (\log Y)^{c}$ for $i\in\{1,2\}$, we have $\omega(d_{j})\ll \log_{(2)} Y$ for all $1\le j\le 2n+2$. For all $1\le i<j\le 2n+2$, $d_{i,j}$ will divide one of $b_{1}^{2}$, $b_{2}^{2}$, $b_{1}b_{2}$, or $(a_{1}b_{2}-a_{2}b_{1})u_{k}$ for some $k$. Recall that $\tau(a_{1}b_{2}-a_{2}b_{1})\ll (\log Y)^{c}$, so if we also assume that $\tau(u_{j})\ll (\log Y)^{c}$ for all $1\le j\le n+1$ then we get $\omega(d_{i,j})\ll \log_{(2)} Y$ for all $1\le i<j\le 2n+2$. By a similar argument to the proof of \eqref{frakcsizeclaim}, this implies that $\fc(A,\tilde{\bu})\ll (\log_{(3)} Y)^{2n+2}$ (in \eqref{frakcsizeclaim}, we assumed that $T$ is sufficiently large, and $Y$ could be small, but this is not a problem because in the case where $X\le Y$ is small, $\ell'_{2,n}(X,Y)\ll 1$).

Now by the fact that there are at most $O((s_{2}/s_{1})^{n})$ of these balls, and at most $O(s_{1}^{n+1})$ choices of $\bu$, we get that
\begin{align} \label{lXYboundeq4}\nonumber\sum_{\substack{\bu,\bv\in \mathcal{V}\\ \dim \Span(\{\bx,\by\})=2\\ \tau(u_{j})\ll (\log Y)^{c}\text{for all $j$}}} \prod_{i,j} \mathbf{1}_{\mathcal{P}}(a_{i}u_{j}+b_{i}v_{j})&\ll \frac{s_{1}^{n+2}s_{2}^{n}}{(\log s_{1})^{2n+2}}(\log_{(3)} Y)^{2n+2}\\ &\ll_{\eps,\eta} \frac{s_{1}^{n+2}s_{2}^{n}}{(\log X)^{n+1}(\log Y)^{n+1}}(\log_{(3)} Y)^{2n+2},\end{align}
where we used the assumption $s_{1}\ge X^{\eps}$ for the second inequality.

Now we deal with those $\bu$ where $\tau(u_{j})\gg (\log Y)^{c}$ for some $j$ (recall that $\tau(0)=\infty$, so this includes $u_{j}=0$). The number of such $\bu$ is at most $O(s_{1}^{n}+s_{1}^{n+1}/(\log Y)^{c-1})$ by an application of \eqref{powerdivisorsum} with $t=1$, and for each such $\bu$ there are at most $O(s_{1}s_{2}^{n})$ possible values of $\bv$. Taking $c$ sufficiently large and combining this with \eqref{lXYboundeq4} gives:
\begin{equation}\label{lXYboundeq1}\mathcal{Q}(a_{1},a_{2},b_{1},b_{2})\ll_{\eps,\eta} \frac{s_{1}^{n+2}s_{2}^{n}(\log_{(3)} Y)^{2n+2}}{(\log X)^{n+1}(\log Y)^{n+1}}+s_{1}^{n+1}s_{2}^{n},\end{equation}
for all $(a_{1},a_{2},b_{1},b_{2})\in \mathcal{X}_{1}$. The assumption that $s_{1}\ge X^{\eps}$ then means that we can ignore the $s_{1}^{n+1}s_{2}^{n}$ term. Hence
\begin{equation}\label{mathcalZ1bound}
\mathcal{Z}_{1} \ll_{\eps,\eta} \frac{X^{2}Y^{2}s_{1}^{n}s_{2}^{n-2}(\log_{(3)} Y)^{2n+2}}{(\log X)^{n+1}(\log Y)^{n+1}}.
\end{equation}

We now bound $\mathcal{Z}_{2}$. Let $(a_{1},a_{2},b_{1},b_{2})\in \mathcal{X}_{2}$ and suppose that $b_{1}=0$ and $\tau(b_{2})\ll (\log Y)^{c}$. We have that $a_{1}u_{j}$ is prime if and only if $a_{1}=1$ and $u_{j}$ is prime or $a_{1}$ is prime and $u_{j}=1$. In any case, there are at most $O(s_{1}/\log s_{1})$ possible choices for each $u_{j}$, which overall saves a factor of $(\log s_{1})^{n+1} \asymp_{\eps} (\log X)^{n+1}$ over the trivial bound. For each such $\bu$, we divide the set of $\bv$ up into balls of radius $O(s_{1})$ as before and apply Lemma \ref{matrixprimecount} on each ball, but this time with $A=b_{2}I_{n+1}$ and $\bb=a_{2}\bu$. This saves an additional factor of $(\log Y)^{n+1}/(\log_{(3)} Y)^{n+1}$ and we get \eqref{lXYboundeq1} again. The case $b_{2}=0$, $\tau(b_{1})\ll (\log Y)^{c}$ is similar. Hence
\begin{equation}\label{mathcalZ2bound}
\mathcal{Z}_{2} \ll_{\eps,\eta} \frac{X^{2}Y^{2}s_{1}^{n}s_{2}^{n-2}(\log_{(3)} Y)^{2n+2}}{(\log X)^{n+1}(\log Y)^{n+1}}.
\end{equation}

Recall that, for $(a_{1},a_{2},b_{1},b_{2})\in \mathcal{X}$, we have $a_{1}b_{2}-a_{2}b_{1}\ne 0$ so in particular, $b_{1}$ and $b_{2}$ can't both be $0$. Let $(a_{1},a_{2},b_{1},b_{2})\in \mathcal{X}_{3}$ and suppose $b_{1}=0$ and $\tau(b_{2})\gg (\log Y)^{c}$. We must have $b_{2} \ne 0$, so the number of $b_{2}\ll Y/s_{2}$ such that $b_{2}\ne 0$ and $\tau(b_{2})\gg (\log Y)^{c}$ is at most $O(Y/(s_{2}(\log Y)^{c-1}))$, so by choosing $c$ large enough and recalling that there are $O(XY/s_{1}^{2})$ choices for $a_{1}$ and $a_{2}$, we see that the number of such $(a_{1},a_{2},b_{1},b_{2})$ is at most
\[O\left(\frac{XY^{2}}{s_{1}^{2}s_{2}(\log X)^{n+1}(\log Y)^{n+1}}\right).\]
By a similar argument, the number of $(a_{1},a_{2},b_{1},b_{2})\in \mathcal{X}_{3}$ with $b_{2}=0$ and $\tau(b_{1})\gg (\log Y)^{c}$, satisfies the same bound but with $X^{2}Y$ in place of $XY^{2}$, which is smaller because $X\le Y$. Hence, using the trivial bound $\mathcal{Q}(a_{1},a_{2},b_{1},b_{2})\le \#\mathcal{V}\ll s_{1}^{n+2}s_{2}^{n}$, we see that
\begin{equation}\label{mathcalZ3bound}\mathcal{Z}_{3}\ll \frac{XY^{2}s_{1}^{n}s_{2}^{n-1}}{(\log X)^{n+1}(\log Y)^{n+1}}.\end{equation}

It just remains to bound $\mathcal{Z}_{4}$. Let $(a_{1},a_{2},b_{1},b_{2})\in \mathcal{X}_{4}$. Then $b_{1}$, $b_{2}$ and $a_{1}b_{2}-a_{2}b_{1}$ are all nonzero, but one of $\tau(b_{1})\gg (\log Y)^{c}$, $\tau(b_{2})\gg (\log Y)^{c}$ or $\tau(a_{1}b_{2}-a_{2}b_{1})\gg (\log Y)^{c}$ holds. In the first two cases, it is easy to see that the number of such $a_{1}$, $b_{1}$, $a_{2}$, $b_{2}$ is at most
\begin{equation}\label{lXYboundeq5}O\left(\frac{X^{2}Y^{2}}{s_{1}^{2}s_{2}^{2}(\log X)^{n+1}(\log Y)^{n+1}}\right),\end{equation}
when $c$ is sufficiently large. When $\tau(a_{1}b_{2}-a_{2}b_{1})\gg (\log Y)^{c}$ and $a_{1},a_{2}\ne 0$ then we also get \eqref{lXYboundeq5} by using \eqref{divisordetsum} with $U_{i}=X/s_{i}$, $V_{i}=Y/s_{i}$, and $Z=XY/(s_{1}s_{2})$. If $a_{1}=0$ then $a_{2}\ne 0$ and $\max\{\tau(a_{2}),\tau(b_{1})\}\gg (\log Y)^{c/2}$ and we also deduce the bound \eqref{lXYboundeq5}. The case $a_{2}=0$ is similar. Hence, by the trivial bound $\mathcal{Q}(a_{1},a_{2},b_{1},b_{2})\ll s_{1}^{n+2}s_{2}^{n}$, we conclude that
\begin{equation}\label{mathcalZ4bound} \mathcal{Z}_{4}\ll \frac{X^{2}Y^{2}s_{1}^{n}s_{2}^{n-2}}{(\log X)^{n+1}(\log Y)^{n+1}}.
\end{equation}

Inserting \eqref{mathcalZ1bound}, \eqref{mathcalZ2bound}, \eqref{mathcalZ3bound}, and \eqref{mathcalZ4bound} into \eqref{mathcalZsumbound} we find that, for $s_{1}\ge X^{\eps}$,
\begin{align}\label{lXYboundeq6}\sum_{L\in S_{2,n}(s_{1},s_{2})}\pi(L,X,Y)\ll_{\eps,\eta}\left(X^{2}Y^{2}s_{1}^{n}s_{2}^{n-2}+XY^{2}s_{1}^{n}s_{2}^{n-1}\right)\frac{(\log_{(3)} Y)^{2n+2}}{(\log X)^{n+1}(\log Y)^{n+1}}.\end{align}

\bigskip Now we bound $\ell'_{2,n}(X,Y;\Delta)$. Using \eqref{lXYboundeq6} for the terms with $s_{1}\ge X^{\eps}$ and using \eqref{lXYboundeq3} for the terms with $s_{1}<X^{\eps}$ and $s_{2}\le X^{1-\eps}$, we have
\begin{align*}\ell'_{2,n}(&X,Y;\Delta)\ll_{\eps,\eta}\\ &\sum_{s_{2}\le Y}\sum_{\substack{s_{1}\le \min\{s_{2},X\}\\s_{1}s_{2}\ll \Delta}} \left(X^{2}Y^{2}s_{1}^{n}s_{2}^{n-2}+XY^{2}s_{1}^{n}s_{2}^{n-1}\right)\frac{(\log_{(3)} Y)^{2n+2}}{(\log X)^{n+1}(\log Y)^{n+1}}\\&+\sum_{\substack{s_{1}s_{2}\ll \Delta\\s_{1}\ll \min\{s_{2},X^{\eps},\Delta^{1/2}\}}}X^{2}Y^{2}s_{1}^{n-1}s_{2}^{n-2}\\&+\sum_{X^{1-\eps}<s_{2}\le Y}\sum_{\substack{s_{1}<X^{\eps}\\s_{1}s_{2}\ll \Delta}}\sum_{L\in S_{2,n}(s_{1},s_{2})}\pi(L,X)\pi(L,Y).\end{align*}
By the same argument as the last part of the proof of \cite[Lemma 3.20]{bbs2023}, the first term is bounded by
\[ \frac{X^{2}Y^{2}\Delta^{n-1}(\log_{(3)} Y)^{2n+2}}{(\log X)^{n+1}(\log Y)^{n+1}}.\]
For the second term we have
\begin{align*}\sum_{\substack{s_{1}s_{2}\ll \Delta\\s_{1}\ll \min\{s_{2},X^{\eps},\Delta^{1/2}\}}}s_{1}^{n-1}s_{2}^{n-2} &\ll \sum_{\substack{s_{2}\le \Delta^{1/2}\\s_{1}\le s_{2}}} s_{1}^{n-1}s_{2}^{n-2}+\sum_{\substack{s_{1} \ll \Delta/s_{2}\\ s_{2}>\Delta^{1/2}}} s_{1}^{n-1}s_{2}^{n-2}\\ &\ll \sum_{s_{2} \le \Delta^{1/2}} s_{2}^{2n-3} + \Delta^{n-1}\sum_{s_{2}> \Delta^{1/2}} \frac{1}{s_{2}}\\ &\ll \Delta^{n-3/2}.\end{align*}
For the third term, by using \cite[Lemma 3.6]{bbs2023}, we have the bound
\[\pi(L,X)\pi(L,Y)\le \#(L\cap \mathcal{B}_{n+1}(X))\#(L\cap \mathcal{B}_{n+1}(Y))\ll \left( \frac{X^{2}}{s_{1}s_{2}}+\frac{X}{s_{1}}\right)\frac{Y^{2}}{s_{1}s_{2}}.\]
Combining this with the fact that $s_{1}/s_{2}< X^{2\eps-1}$ and \cite[Lemma 3.19]{bbs2023}, we get
\begin{align*}\sum_{X^{1-\eps}<s_{2}\le Y}\sum_{\substack{s_{1}\le X^{\eps}\\s_{1}s_{2}\ll \Delta}}&\sum_{L\in S_{2,n}(s_{1},s_{2})}\pi(L,X)\pi(L,Y)\\ &\ll \sum_{X^{1-\eps}<s_{2}\le Y}\sum_{\substack{s_{1}\le X^{\eps}\\s_{1}s_{2}\ll \Delta}} \left( X^{2}Y^{2}s_{1}^{n}s_{2}^{n-2}+XY^{2}s_{1}^{n}s_{2}^{n-1}\right)\\ &\ll X^{2\eps +1}Y^{2}\sum_{s_{1}s_{2}\ll \Delta}s_{1}^{n-1}s_{2}^{n-1}+XY^{2}\sum_{s_{1}\le X^{\eps}}\sum_{s_{2}\ll \Delta/s_{1}}s_{1}^{n}s_{2}^{n-1}\\ &\ll X^{2\eps+1}Y^{2}\Delta^{n-1}(\log \Delta)^{2}+X^{1+\eps}Y^{2}\Delta^{n-1}.\end{align*}
Recalling \eqref{fdtrivbound} we see that without loss of generality, $\Delta\le XY$ which implies that $\log \Delta\ll_{\eta} \log X$. Taking $\eps<1/2$ gives the desired bound.
\end{proof}
\end{lemma}

Now define
\[\Omega'_{d,n}(B)=\left\{ (\bx,\by)\in \mathcal{P}^{n+1}\times \mathcal{P}^{n+1} : \begin{array}{l} \Vert \bx\Vert, \Vert \by\Vert\le B^{1/(n+1-d)}\\\bx\ne \by\\\bx,\by\notin \Span\{(1,\dots,1)\}\end{array}\right\},\]
\[E'_{d,n}(B)=(\log B)^{2n+2}\sum_{(\bx,\by)\in \Omega'_{d,n}(B)}\frac{1}{\det\left( \Lambda_{\nu_{d,n}(\bx)}\cap \Lambda_{\nu_{d,n}(\by)}\right)}.\]
It can be easily shown that if $(\bx,\by)\in\Omega'_{d,n}(B)$ then $\bx$ and $\by$ must be linearly independent. This $E'_{d,n}$ is analogous to the function $E_{d,n}$ from \cite{bbs2023} and we wish to have bounds that replace those in \cite[Lemma 4.5]{bbs2023}.
\begin{lemma}\label{Esandwich}
Let $d\ge 2$ and $n\ge d$ with $(d,n)\ne (2,2)$. Then
\[B^{2}\ll E'_{d,n}(B)\ll B^{2}(\log_{(3)} B)^{2n+2}.\]
\begin{proof}
We follow the same proof as in \cite[Lemma 4.5]{bbs2023} but using Lemma \ref{lXYbound} instead of \cite[Lemma 3.21]{bbs2023}. By the upper bound in \cite[Lemma 4.4]{bbs2023},
\[E'_{d,n}(B)\gg (\log B)^{2n+2}\sum_{(\bx,\by)\in \Omega'_{d,n}(B)}\frac{1}{\| \bx\|^{d}\| \by\|^{d}},\]
which gives the lower bound by considering that, by the prime number theorem, the number of $(\bx,\by)\in \Omega'_{d,n}(B)$ with $\|\bx\|,\|\by\|> \frac{1}{2}B^{1/(n+1-d)}$ has size $\gg B^{2(n+1)/(n+1-d)}/(\log B)^{2n+2}$. By using \cite[Lemma 4.4]{bbs2023} and breaking up $\| \bx\|$, $\| \by\|$ and $\mathfrak{d}_{2}(\bx,\by)$ into dyadic intervals, we get
\[E'_{d,n}(B)\ll (\log B)^{2n+2}\sum_{X\le Y\ll B^{1/(n+1-d)}}\sum_{\Delta_{2}\ll XY}\frac{1}{\Delta_{2}(XY)^{d-1}}\ell'_{2,n}(X,Y;\Delta_{2}).\]
Here, and elsewhere unless stated otherwise, sums over $X$, $Y$ and $\Delta_{2}$ are assumed to be over dyadic ranges and the upper bound for $\Delta_{2}$ comes from \eqref{fdtrivbound}.

By using Lemma \ref{lXYbound} when $X\ge Y^{\eta}$ and \cite[Lemma 3.21]{bbs2023} otherwise we have,
\begin{align*}\frac{E'_{d,n}(B)}{(\log B)^{2n+2}}\ll_{\eta}& \sum_{Y^{\eta}\le X\le Y\ll B^{1/(n+1-d)}}\frac{(\log_{(3)} Y)^{2n+2}}{(XY)^{d-3}(\log X)^{n+1}(\log Y)^{n+1}}\sum_{\Delta_{2}\ll XY} \Delta_{2}^{n-2}\\&+\sum_{\substack{X<Y^{\eta}\\ Y\ll B^{1/(n+1-d)}}}\frac{1}{(XY)^{d-3}}\sum_{\Delta_{2}\ll XY}\Delta_{2}^{n-2}\\&+\sum_{X\le Y\ll B^{1/(n+1-d)}}\frac{1}{(XY)^{d-3}}\sum_{\Delta_{2}\ll XY}\Delta_{2}^{n-5/2}\\ \ll& \frac{B^{2}(\log_{(3)} B)^{2n+2}}{(\log B)^{2n+2}}+B^{1+\eta}+B^{2-1/(n+1-d)}.\end{align*}
In the last line we have used the fact that $n\ge 3$, and to deal with the factors of $(\log X)^{n+1}$ and $(\log Y)^{n+1}$ in the first sum, we restrict the sum to $Y \ge B^{\eps}$, say, for some fixed $\eps>0$, so that $\log X \asymp \log Y \asymp \log B$, and then note that the remaining parts of the sum are negligible if $\eps$ is chosen small enough. Now take for example $\eta=1/2$.
\end{proof}
\end{lemma}

Recall the definition \eqref{Wdef} of $W$ and $w$.
\begin{lemma}\label{Wbound}
We have
\[W\ll (\log B)^{4/\log_{(3)} B}.\]
\begin{proof}
By the prime number theorem:
\[\log W=\sum_{p\le w} \log p \left(\left\lceil \frac{\log w}{\log p}\right\rceil+1\right)\le \pi(w)\log w+2\sum_{p\le w}\log p\le 4w,\]
when $B$ (and hence $w$) is sufficiently large. The result follows by taking $\exp$ of both sides.
\end{proof}
\end{lemma}
Now define
\[F'_{d,n}(B)=(\log B)^{2n+2}\sum_{(\bx,\by)\in\Omega'_{d,n}(B)}\frac{\mathcal{E}_{\bx,\by}(B)}{\det \left(\Lambda_{\nu_{d,n}(\bx)}\cap\Lambda_{\nu_{d,n}(\by)}\right)},\]
where $\mathcal{E}_{\bx,\by}$ is defined as in \cite[(4.10)]{bbs2023}, by
\[\mathcal{E}_{\bx,\by}(B)=\min\left\{1,\frac{\Delta(\bx,\by)^{2}}{\alpha^{2}}\right\}+\mathbf{1}_{\mathcal{G}(\bx,\by)\,\nmid\, \frac{W}{\rad(W)}}.\]
Here,
\[\Delta(\bx,\by)=\frac{\Vert \nu_{d,n}(\bx)\Vert\cdot\Vert \nu_{d,n}(\by)\Vert}{\det(\bZ\nu_{d,n}(\bx)\oplus\bZ\nu_{d,n}(\by))},\]
$\mathcal{G}(\bc_{1},\dots,\bc_{k})$ is the greatest common divisor of the $k\times k$ minors of the matrix with columns $\bc_{1},\dots,\bc_{k}$, and for $n\in \bZ^{+}$, we define $\rad(n) = \prod_{p \mid n} p$. We also recall that $\alpha=\log B$. The following is an analogue of \cite[Lemma 4.6]{bbs2023}.
\begin{lemma}\label{Fbound}
Let $d\ge 2$ and $n\ge d$ with $(d,n)\ne (2,2)$. Then for all $\eps>0$,
\[F'_{d,n}(B)\ll_{\eps} \frac{B^{2}}{(\log_{(2)} B)^{n-2-\eps}}.\]
\begin{proof}
Let
\[F_{d,n}^{(1)}(B)=\sum_{(\bx,\by)\in\Omega'_{d,n}(B)}\frac{1}{\det\left(\Lambda_{\nu_{d,n}(\bx)}\cap\Lambda_{\nu_{d,n}(\by)}\right)}\min\left\{1,\frac{\Delta(\bx,\by)^{2}}{\alpha^{2}}\right\},\]
and
\[F_{d,n}^{(2)}(B)=\sum_{(\bx,\by)\in\Omega'_{d,n}(B)}\frac{1}{\det\left(\Lambda_{\nu_{d,n}(\bx)}\cap\Lambda_{\nu_{d,n}(\by)}\right)}\mathbf{1}_{\mathcal{G}(\bx,\by)\,\nmid\, \frac{W}{\rad(W)}},\]
so that
\[F'_{d,n}(B)=(\log B)^{2n+2}\left(F_{d,n}^{(1)}(B)+F_{d,n}^{(2)}(B)\right).\]
It follows from \cite[Lemma 4.4 and equation (4.4)]{bbs2023} that
\[\Delta(\bx,\by)\ll \frac{\Vert \bx\Vert\cdot\Vert \by\Vert}{\fd_{2}(\bx,\by)},\]
and
\begin{multline*}F_{d,n}^{(1)}(B)\ll\\ \sum_{X\le Y\ll B^{1/(n+1-d)}}\sum_{\Delta_{2}\ll XY}\frac{1}{\Delta_{2}(XY)^{d-1}}\left(\min\left\{1,\frac{(XY)^{2}}{\Delta_{2}^{2}\alpha^{2}}\right\}\right)\ell'_{2,n}(X,Y;\Delta_{2}).\end{multline*}
We let $\eta>0$ to be chosen later. Now we use the fact that
\[\min\left\{1,\frac{(XY)^{2}}{\Delta_{2}^{2}\alpha^{2}}\right\}\le \frac{(XY)^{1/2}}{\Delta_{2}^{1/2}\alpha^{1/2}},\]
together with Lemma \ref{lXYbound} when $X\ge Y^{\eta}$ and \cite[Lemma 3.21]{bbs2023} when $X<Y^{\eta}$, to get:
\begin{align*}F_{d,n}^{(1)}(B)&\ll_{\eta} \frac{1}{\alpha^{1/2}}\sum_{Y^{\eta}\le X\le Y\ll B^{1/(n+1-d)}}\frac{(\log_{(3)} Y)^{2n+2}}{(\log X)^{n+1}(\log Y)^{n+1}(XY)^{d-7/2}}\sum_{\Delta_{2}\ll XY}\Delta_{2}^{n-5/2}\\
&+\frac{1}{\alpha^{1/2}}\sum_{X\le Y\ll B^{1/(n+1-d)}}\frac{1}{(XY)^{d-7/2}}\sum_{\Delta_{2}\ll XY}\Delta_{2}^{n-3}\\
&+\frac{1}{\alpha^{1/2}}\sum_{\substack{X<Y^{\eta}\\Y\ll B^{1/(n+1-d)}}}\frac{1}{(XY)^{d-7/2}}\sum_{\Delta_{2}\ll XY}\Delta_{2}^{n-5/2}\\
&\ll\frac{B^{2}(\log_{(3)} B)^{2n+2}}{(\log B)^{2n+5/2}}+\frac{1}{\alpha^{1/2}}\sum_{X\le Y\ll B^{1/(n+1-d)}}(XY)^{n+1/2-d}\log XY\\
&+\frac{1}{\alpha^{1/2}}\sum_{\substack{X<Y^{\eta}\\Y\ll B^{1/(n+1-d)}}} (XY)^{n+1-d}\\
&\ll\frac{B^{2}(\log_{(3)} B)^{2n+2}}{(\log B)^{2n+5/2}}+B^{2-1/(n+1-d)}(\log B)^{1/2}+\frac{B^{1+\eta}}{(\log B)^{1/2}}.
\end{align*}
We can then take $\eta=1/2$ for example.

Now we will bound $F_{d,n}^{(2)}$. In the proof of \cite[Lemma 4.6]{bbs2023}, it is shown that if $\mathcal{G}(\bx,\by)\nmid \frac{W}{\rad(W)}$ then
\[\fd_{2}(\bx,\by)\le \frac{\Vert \bx\Vert \cdot \Vert \by\Vert}{w},\]
and this still holds in our case, where we have defined $w$ to be smaller than in \cite{bbs2023}.
Hence, by \cite[Lemma 4.4]{bbs2023},
\[F_{d,n}^{(2)}(B)\ll \sum_{X\le Y\ll B^{1/(n+1-d)}}\sum_{\Delta_{2}\ll XY/w}\frac{1}{\Delta_{2}(XY)^{d-1}}\ell'_{2,n}(X,Y;\Delta_{2}),\]
and so, for $\eta>0$,
\begin{align*}F_{d,n}^{(2)}(B)\ll_{\eta}& \sum_{Y^{\eta}\le X\le Y\ll B^{1/(n+1-d)}}\frac{(\log_{(3)} Y)^{2n+2}}{(\log X)^{n+1}(\log Y)^{n+1}(XY)^{d-3}}\sum_{\Delta_{2}\ll XY/w}\Delta_{2}^{n-2}\\
&+\sum_{X\le Y\ll B^{1/(n+1-d)}}\frac{1}{(XY)^{d-3}}\sum_{\Delta_{2}\ll XY/w}\Delta_{2}^{n-5/2}\\
&+\sum_{\substack{X<Y^{\eta}\\Y\ll B^{1/(n+1-d)}}}\frac{1}{(XY)^{d-3}}\sum_{\Delta_{2}\ll XY/w}\Delta_{2}^{n-2}\\
\ll &\frac{B^{2}(\log_{(3)} B)^{2n+2}}{w^{n-2}(\log B)^{2n+2}}+\frac{B^{2-1/(n+1-d)}}{w^{n-5/2}}+\frac{B^{1+\eta}}{w^{n-2}}.
\end{align*}
Now choose $\eta=1/2$ and note that $w\gg (\log_{(2)} B)^{1-\eps}$ to complete the proof.
\end{proof}
\end{lemma}
Now we define
\[\hat{D}_{d,n}(A,B)=\sum_{V\in\bV_{d,n}(A)}N'_{V}(B)^{2}-(\log B)^{n+1}\sum_{V\in\bV_{d,n}(A)}N'_{V}(B),\]
\[\hat{D}_{d,n}^{\mathrm{mix}}(A,B)=\sum_{V\in\bV_{d,n}(A)}N'_{V}(B)N^{\mathrm{ploc}}_{V}(B)-(\log B)^{n+1}\sum_{V\in\bV_{d,n}(A)}\hat{\Delta}_{V}^{\mathrm{mix}}(B),\]
and
\[\hat{D}_{d,n}^{\mathrm{loc}}(A,B)=\sum_{V\in\bV_{d,n}(A)}N^{\mathrm{ploc}}_{V}(B)^{2}-(\log B)^{n+1}\sum_{V\in\bV_{d,n}(A)}\hat{\Delta}_{V}^{\mathrm{loc}}(B),\]
where:
\begin{equation}\label{Deltamixdef}\hat{\Delta}_{V}^{\mathrm{mix}}(B)=(\log B)^{n+1}\frac{\alpha W}{\Vert \ba_{V}\Vert}\sum_{\substack{\bx\in\Xi'_{d,n}(B)\\ \ba_{V}\in\Lambda_{\nu_{d,n}(\bx)}}}\frac{1}{\Vert\nu_{d,n}(\bx)\Vert}\end{equation}
and
\begin{equation}\label{Deltalocdef}\hat{\Delta}_{V}^{\mathrm{loc}}(B)=(\log B)^{n+1}\frac{\alpha^{2} W^{2}}{\Vert \ba_{V}\Vert^{2}}\sum_{\substack{\bx\in\Xi'_{d,n}(B)\\ \ba_{V}\in\Lambda_{\nu_{d,n}(\bx)}^{(W)}\cap\mathcal{C}_{\nu_{d,n}(\bx)}^{(\alpha)}}} \frac{1}{\Vert\nu_{d,n}(\bx)\Vert^{2}}.\end{equation}
Also, let
\[\iota'_{d,n}=\frac{V_{N_{d,n}-2}}{2\zeta(N_{d,n}-2)},\]
where $V_{N}$ is the volume of the unit ball $\mathcal{B}_{N}(1)$ for $N\ge 1$.
In the proofs of the following lemmas, we also need the following definitions from \cite{bbs2023}. For $\bx,\by\in \bZ^{n+1}$, $k \in \{0,1\}$, $\mathcal{R} \subset \bR^{N_{d,n}}$, $\Lambda$ a lattice and $W, \alpha$ as defined in \eqref{Wdef}, define
\begin{align}\label{lotsofdefs}\nonumber \Gamma_{\bx,\by}=\Lambda_{\nu_{d,n}(\bx)}\cap\Lambda_{\nu_{d,n}(\by)}, &\qquad
\Gamma_{\bx,\by}^{\mathrm{mix}}(W)=\Lambda_{\nu_{d,n}(\bx)}\cap\Lambda_{\nu_{d,n}(\by)}^{(W)}, \\
\mathcal{T}_{\by}^{\mathrm{mix}}=\mathcal{B}_{N_{d,n}}(A)\cap\mathcal{C}_{\nu_{d,n}(\by)}^{(\alpha)}, &\qquad
\mathcal{S}_{k}^{\ast}(\Lambda;\mathcal{R})=\sum_{\ba\in\Lambda\cap\bZ_{\mathrm{prim}}^{N_{d,n}}\cap\mathcal{R}} \frac{1}{\Vert\ba\Vert^{k}}.\end{align}
In the following four lemmas, we approximate each of $\hat{D}$, $\hat{D}^{\mathrm{mix}}$ and $\hat{D}^{\mathrm{loc}}$. Each will correspond to an analogous lemma from \cite[Section 4]{bbs2023}. The first is an analogue of \cite[Lemma 4.8]{bbs2023}.
\begin{lemma}\label{Dapprox}
Let $d\ge 2$ and $n\ge d$ with $(d,n)\notin \{(2,2),(3,3),(4,4)\}$. Then for all $m\ge 1$ and $A\ge B/(\log B)^{m}$, we have,
\[\hat{D}_{d,n}(A,B)=\iota'_{d,n}A^{N_{d,n}-2}E'_{d,n}(B)\left(1+O_{m}\left(\frac{(\log A)^{2n+6+m/2}}{A^{1/2}}\right)\right).\]
\begin{proof}
We have
\begin{align*}\hat{D}_{d,n}(A,B)&=\sum_{\substack{\bx,\by\in \Xi'_{d,n}(B)\\ \bx\ne \by}} (\log B)^{2n+2}\#\{V\in \bV_{d,n}(A):\bx,\by\in V(\mathcal{P})\}\\&=\frac{(\log B)^{2n+2}}{2}\sum_{(\bx,\by)\in\Omega'_{d,n}(B)}\mathcal{S}_{0}^{\ast}(\Gamma_{\bx,\by};\mathcal{B}_{N_{d,n}}(A)),\end{align*}
where the $1/2$ factor comes from the fact that there are two $\ba_{V}$ for each $V$.

Let
\[\hat{\Sigma}_{d,n}^{(1)}(A,B)=\frac{(\log B)^{2n+2}}{2}\sum_{\substack{(\bx,\by)\in\Omega'_{d,n}(B)\\\mu(\bx,\by)\le A}}\mathcal{S}_{0}^{\ast}(\Gamma_{\bx,\by};\mathcal{B}_{N_{d,n}}(A)),\]
where, as in \cite[(4.27)]{bbs2023},
\[\mu(\bx,\by)=3n^{2}\max\left\{\frac{\fd_{2}(\bx,\by)}{\fd_{3}(\bx,\by)},\frac{\Vert\bx\Vert\cdot\Vert\by\Vert}{\fd_{2}(\bx,\by)^{2}}\right\},\]
and analogously to \cite[(4.29)]{bbs2023}
\[\hat{\Sigma}_{d,n}^{(2)}(A,B)=\hat{D}_{d,n}(A,B)-\hat{\Sigma}_{d,n}^{(1)}(A,B).\]

By following, with slight modifications, the argument in \cite[Lemma 4.8]{bbs2023} up to equation (4.30) therein, using Lemma \ref{Esandwich} instead of \cite[Lemma 4.5]{bbs2023}, we arrive at
\begin{equation}\label{hatSigmaapprox}\hat{\Sigma}^{(1)}_{d,n}(A,B)=\iota'_{d,n}A^{N_{d,n}-2}E'_{d,n}(B)\left(1+O\left(\hat{\mathcal{E}}_{d,n}^{(1)}(A,B)\right)\right),\end{equation}
where
\[\hat{\mathcal{E}}_{d,n}^{(1)}(A,B)=\frac{(\log B)^{2n+2}}{AB^{2}}\sum_{(\bx,\by)\in\Omega'_{d,n}(B)}\frac{\mu(\bx,\by)}{\det(\Gamma_{\bx,\by})}.\]
We trivially have 
\begin{equation}\label{hatcalEboundcalE}\hat{\mathcal{E}}_{d,n}^{(1)}(A,B)\le (\log B)^{2n+2}\mathcal{E}_{d,n}^{(1)}(A,B)\end{equation}
where $\mathcal{E}_{d,n}^{(1)}$ is as defined in the proof of \cite[Lemma 4.8]{bbs2023}. It was shown there that when $A\ge B/(\log B)$, we have,
\begin{equation}\label{calEbound1}\mathcal{E}_{d,n}^{(1)}(A,B)\ll \frac{(\log A)^{9/2}}{A^{1/2}}.\end{equation}
We need a bound that holds for $A\ge B/(\log B)^{m}$, and fortunately the proof of \eqref{calEbound1} still works here with the only change being an increase in the exponent of $\log A$. We have,
\[\mathcal{E}_{d,n}^{(1)}(A,B)\ll \frac{(\log A)^{4+m/2}}{A^{1/2}}.\]
From this, and \eqref{hatcalEboundcalE} it follows that
\[\hat{\mathcal{E}}_{d,n}^{(1)}(A,B)\ll \frac{(\log A)^{2n+6+m/2}}{A^{1/2}}.\]
Similarly, we have
\begin{equation}\label{sigma2hatboundsigma2}\hat{\Sigma}_{d,n}^{(2)}(A,B)\le (\log B)^{2n+2}\Sigma_{d,n}^{(2)}(A,B),\end{equation}
where $\Sigma_{d,n}^{(2)}$ is as defined in \cite[(4.29)]{bbs2023}. When $A\ge B/(\log B)$ we have a bound \cite[(4.40)]{bbs2023} which states that
\[\Sigma_{d,n}^{(2)}(A,B)\ll A^{N_{d,n}-2} E_{d,n}(B)\frac{(\log A)^{2N_{n,n}-10}}{A^{2/3}},\]
where $E_{d,n}(B)$ is as defined in \cite[(4.6)]{bbs2023}. As before, the proof of this still works when $A\ge B/(\log B)^{m}$ at the cost of increasing the power of $\log A$. The bound we get is
\[\Sigma_{d,n}^{(2)}(A,B)\ll_{m}A^{N_{d,n}-2} E_{d,n}(B)\frac{(\log A)^{(2N_{n,n}-12)m+2}}{A^{2/3}}.\]
Combining this with \eqref{sigma2hatboundsigma2}, Lemma \ref{Esandwich} and the bound $E_{d,n}(B)\ll B^{2}$ from \cite[Lemma 4.5]{bbs2023}, we get
\[\hat{\Sigma}_{d,n}^{(2)}(A,B)\ll_{m} A^{N_{d,n}-2}E'_{d,n}(B)\frac{(\log A)^{(2N_{n,n}-12)m+2n+4}}{A^{2/3}},\]
which gives the result.
\end{proof}
\end{lemma}

Now, we deal with the $n=d=4$ case, in an analogue of \cite[Lemma 4.9]{bbs2023}.
\begin{lemma}\label{Dapprox44}
Suppose $A\ge B/(\log B)^{m}$, where $m\ge 1$, then
\[\hat{D}_{4,4}(A,B)=\iota'_{4,4}A^{N_{4,4}-2}E'_{4,4}(B)\left(1+O_{m}\left(\frac{1}{A^{1/21}}\right)\right).\]
\begin{proof}
Some parts of the proof of Lemma \ref{Dapprox} still work in this case. In particular, we may define $\hat{\Sigma}_{4,4}^{(1)}$ and $\hat{\Sigma}_{4,4}^{(2)}$ as before and \eqref{hatSigmaapprox}, \eqref{hatcalEboundcalE}, and \eqref{sigma2hatboundsigma2} still hold. As mentioned in the proof of \cite[Lemma 4.9]{bbs2023} the bound \cite[(4.31)]{bbs2023} also holds. Hence, we have,
\[\mathcal{E}_{4,4}^{(1)}(A,B)\ll (\log B)^{4}\frac{B^{2/3}}{A},\]
so, using \eqref{hatcalEboundcalE} and the fact that $B\ll A(\log A)^{m}$,
\[\hat{\mathcal{E}}_{4,4}^{(1)}(A,B)\ll \frac{(\log A)^{14+2m/3}}{A^{1/3}}.\]
This gives, using \eqref{hatSigmaapprox},
\[\hat{D}_{4,4}(A,B)=\iota'_{4,4}A^{N_{4,4}-2}E'_{4,4}(B)\left(1+O\left(\frac{(\log A)^{14+2m/3}}{A^{1/3}}\right)\right)+\hat{\Sigma}_{4,4}^{(2)}(A,B).\]
We bound $\hat{\Sigma}_{4,4}^{(2)}(A,B)$ in terms of $\Sigma_{4,4}^{(2)}(A,B)$ using \eqref{sigma2hatboundsigma2}.
In the proof of \cite[Lemma 4.9]{bbs2023} it is shown that when $A\ge B/(\log B)$,
\[\Sigma_{4,4}^{(2)}(A,B)\ll A^{N_{4,4}-2}E_{4,4}(B)\cdot \frac{(\log A)^{134}}{A^{5/101}}.\]
By the same argument, but using the assumption $A\ge B/(\log B)^{m}$ instead of $A\ge B/(\log B)$, we can show that
\[\Sigma_{4,4}^{(2)}(A,B)\ll A^{N_{4,4}-2}E_{4,4}(B)\cdot\frac{(\log A)^{132m+2}}{A^{5/101}}.\]
Recall that $E_{4,4}(B) \ll B^{2} \ll E'_{d,n}(B)$ from Lemma \ref{Esandwich} and \cite[Lemma 4.5]{bbs2023}. Then we then have, by \eqref{sigma2hatboundsigma2}:
\[\hat{\Sigma}_{4,4}^{(2)}(A,B)\ll (\log A)^{2n+2}\Sigma_{4,4}^{(2)}(A,B)\ll A^{N_{4,4}-2}E'_{4,4}(B)\cdot \frac{(\log A)^{2n+132m+4}}{A^{5/101}}.\]
The result follows.
\end{proof}
\end{lemma}

Now we prove an analogue of \cite[Lemma 4.10]{bbs2023}.
\begin{lemma}\label{Dmixapprox}
Let $d\ge 2$ and $n\ge d$ with $(d,n)\ne (2,2)$. Then for all $B^{5/6}\le A\le B^{2}$, we have,
\[\hat{D}_{d,n}^{\mathrm{mix}}(A,B)=\iota'_{d,n}A^{N_{d,n}-2}E'_{d,n}(B)\left(1+O\left(\frac{1}{(\log_{(2)} A)^{n-2-\eps}}\right)\right).\]
\begin{proof}
We have, recalling \eqref{lotsofdefs},
\[\hat{D}_{d,n}^{\mathrm{mix}}(A,B)=(\log B)^{2n+2}\frac{\alpha W}{2}\sum_{(\bx,\by)\in\Omega'_{d,n}(B)}\frac{\mathcal{S}_{1}^{\ast}\left(\Gamma_{\bx,\by}^{\mathrm{mix}}(W);\mathcal{T}_{\by}^{\mathrm{mix}}(A,\alpha)\right)}{\Vert \nu_{d,n}(\by)\Vert}.\]
We then take, analogously to \cite[(4.49)]{bbs2023},
\[\hat{\Sigma}_{1}^{\mathrm{mix}}(A,B)=(\log B)^{2n+2}\frac{\alpha W}{2}\sum_{\substack{(\bx,\by)\in\Omega'_{d,n}(B)\\\mu(\bx)\le A/W}}\frac{\mathcal{S}_{1}^{\ast}\left(\Gamma_{\bx,\by}^{\mathrm{mix}}(W);\mathcal{T}_{\by}^{\mathrm{mix}}(A,\alpha)\right)}{\Vert\nu_{d,n}(\by)\Vert},\]
where, as in \cite[(4.15)]{bbs2023},
\begin{equation}\label{mudef}\mu(\bx)=n\frac{\Vert \bx\Vert}{\fd_{2}(\bx)},\end{equation}
and
\[\hat{\Sigma}_{2}^{\mathrm{mix}}(A,B)=\hat{D}_{d,n}^{\mathrm{mix}}(A,B)-\hat{\Sigma}_{1}^{\mathrm{mix}}(A,B).\]
We now claim that:
\begin{align*}\hat{\Sigma}_{1}^{\mathrm{mix}}&(A,B)=\\ &\iota'_{d,n}A^{N_{d,n}-2}E'_{d,n}(B)\left(1+O\left(\frac{1}{w^{N_{d,n}-3}}+\frac{F'_{d,n}(B)}{B^{2}}+\frac{G_{d,n}(B)}{B^{2}}\alpha^{2n+2}\right)\right),\end{align*}
where, as in \cite{bbs2023} (the equation following \cite[(4.60)]{bbs2023}),
\[G_{d,n}(B)=\frac{1}{A^{3/4}}\sum_{(\bx,\by)\in\Omega_{d,n}(B)}\frac{\mu(\bx)}{\det\left(\Lambda_{\nu_{d,n}(\bx)}\cap\Lambda_{\nu_{d,n}(\by)}\right)}.\]
This follows by the same argument as \cite[(4.60)]{bbs2023}, but with Lemma \ref{Esandwich} instead of \cite[Lemma 4.5]{bbs2023}. It is important to mention that in \cite{bbs2023}, they use larger values of $w$ and $W$, but the argument also holds for our choice of $w$ and $W$.

By \cite[(4.61)]{bbs2023}, we have
\[G_{d,n}(B)\ll \frac{B^{2}}{A^{1/3}}.\]
Combining this with Lemma \ref{Fbound} and recalling the definition of $w$ gives:
\[\hat{\Sigma}_{1}^{\mathrm{mix}}(A,B)=\iota'_{d,n}A^{N_{d,n}-2}E'_{d,n}(B)\left(1+O\left(\frac{1}{(\log_{(2)} A)^{n-2-\eps}}\right)\right).\]
We have used the fact that $\log_{(2)} B\gg \log_{(2)} A$ which follows from $A\le B^{2}$ and the fact that $N_{d,n}-3\ge n-2$.

Now we bound $\hat{\Sigma}_{2}^{\mathrm{mix}}$. By examining the proof of \cite[Lemma 4.10]{bbs2023}, we see that \cite[(4.63)]{bbs2023} still holds for our smaller values of $w$ and $W$ and therefore, by Lemma \ref{Esandwich} and \cite[Lemma 4.5]{bbs2023}:
\[\hat{\Sigma}_{2}^{\mathrm{mix}}(A,B)\le (\log B)^{2n+2}\Sigma_{2}^{\mathrm{mix}}(A,B)\ll A^{N_{d,n}-2}E'_{d,n}(B)\frac{\alpha^{2n+2}}{A^{1/10}},\]
where $\Sigma_{2}^{\mathrm{mix}}$ is as defined in \cite[(4.50)]{bbs2023}. Putting together the bounds for $\hat{\Sigma}_{1}^{\mathrm{mix}}$ and $\hat{\Sigma}_{2}^{\mathrm{mix}}$ completes the proof.
\end{proof}
\end{lemma}

We now prove an analogue of \cite[Lemma 4.11]{bbs2023}.
\begin{lemma}\label{Dlocapprox}
Let $d\ge 2$ and $n\ge d$ with $(d,n)\ne (2,2)$. Then we have, for all $B^{1/2}\le A\le B^{2}$:
\[\hat{D}_{d,n}^{\mathrm{loc}}(A,B)=\iota'_{d,n}A^{N_{d,n}-2}E'_{d,n}(B)\left(1+O\left(\frac{1}{(\log_{(2)} A)^{n-2-\eps}}\right)\right).\]
\begin{proof}
By a similar argument to the proof of \cite[Lemma 4.11]{bbs2023}, we have
\[\hat{D}_{d,n}^{\mathrm{loc}}(A,B)=\iota'_{d,n}A^{N_{d,n}-2}E'_{d,n}(B)\left(1+O\left(\frac{1}{w^{N_{d,n}-3}}+\frac{F'_{d,n}(B)}{B^{2}}\right)\right).\]
The result then follows from Lemma \ref{Fbound} and the fact $\log_{(2)} B\gg \log_{(2)} A$ (because $A\le B^{2}$).
\end{proof}
\end{lemma}

There is one more lemma to prove before we can prove Theorem \ref{mainvar}. This is an analogue of \cite[Lemma 4.7]{bbs2023}.
\begin{lemma}\label{sumNVbound}
Let $d\ge 2$ and $n\ge d$. Assume that $A\ge B^{4/5}$. Then we have
\[\frac{1}{\#\bV_{d,n}(A)}\sum_{V\in\bV_{d,n}(A)}N'_{V}(B)\ll \frac{B}{A}.\]
\begin{proof}
We mainly follow the proof of \cite[Lemma 4.7]{bbs2023}. First, note that
\[\sum_{V\in\bV_{d,n}(A)}N'_{V}(B)=\frac{(\log B)^{n+1}}{2} \sum_{\bx \in \Xi'_{d,n}(B)} \#\left( \Lambda_{\nu_{d,n}(\bx)}\cap \bZ_{\mathrm{prim}}^{N_{d,n}}\cap \mathcal{B}_{N_{d,n}}(A)\right).\]
Recalling the definition \eqref{mudef} of $\mu(\bx)$, we define,
\[\hat{M}_{d,n}^{(1)}(A,B)=\frac{(\log B)^{n+1}}{2} \sum_{\substack{\bx \in \Xi'_{d,n}(B)\\ \mu(\bx)\le A}} \#\left( \Lambda_{\nu_{d,n}(\bx)}\cap \bZ_{\mathrm{prim}}^{N_{d,n}}\cap \mathcal{B}_{N_{d,n}}(A)\right)\]
and
\[\hat{M}_{d,n}^{(2)}(A,B)=\sum_{V\in\bV_{d,n}(A)}N'_{V}(B)-\hat{M}_{d,n}^{(1)}(A,B).\]
These are analogous to the quantities $M_{d,n}^{(1)}(A,B)$ and $M_{d,n}^{(2)}(A,B)$ from the proof of \cite[Lemma 4.7]{bbs2023}. Following the same argument as in \cite[Lemma 4.7]{bbs2023}, we can show that
\[\hat{M}_{d,n}^{(1)}(A,B)\ll (\log B)^{n+1}\sum_{\bx \in \Xi'_{d,n}(B)} \frac{A^{N_{d,n}-1}}{\|\bx\|^{d}}\ll A^{N_{d,n}-1}B.\]
We also have, from the penultimate equation in the proof of \cite[Lemma 4.7]{bbs2023},
\begin{align*}\hat{M}_{d,n}^{(2)}(A,B)&\ll (\log B)^{n+1}M_{d,n}^{(2)}(A,B)\\ &\ll A^{N_{d,n}-1}\left(\frac{B^{1+1/(n+1-d)}}{A^{n}}+\frac{B^{(n+2)/(n+1-d)}}{A^{2n}}\right)(\log B)^{n+2},\end{align*}
which is sufficient, recalling that $A\ge B^{4/5}$.
\end{proof}
\end{lemma}

We are now finally ready to deduce our main variance estimate.
\begin{proof}[Proof of Theorem \ref{mainvar}]
Defining
\[\hat{K}(A,B)=(\log B)^{n+1}\sum_{V\in \bV_{d,n}(A)} \left(N'_{V}(B)+\hat{\Delta}_{V}^{\mathrm{mix}}(B)+\hat{\Delta}_{V}^{\mathrm{loc}}(B)\right),\]
we have
\begin{align*}\sum_{V\in\bV_{d,n}(A)}\left(N'_{V}(B)-N^{\mathrm{ploc}}_{V}(B)\right)^{2}&=\hat{D}_{d,n}(A,B)-2\hat{D}_{d,n}^{\mathrm{mix}}(A,B)\\&+\hat{D}_{d,n}^{\mathrm{loc}}(A,B)+O(\hat{K}(A,B)).\end{align*}
Recall that $B=A(\log A)^{m}$. By combining Lemmas \ref{Dapprox}, \ref{Dapprox44}, \ref{Dmixapprox} and \ref{Dlocapprox}, and using the upper bound for $E'_{d,n}$ given in Lemma \ref{Esandwich}, we get
\[\frac{1}{\#\bV_{d,n}(A)}\left(\hat{D}_{d,n}(A,B)-2\hat{D}_{d,n}^{\mathrm{mix}}(A,B)+\hat{D}_{d,n}^{\mathrm{loc}}(A,B)\right)\ll \frac{B^{2}}{A^{2}(\log_{(2)} B)^{n-2-\eps}}.\]
Now we claim that
\begin{equation}\label{hatKbound}\frac{\hat{K}(A,B)}{\#\bV_{d,n}(A)}\ll (\log B)^{n+1}\frac{B}{A}.\end{equation}
We bound the contribution from $N'_{V}(B)$ using Lemma \ref{sumNVbound}, so we consider $\hat{\Delta}_{V}^{\mathrm{mix}}(B)$ and $\hat{\Delta}_{V}^{\mathrm{loc}}(B)$. Recalling \eqref{Deltamixdef} and applying partial summation, we obtain, in a similar way to \cite[Equation (4.71)]{bbs2023} 
\begin{align*}\frac{1}{\#\bV_{d,n}(A)}\sum_{V\in\bV_{d,n}(A)}\hat{\Delta}_{V}^{\mathrm{mix}}(B)&\ll (\log B)^{n+1}\frac{\alpha W}{A^{2}} \sum_{\bx \in \Xi'_{d,n}(B)} \frac{1}{\| \nu_{d,n}(\bx) \|} \\&\ll \frac{B(\log B)^{1+o(1)}}{A^{2}}\end{align*}
and in a similar way to \cite[(4.72)]{bbs2023}, it follows from the trivial bound
\[\hat{\Delta}_{V}^{\mathrm{loc}}(B) \ll (\log B)^{n+1} \frac{\alpha^{2}W^{2}}{\|\ba_{V}\|^{2}} \sum_{\bx \in \Xi'_{d,n}(B)} \frac{1}{\|\nu_{d,n}(\bx)\|},\]
that
\[\frac{1}{\#\bV_{d,n}(A)}\sum_{V\in\bV_{d,n}(A)}\hat{\Delta}_{V}^{\mathrm{loc}}(B)\ll \frac{B(\log B)^{2+o(1)}}{A^{2}}.\]
The inequality \eqref{hatKbound} then follows. It then follows from our choice of $B$ that
\[\frac{\hat{K}(A,B)}{\#\bV_{d,n}(A)}\ll \frac{B^{2}}{A^{2}(\log B)^{m-(n+1)}}.\]
Now recall that we assumed $m>n+1$.
\end{proof}
\begin{remark}
The lower bound of $m>n+1$ is necessary because of the contribution from $(\log B)^{n+1}\sum_{V} N'_{V}(B)$. This is a count of solutions to $\langle \ba,\nu_{d,n}(\bx)\rangle=0$ weighted by $(\log B)^{2n+2}$. These correspond to the weighted solutions $(\ba,\bx,\by)$ to $\langle \ba,\nu_{d,n}(\bx)\rangle=\langle \ba,\nu_{d,n}(\by)\rangle=0$ with $\bx=\by\in \mathcal{P}^{n+1}$, with a weight of $(\log B)^{n+1}$ for each of the variables $\bx$, $\by$. This results in a weighting that is too large by a factor of $(\log B)^{n+1}$, which means that unless $B/A$ is large enough, it will dominate over $\hat{D}_{d,n}(A,B)$, (which is the weighted count of solutions with $\bx \ne \by$). 
\end{remark}

\section{Bounding the local counting function from below}\label{localcountlowerboundsection}
In this section, we establish Theorem \ref{localcountnotsmall}.
\begin{defn} Recall the definition \eqref{defcalC} of $\mathcal{C}_{\nu_{d,n}(\bu)}^{(\gamma)}$. For $\ba\in \bR^{N_{d,n}}$ and $\gamma>0$, let
\[\tau'(\ba;\gamma)=\gamma\cdot\mathrm{vol}\left(\left\{\bu\in\mathcal{B}_{n+1}(1)\cap(\bR^{+})^{n+1}:\ba\in \mathcal{C}_{\nu_{d,n}(\bu)}^{(\gamma)}\right\}\right).\]
Then, recalling that $\alpha=\log B$, we let
\begin{equation}\label{defarchfactor}\fJ'_{V}(B)=\tau'(\ba_{V};\alpha),\end{equation}
which is the Archimedean local factor.

Given $Q\ge 1$ and $N\ge 1$ integers and $\ba\in\bZ^{N_{d,n}}$, we let
\[\fR'_{N}(Q)=\left((\bZ/Q\bZ)^{\times}\right)^{N}\]
and
\[\sigma'(\ba;Q)=\frac{Q}{\varphi(Q)^{n+1}}\cdot\#\left\{\bb\in \fR'_{n+1}(Q):\ba\in\Lambda_{\nu_{d,n}(\bb)}^{(Q)}\right\}.\]
We then let
\[\fS'_{V}(B)=\sigma'(\ba_{V};W),\]
which is the non-Archimedean local factor.
\end{defn}
\begin{lemma}\label{localcountlowerbound}
Let $d\ge 2$, $n\ge d$, $A\ge 1$ and $B\ge 1$. For any $V\in\bV_{d,n}(A)$, we have, for all $C>0$:
\[\fS'_{V}(B)\fI'_{V}(B)\ll_{C} \frac{A}{B}N^{\mathrm{ploc}}_{V}(B)+\frac{1}{(\log B)^{C}}.\]
\begin{proof}
As in the proof of \cite[Lemma 5.1]{bbs2023}, we can show
\begin{equation}\label{localcountlowerboundineq1}(\log B)^{n+1}\sum_{\substack{\bx\in\Xi'_{d,n}(B)\\\ba_{V}\in\Lambda_{\nu_{d,n}(\bx)}^{(W)}\cap\mathcal{C}_{\nu_{d,n}(\bx)}^{(\alpha)}}} 1\ll \frac{AB^{d/(n+1-d)}}{\alpha W}N^{\mathrm{ploc}}_{V}(B).\end{equation}
We break the summation into residue classes modulo $W$, and discard the contribution from those classes which do not lie in $\fR'_{n+1}(W)$ to get:  
\begin{equation}\label{localcountlowerboundineq2}\sum_{\substack{\bx\in\Xi'_{d,n}(B)\\\ba_{V}\in\Lambda_{\nu_{d,n}(\bx)}^{(W)}\cap\mathcal{C}_{\nu_{d,n}(\bx)}^{(\alpha)}}} 1\ge\sum_{\substack{\bb\in\fR'_{n+1}(W)\\\ba_{V}\in\Lambda_{\nu_{d,n}(\bb)}^{(W)}}}\#\left\{\bx\in \Xi'_{d,n}(B):\begin{array}{l l}\bx\equiv \bb \:\text{(mod $W$)}\\\ba_{V}\in \mathcal{C}_{\nu_{d,n}(\bx)}^{(\alpha)}\end{array}\right\}.\end{equation}
We define
\[R_{V,B}=\left\{\bu\in \mathcal{B}_{n+1}\left(B^{1/(n+1-d)}\right)\cap(\bR^{+})^{n+1}:\ba_{V}\in\mathcal{C}_{\nu_{d,n}(\bu)}^{(\alpha)}\right\}\]
and
\[\mathcal{P}_{\bb}=\left\{\bx\in \mathcal{P}^{n+1}\setminus \Span\{(1,\dots,1)\}:\bx\equiv \bb \text{(mod $W$)}\right\}\]
so that
\[\#\left\{\bx\in \Xi'_{d,n}(B):\begin{array}{l l}\bx\equiv \bb \:\text{(mod $W$)}\\\ba_{V}\in \mathcal{C}_{\nu_{d,n}(\bx)}^{(\alpha)}\end{array}\right\}=\#\left(R_{V,B}\cap \mathcal{P}_{\bb}\right)\]
Then if we combine \eqref{localcountlowerboundineq1} and \eqref{localcountlowerboundineq2}, we get
\begin{equation}\label{localcountlowerboundineq3}\frac{AB^{d/(n+1-d)}}{\alpha W}N_{V}^{\mathrm{ploc}}(B)\gg (\log B)^{n+1}\sum_{\substack{\bb\in\fR'_{n+1}(W)\\ \ba_{V}\in\Lambda_{\nu_{d,n}(\bb)}^{(W)}}}\#\left(R_{V,B}\cap\mathcal{P}_{\bb}\right).\end{equation}
We split the region $R_{V,B}$ into cubes of side length approximately $B^{\theta/(n+1-d)}$ where $3/5<\theta<1$. On such a cube, we can then approximate the number of points with prime coordinates congruent to $\bb$ by using a short intervals version of the Siegel-Walfisz theorem. Let $x$ be large and $q\ll (\log x)^{C_{1}}$, $\gcd(a,q)=1$ where $C_{1}>0$ is arbitrary and fixed. Then for $x/2<z\le x$ and $h\le x^{\theta}$, we have,
\begin{align}\label{SWshortintervals}\pi(z+h,q;a)-\pi(z,q;a)&\gg \frac{\psi(z+h,q;a)-\psi(z,q;a)}{\log x}\\ \nonumber &\ge \frac{h}{\varphi(q)(\log x)}-O\left(\frac{x^{\theta}}{(\log x)^{C_{2}}}\right),\end{align}
where $C_{2}>0$ is arbitrarily large and fixed. This follows readily from the main result of \cite{pps1984}. For this to be nontrivial, we will choose some arbitrary $C_{3}>0$ and let $z_{0},z_{1},\dots,z_{m}$ be such that
\[\frac{B^{1/(n+1-d)}}{(\log B)^{C_{3}}}+\eta=z_{0}<z_{1}<\cdots<z_{m}=B^{1/(n+1-d)}+\eta\]
and $z_{j+1}-z_{j}=h$ for all $0\le j\le m-1$, where $\eta=\eta(V,B)\in (0,1)$ will be chosen later and $m$ is chosen such that we have
\[z_{0}^{\theta}\ll h\le z_{0}^{\theta}.\]
For $0\le j \le m-1$, we apply \eqref{SWshortintervals} with $x=z=z_{j}$ and $q=W$. By Lemma \ref{Wbound}, we have $W \ll \log B$ and we also have $\log z_{j} \asymp \log B$, so we get
\[\pi(z_{j+1},W;a)-\pi(z_{j},W;a)\ge \frac{h}{\varphi(W)(\log B)}-O\left(\frac{z_{j}^{\theta}}{(\log B)^{C_{2}}} \right).\]
We note that $z_{j}^{\theta}\le (z_{0}(\log B)^{C_{3}})^{\theta} \ll h (\log B)^{\theta C_{3}}$, so the error term is at most
\[O\left( h(\log B)^{\theta C_{3}-C_{2}} \right).\]
Note also that $\varphi(W) \le W \ll (\log B)$. Hence, if we choose $C_{2}$ to be large enough in terms of $C_{3}$, then
\[\pi(z_{j+1},W;a)-\pi(z_{j},W;a)\gg \frac{h}{\varphi(W)(\log B)}\]
for all $0\le j\le m-1$ and $\gcd(a,W)=1$.

Now we consider the cubes $C_{\bj}=C_{j_{0},\dots,j_{n}}=[z_{j_{0}},z_{j_{0}+1}]\times\cdots\times[z_{j_{n}},z_{j_{n}+1}]$ for $\bj \in \{0,\dots,m-1\}^{n+1}$.  We have,
\begin{equation}\label{cubecount}\#\left(C_{\bj}\cap \mathcal{P}_{\bb}\right)\gg \frac{h^{n+1}}{\varphi(W)^{n+1}(\log B)^{n+1}}\end{equation}
for all $\bj\in \{0,\dots,m-1\}^{n+1}$. We will obtain a lower bound for $\#(R_{V,B}\cap \mathcal{P}_{\bb})$ by restricting ourselves to those cubes which lie entirely within $R_{V,B}$, on which we can apply \eqref{cubecount}, and show that the total volume of these cubes is close to the volume of $R_{V,B}$. Any points in $R_{V,B}$ which are not covered by one of these cubes must either lie in a cube which intersects $R_{V,B}$ but is not completely contained in it, or one of its coordinates must be less than $z_{0}=B^{1/(n+1-d)}/(\log B)^{C_{3}}+\eta$. Hence we can split $R_{V,B}$ into pieces $R_{V,B}^{(j)}$, $j \in \{1,2,3\}$ where
\begin{align*}
R_{V,B}^{(1)}&=\bigcup_{\substack{\bj \in \{0,\dots,m-1\}^{n+1}\\ C_{\bj} \subset R_{V,B}}} C_{\bj}\\
R_{V,B}^{(2)}&=\bigcup_{\substack{\bj \in \{0,\dots,m-1\}^{n+1}\\ C_{\bj} \cap R_{V,B} \ne \emptyset, \:C_{\bj} \not\subset R_{V,B}}} C_{\bj} \\
R_{V,B}^{(3)}&=\bigcup_{i=0}^{n} R_{V,B} \cap \{\bx \in \left(\bR^{+}\right)^{n+1}: x_{i} < B^{1/(n+1-d)}/(\log B)^{C_{3}}+\eta\}.
\end{align*}
As $R_{V,B}$ is contained in the ball of radius $B^{1/(n+1-d)}$ centred at the origin, we get
\begin{equation}\label{RVBedgevolume}\mathrm{vol}\left(R_{V,B}^{(3)}\right) \ll B^{(n+1)/(n+1-d)}/(\log B)^{C_{3}}.\end{equation}
For $\bR_{V,B}^{(2)}$, let $\fC$ be the number of cubes of the form $C_{\bj}$ which intersect $R_{V,B}$, but do not lie completely within it. This is the same as the number of cubes which intersect the boundary of $R_{V,B}$. The boundary of $R_{V,B}$ is contained in the union of two algebraic sets, namely the sphere of radius $B^{1/(n+1-d)}$ and centre $\mathbf{0}$, and $\{\bx\in \bR^{n+1}:4\alpha^{2}\langle \ba_{V},\nu_{d,n}(\bx)\rangle^{2}=\Vert \ba_{V}\Vert^{2}\Vert \nu_{d,n}(\bx)\Vert^{2}\}$. We will call these sets $X_{1}$ and $X_{2}$ respectively.

The set $X_{1}$ is the vanishing set of a single nonzero polynomial of degree $2$, and has codimension 1 (i.e. dimension $n$, where we mean the dimension as a real manifold), while $X_{2}$ is the vanishing set of the function $g_{\ba}(\bx)=4\alpha^{2}\langle \ba_{V},\nu_{d,n}(\bx)\rangle^{2}-\Vert \ba_{V}\Vert^{2}\Vert \nu_{d,n}(\bx)\Vert^{2}$, which is a polynomial in $\bx$ of degree at most $2d$ (whose coefficients depend on $\ba_{V}$ and $B$). We want $X_{2}$ to also have codimension at least 1. This will hold provided $g_{\ba}$ is not the zero polynomial. To see this, suppose that $g_{\ba}$ is identically zero. Then $\|\nu_{d,n}(\bx)\|^{2}=(2\alpha/\|\ba_{V}\|)^{2}\langle \ba_{V},\nu_{d,n}(\bx)\rangle^{2}$ for all $\bx \in \bR^{n+1}$. This means that $\|\nu_{d,n}(\bx)\|^{2}$ is a square in the ring $\bR[\bx]$. But then $h(x)=\| \nu_{d,n}(1,x,0,\dots,0)\|^{2}=1+x^{2}+\cdots+x^{2d}$ is a square in $\bR[x]$. We can easily factorise $h(x)$ into distinct linear factors over $\bC$, so it is squarefree and this is impossible.

We now use the main theorem of \cite{sz2014}. For some integers $N, D\ge 1$, suppose we have an algebraic set $X \subset \bR^{N}$ of (manifold) dimension $l<N$, which is the vanishing set of a set of polynomials with total degree $D$. We define a box to be a set of the form $\mathcal{B}=[a_{1},b_{1}]\times\cdots\times [a_{N},b_{N}]$ for integers $a_{i}\le b_{i}$. A box can be divided up into unit cubes whose vertices lie in $\bZ^{N}$. The main theorem of \cite{sz2014} tells us that the number of these unit cubes which intersect $X$ is at most a certain quantity $N_{\mathcal{B}}(D,l)$, where
\[N_{\mathcal{B}}(D,l) \ll_{D} \max_{1\le i_{1} <\dots<i_{l}\le N} \prod_{j=1}^{l}(b_{i_{j}}-a_{i_{j}}),\]
provided that none of the vertices of these cubes intersect a certain set $S_{X}$. This set $S_{X}$ is the union of all $(N-l-1)$-dimensional planes parallel to a coordinate subspace which intersect $X$. Crucially, the implied constant only depends on $D$ and $N$, not on the coefficients of the polynomials defining $X$. 
\begin{remark}
In the statement of the main theorem of \cite{sz2014}, the set $S_{X}$ is defined to instead be a union of $(N-l)$-dimensional planes. However, this appears to be a typo, and in the proof it is defined as we have stated.
\end{remark}
In order to apply the theorem, we rescale everything by $h^{-1}$, so that our cubes $C_{\bj}$ become unit cubes, and translate them by a small distance so that their vertices lie in $\bZ^{n+1}$. The resulting cubes form a box $\mathcal{B}$ of width less than $B^{1/(n+1-d)}h^{-1}$. Let $Y_{1},Y_{2}$ be the result of dilating and translating $X_{1}$ and $X_{2}$ respectively. Assume for now that
\begin{equation}\label{SYcondition}S_{Y_{1}}\cap \bZ^{n+1}=S_{Y_{2}}\cap\bZ^{n+1}=\emptyset.
\end{equation}
The quantity $\fC$ is at most the number of unit cubes in $\mathcal{B}$ which intersect one of $Y_{1},Y_{2}$, so by applying \cite{sz2014} with $N=n+1$, $D\le 2d$, $X=Y_{i}$, and $l=\dim Y_{i}\le n$ we get the bound
\[\fC \ll \left( B^{1/(n+1-d)}h^{-1}\right)^{n}.\]

Each cube $C_{\bj}$ has volume $h^{n+1}$ so it follows that
\begin{equation}\label{boundarycubes}\mathrm{vol}\left(R_{V,B}^{(2)}\right) \le \fC h^{n+1} \ll B^{n/(n+1-d)}h\ll B^{(n+\theta)/(n+1-d)}.\end{equation}
It remains to verify \eqref{SYcondition}, and this is where we make our choice of $\eta$. The main theorem of \cite{sz2014} also tells us that the sets $S_{Y_{1}},S_{Y_{2}}$ are algebraic sets properly contained in $\bR^{n+1}$, so in particular they are nowhere dense. Hence, if we translate them by a sufficiently small amount then we can ensure that they do not intersect the discrete set $\bZ^{n+1}$. Changing the value of $\eta$ corresponds to a translation of $Y_{1},Y_{2}$, and hence of $S_{Y_{1}},S_{Y_{2}}$, so there is some choice of $\eta$ for which \eqref{SYcondition} holds.

Combining \eqref{RVBedgevolume} and \eqref{boundarycubes}, we get
\begin{equation}\label{volapprox}\mathrm{vol}\left(R_{V,B}^{(1)}\right) = \mathrm{vol}\left(R_{V,B}\right) + O\left( B^{(n+1)/(n+1-d)}/(\log B)^{C_{3}} \right).\end{equation}

Now we use \eqref{cubecount} for each cube lying entirely within $R_{V,B}$ to get that
\begin{equation}\label{totalcubecount}\#(R_{V,B}\cap\mathcal{P}_{\bb}) \gg \frac{\mathrm{vol}\left(R_{V,B}^{(1)}\right)}{\varphi(W)^{n+1}(\log B)^{n+1}}.\end{equation}
Combining \eqref{volapprox} and \eqref{totalcubecount} and using the fact that
\[\mathrm{vol}(R_{V,B})=B^{(n+1)/(n+1-d)}\fJ'_{V}(B)/\alpha,\]
we get
\[\#(R_{V,B}\cap\mathcal{P}_{\bb})+\frac{B^{(n+1)/(n+1-d)}}{(\log B)^{C_{3}}}\gg \frac{B^{(n+1)/(n+1-d)}}{\varphi(W)^{n+1}(\log B)^{n+1}}\frac{\fJ'_{V}(B)}{\alpha}.\]
If we then put this into \eqref{localcountlowerboundineq3}, and recall the definition of $\fS'_{V}(B)$, we get
\[B^{(n+1)/(n+1-d)}\frac{\fS'_{V}(B)}{W}\frac{\fJ'_{V}(B)}{\alpha}\ll \frac{AB^{d/(n+1-d)}}{\alpha W}N^{\mathrm{ploc}}_{V}(B)+W^{n+1}\frac{B^{(n+1)/(n+1-d)}}{(\log B)^{C_{3}-n-1}}.\]
Rearranging and recalling Lemma \ref{Wbound} gives the result.
\end{proof}
\end{lemma}
The following two lemmas will be proved in the ensuing subsections.
\begin{lemma}\label{nonarchnotsmall}
Let $d\ge 2$ and $n\ge 3$. Let $\psi,\beta:\bR_{>0}\rightarrow\bR_{>1}$ be such that $\psi(A),\beta(A)\le A$. Then for all $C,\eps>0$, we have
\[\frac{1}{\#\bV_{d,n}^{\mathrm{ploc}}(A)} \#\left\{V\in\bV_{d,n}^{\mathrm{ploc}}(A):\fS'_{V}(A\psi(A))<\frac{C}{\beta(A)}\right\}\ll_{C,\eps} \frac{1}{\beta(A)^{1/(3n)-\eps}}.\] 
\end{lemma}
\begin{lemma}\label{archnotsmall}
Let $d\ge 2$ and $n\ge 3$. Let $\psi,\beta:\bR_{>0}\rightarrow\bR_{>1}$ be such that $\psi(A)\le A$, and $\beta(A)\le (\log A)^{(3n+1)/2}$. Then for all $C,\eps>0$, we have
\[\frac{1}{\#\bV_{d,n}^{\mathrm{ploc}}(A)} \#\left\{V\in\bV_{d,n}^{\mathrm{ploc}}(A):\fJ'_{V}(A\psi(A))<\frac{C}{\beta(A)}\right\}\ll_{C,\eps} \frac{1}{\beta(A)^{2/(3n+1)-\eps}}.\]
\begin{proof}[Proof of Theorem \ref{localcountnotsmall}]
Define
\[\hat{\mathscr{Q}}_{\xi}=\frac{1}{\#\bV_{d,n}^{\mathrm{ploc}}(A)}\cdot\#\left\{V\in\bV_{d,n}^{\mathrm{ploc}}(A):N^{\mathrm{ploc}}_{V}(A\psi(A))\le \frac{\psi(A)}{\xi(A)^{1/2}}\right\}.\]
By Lemma \ref{localcountlowerbound} with $C>(9n+1)/2$, $B=A\psi(A)$, and the assumption $\xi(A)\le(\log A)^{9n+1}$, we have a constant $c>0$ depending only on $d,n$ such that
\[\hat{\mathscr{Q}}_{\xi}\le \frac{1}{\#\bV_{d,n}^{\mathrm{ploc}}(A)}\cdot\#\left\{V\in\bV_{d,n}^{\mathrm{ploc}}(A): \fS'_{V}(A\psi(A))\cdot \fI'_{V}(A\psi(A))\le \frac{c}{\xi(A)^{1/2}}\right\}.\]
Let $0<u<1/2$ be a fixed parameter to be chosen later and $v=1/2-u$. Let
\begin{align*}\mathcal{A}_{1}&=\left\{V\in\bV_{d,n}^{\mathrm{ploc}}(A): \fS'_{V}(A\psi(A))\le \frac{c^{1/2}}{\xi(A)^{u}}\right\},\\
\mathcal{A}_{2}&=\left\{V\in\bV_{d,n}^{\mathrm{ploc}}(A): \fI'_{V}(A\psi(A))\le \frac{c^{1/2}}{\xi(A)^{v}}\right\}.
\end{align*}
If $\fS'_{V}(A\psi(A))\cdot \fI'_{V}(A\psi(A))\le \frac{c}{\xi(A)^{1/2}}$ then $V\in \mathcal{A}_{1}$ or $V\in \mathcal{A}_{2}$, so applying Lemma \ref{nonarchnotsmall} with $\beta(A)=\xi(A)^{u}$ and Lemma \ref{archnotsmall} with $\beta(A)=\xi(A)^{v}$ (assuming $\xi(A)^{v}\le (\log A)^{(3n+1)/2}$), we get that,
\[\hat{\mathscr{Q}}_{\xi}\ll \frac{1}{\xi(A)^{u/(3n)-\eps}} + \frac{1}{\xi(A)^{2v/(3n+1)-\eps}}.\]
We now choose $u=3n/(9n+1)$ and find that, since $\xi(A)\le (\log A)^{9n+1}$, the inequality $\xi(A)^{v}\le (\log A)^{(3n+1)/2}$ is satisfied. We conclude that
\[\hat{\mathscr{Q}}_{\xi}\ll \frac{1}{\xi(A)^{1/(9n+1)-\eps}}.\]
\end{proof}
\end{lemma}
\subsection{The non-Archimedean factor is rarely small.}
The purpose of this subsection will be to prove Lemma \ref{nonarchnotsmall}.

For integers $N,Q\ge 1$, let $\fR_{N}$ be defined, as in \cite{bbs2023}, by
\[\fR_{N}(Q)=\left\{\bb\in(\bZ/Q\bZ)^{N}:\gcd(Q,\bb)=1\right\}.\]
Then note that, for a prime $p$ and integer $r\ge 1$,
\[\#\fR_{N}(p^{r})=p^{rN}\left(1-\frac{1}{p^{N}}\right),\]
and
\[\#\fR'_{N}(p^{r})=\varphi(p^{r})^{N}=p^{rN}\left(1-\frac{1}{p}\right)^{N}.\]
\begin{lemma}\label{sigmamean}
Let $d\ge 2$ and $n\ge 3$. Then for all primes $p$, and integers $r\ge 1$, we have
\[\frac{1}{\#\fR_{N_{d,n}}(p^{r})}\sum_{\ba\in\fR_{N_{d,n}}(p^{r})}\sigma'(\ba;p^{r})=1+O\left(\frac{1}{p^{N_{d,n}-1}}\right).\]
\begin{proof}
By changing the order of summation we have
\[\sum_{\ba\in\fR_{N_{d,n}}(p^{r})}\sigma'\left(\ba;p^{r}\right)=\frac{p^{r}}{\varphi(p^{r})^{n+1}}\sum_{\bb\in\fR'_{n+1}(p^{r})}\#\left\{\ba\in\fR_{N_{d,n}}(p^{r}):\ba\in\Lambda_{\nu_{d,n}(\bb)}^{(p^{r})}\right\}.\]
Then we have, as in the proof of \cite[Lemma 5.4]{bbs2023} that
\[\#\left\{\ba\in\fR_{N_{d,n}}(p^{r}):\ba\in\Lambda_{\nu_{d,n}(\bb)}^{(p^{r})}\right\}=p^{r(N_{d,n}-1)}\left(1-\frac{1}{p^{N_{d,n}-1}}\right).\]
Now
\begin{align*}\frac{1}{\#\fR_{N_{d,n}}(p^{r})}\sum_{\ba\in\fR_{N_{d,n}}(p^{r})}\sigma'(\ba;p^{r})&=\frac{\#\fR'_{n+1}(p^{r})}{\varphi(p^{r})^{n+1}}\frac{p^{rN_{d,n}}}{\#\fR_{N_{d,n}}(p^{r})}\left(1-\frac{1}{p^{N_{d,n}-1}}\right)\\&=\left(1-\frac{1}{p^{N_{d,n}}}\right)^{-1}\left(1-\frac{1}{p^{N_{d,n}-1}}\right).\end{align*}
\end{proof}
\end{lemma}
\begin{lemma}\label{sigmavar}
Let $d\ge 2$, $n\ge 3$. Also, let $p$ be a prime number and $r\ge 1$, then
\[\frac{1}{\#\fR_{N_{d,n}}(p^{r})}\sum_{\ba\in\fR_{N_{d,n}}(p^{r})}(\sigma'(\ba;p^{r})-1)^{2}\ll \frac{1}{p^{n-1}}.\]
\begin{proof}
By a similar method to the proof of \cite[Lemma 5.5]{bbs2023}, we have
\[\sum_{\ba\in\fR_{N_{d,n}}(p^{r})}\sigma'(\ba;p^{r})^{2}=\frac{p^{2r}}{\varphi(p^{r})^{2(n+1)}}\sum_{\bb_{1},\bb_{2}\in\fR'_{n+1}(p^{r})}\#L(\bb_{1},\bb_{2};p^{r}),\]
where
\[L(\bb_{1},\bb_{2};p^{r})=\left\{\ba\in\fR_{N_{d,n}}(p^{r}):\ba\in\Lambda_{\nu_{d,n}(\bb_{1})}^{(p^{r})}\cap\Lambda_{\nu_{d,n}(\bb_{2})}^{(p^{r})}\right\}.\]
Now \cite[(5.13)]{bbs2023} says that:
\[\#L(\bb_{1},\bb_{2};p^{r})=p^{r(N_{d,n}-2)}\mathrm{gcd}(\mathcal{G}(\bb_{1},\bb_{2}),p^{r})\left(1+O\left(\frac{1}{p^{N_{d,n}-2}}\right)\right),\]
where we recall that $\mathcal{G}(\bb_{1},\bb_{2})$ is the gcd of the $2\times 2$ minors of the matrix with columns $\bb_{1},\bb_{2}$. There is a slight issue that $\bb_{1}$, $\bb_{2}$ are only defined (mod $p^{r}$), but if we choose representatives for these congruence classes then, although $\mathcal{G}(\bb_{1},\bb_{2})$ may depend on this choice, $\gcd(\mathcal{G}(\bb_{1},\bb_{2}),p^{r})$ does not.

Observe that
\[\sum_{\bb_{1},\bb_{2}\in\fR'_{n+1}(p^{r})}\mathrm{gcd}(\mathcal{G}(\bb_{1},\bb_{2}),p^{r})=\#\fR'_{n+1}(p^{r})^{2}+O\left(\sum_{e=1}^{r}p^{e}\cdot\#\hat{\fF}^{(e)}(p^{r})\right),\]
where, for $e\in\{1,\dots,r\}$,
\[\hat{\fF}^{(e)}(p^{r})=\{(\bb_{1},\bb_{2})\in\fR'_{n+1}(p^{r})^{2}:\mathrm{gcd}(\mathcal{G}(\bb_{1},\bb_{2}),p^{r})=p^{e}\}.\]
For given $\bb_{1}\in\fR'_{n+1}(p^{r})$, there are at most $\varphi(p^{r})p^{(r-e)n}$ choices of $\bb_{2}$ such that $(\bb_{1},\bb_{2})\in\hat{\fF}^{(e)}(p^{r})$, and so
\[\#\hat{\fF}^{(e)}(p^{r})\ll \varphi(p^{r})^{n+2}p^{n(r-e)}.\]
Therefore
\[\sum_{\bb_{1},\bb_{2}\in\fR'_{n+1}(p^{r})}\mathrm{gcd}(\mathcal{G}(\bb_{1},\bb_{2}),p^{r})=\varphi(p^{r})^{2(n+1)}\left(1+O\left(\frac{1}{p^{n-1}}\right)\right).\]
Combining all the above equations, we get
\[\sum_{\ba\in\fR_{N_{d,n}}(p^{r})}\sigma'(\ba;p^{r})^{2}=p^{rN_{d,n}}\left(1+O\left(\frac{1}{p^{n-1}}\right)\right),\]
and
\[\frac{1}{\#\fR_{N_{d,n}}(p^{r})}\sum_{\ba\in\fR_{N_{d,n}}(p^{r})}\sigma'(\ba;p^{r})^{2}=1+O\left(\frac{1}{p^{n-1}}\right).\]
The result then follows by applying Lemma \ref{sigmamean}.
\end{proof}
\end{lemma}
For $e\in\{0,\dots,r\}$, let
\[\hat{\fR}_{N_{d,n}}^{(e)}(p^{r})=\left\{\ba\in\fR_{N_{d,n}}(p^{r}):\exists\bx\in\fR'_{n+1}(p^{r})\:\begin{array}{l l}f_{\ba}(\bx)\equiv 0 \:\text{(mod $p^{r}$)}\\v_{p}(\nabla f_{\ba}(\bx))=e\end{array}\right\},\]
where, for $\bz=(z_{1},\dots,z_{N})\in(\bZ/p^{r}\bZ)^{N}$, $v_{p}(\bz)$ is the largest integer $e\in\{0,\dots,r\}$ such that $p^{e}\mid z_{j}$ for all $j$.

Note that $\hat{\fR}^{(e)}_{N_{d,n}}(p^{r})\subset \fR_{N_{d,n}}^{(e)}(p^{r})$, where
\[\fR_{N_{d,n}}^{(e)}(p^{r})=\left\{\ba\in\fR_{N_{d,n}}(p^{r}):\exists\bx\in\fR_{n+1}(p^{r})\:\begin{array}{l l}f_{\ba}(\bx)\equiv 0 \:\text{(mod $p^{r}$)}\\v_{p}(\nabla f_{\ba}(\bx))=e\end{array}\right\}.\]
An upper bound for $\#\fR_{N_{d,n}}^{(e)}(p^{r})$ was shown in \cite[Lemma 5.6]{bbs2023}, from which it immediately follows that, for all $d\ge 2$ and $n\ge 3$:
\begin{equation}\label{hatRebound}\#\hat{\fR}_{N_{d,n}}^{(e)}(p^{r})\le 2p^{rN_{d,n}-e}.\end{equation}
\begin{lemma}\label{sigmaboundinhatR}
Let $d\ge 2$, $n\ge 3$ and $\ba\in\hat{\fR}_{N_{d,n}}^{(e)}(p^{r})$. Then
\[\sigma'(\ba;p^{r})\ge \frac{1}{p^{(e+1)n}}.\]
\begin{proof}
When $e=r$, we have at least one solution $\bx$ of $f_{\ba}(\bx)\equiv 0$ (mod $p^{r}$). Then $\ba\in\Lambda_{\nu_{d,n}(\bx)}^{(p^{r})}$, so $\sigma'(\ba;p^{r})\ge p^{r}/\varphi(p^{r})^{n+1}$ which is sufficient. Suppose now that $e<r$. Let $\bx\in \fR'_{n+1}(p^{r})$ be a solution of $f_{\ba}(\bx)\equiv 0$ (mod $p^{r}$) and as in the proof of \cite[Lemma 5.7]{bbs2023}, let
\[\fC_{\ba,\bx}(p^{c};p^{r})=\left\{\bb\in\fR_{n+1}(p^{r}):\begin{array}{l l}\bb\equiv\bx\:\text{(mod $p^{c}$)}\\f_{\ba}(\bb)\equiv 0\:\text{(mod $p^{r}$)}\end{array}\right\}.\]
Note that if $c\ge 1$ and $\bb\equiv \bx$ (mod $p^{c}$), then $\bb\in \fR'_{n+1}(p^{r})$, so in fact:
\[\sigma'(\ba;p^{r})\ge \frac{p^{r}\#\fC_{\ba,\bx}(p^{e+1};p^{r})}{\varphi(p^{r})^{n+1}}\ge \frac{\#\fC_{\ba,\bx}(p^{e+1};p^{r})}{p^{rn}}.\]
The result then follows by the same proof as \cite[Lemma 5.7]{bbs2023}.
\end{proof}
\end{lemma}
\begin{proof}[Proof of Lemma \ref{nonarchnotsmall}]
Define
\[\hat{\mathscr{F}}_{\xi}(A)=\frac{1}{\#\bV_{d,n}^{\mathrm{ploc}}(A)}\left\{V\in\bV_{d,n}^{\mathrm{ploc}}(A): \fS'_{V}(A\psi(A))< \frac{C}{\beta(A)}\right\}.\]
To begin with, we follow essentially the same argument as the proof of \cite[Proposition 5.2]{bbs2023}, but using Lemma \ref{sigmamean} instead of \cite[Lemma 5.4]{bbs2023}, Lemma \ref{sigmavar} instead of \cite[Lemma 5.5]{bbs2023}, equation \eqref{hatRebound} instead of \cite[Lemma 5.6]{bbs2023} and Lemma \ref{sigmaboundinhatR} instead of \cite[Lemma 5.7]{bbs2023}. We also use the lower bound $\bV_{d,n}^{\mathrm{ploc}}(A)\gg A^{N_{d,n}}$ from Theorem \ref{positivelocaldensity}.

We follow this argument up until the line preceding \cite[(5.27)]{bbs2023}, where the value of $\kappa$ is chosen to be $1/(4n)$. We instead make the choice $\kappa=1/(3n)-\eps$, and this leads us to a bound of
\[\hat{\mathscr{F}}_{\xi}(A) \ll \frac{1}{\beta(A)^{\kappa}} \prod_{p^{r} \| W} \left(1+O_{C}\left(\frac{1}{p^{1+\eps}}\right)\right)\ll_{C,\eps} \frac{1}{\beta(A)^{\kappa}},\]
which completes the proof.
\end{proof}
\begin{remark}
If we instead choose $\kappa=1/(3n)$ then we can show an upper bound of
\[\frac{(\log_{(3)} A)^{C'}}{\beta(A)^{1/(3n)}},\]
for some $C'>0$ depending on $C$. This is an improvement over Lemma \ref{nonarchnotsmall} when $\beta(A)$ is sufficiently large, but this does not lead to any improvements in Theorem \ref{localcountnotsmall}.
\end{remark}
\subsection{The Archimedean factor is rarely small.}
Let
\begin{equation}\label{defIprime}\bI_{d,n}^{\mathrm{ploc}}=\left\{\ba\in\bR^{N_{d,n}}\setminus\{0\}:\exists\bx\in\bS^{n}\cap(\bR^{+})^{n+1}\: f_{\ba}(\bx)=0\right\},\end{equation}
where
\[\bS^{N}=\left\{\bx\in\bR^{N+1}:\Vert\bx\Vert=1\right\}.\]
Also let
\[\mathscr{U}'_{d,n}(A)=\left\{\ba\in\bZ^{N_{d,n}}\cap\mathcal{B}_{N_{d,n}}(A)\cap\bI_{d,n}^{\mathrm{ploc}}:\mathcal{N}(\ba)\not\subset\bI_{d,n}^{\mathrm{ploc}}\right\},\]
where
\[\mathcal{N}(\ba)=\left\{\by\in\bR^{N_{d,n}}:\by-\ba\in\mathcal{B}_{N_{d,n}}(1)\right\}.\]
\begin{lemma}\label{mathscrUbound}
Let $d\ge 2$ and $n\ge 3$. Then
\[\#\mathscr{U}'_{d,n}(A)\ll A^{N_{d,n}-1}.\]
\begin{proof}
Let $\ba\in\mathscr{U}'_{d,n}(A)$ and $\bb\in\mathcal{N}(\ba)\setminus\bI_{d,n}^{\mathrm{ploc}}$. We define
\[M_{\ba}=\sup\left\{t\in[0,1]:\ba+t(\bb-\ba)\in\bI_{d,n}^{\mathrm{ploc}}\right\}\]
and $\bc=\ba+M_{\ba}(\bb-\ba)$. It follows from the Bolzano-Weierstrass theorem that there is at least one $\bx\in\bS^{n}\cap\bR_{\ge 0}^{n+1}$ such that $f_{\bc}(\bx)=0$. If $\bx\in(\bR^{+})^{n+1}$ for one of these $\bx$, then we may apply the method of the proof of \cite[Lemma 5.8]{bbs2023} to show that $\nabla f_{\bc}(\bx)=0$. In this case, we let $r=0$.

Otherwise, we have that for all solutions of $f_{\bc}(\bx)=0$, we have $x_{i}\le 0$ for some $i$. We may choose $\bx\in\bS^{n}\cap\bR_{\ge 0}^{n+1}$ such that $f_{\bc}(\bx)=0$ and suppose without loss of generality that for some $1\le r\le n$, we have $x_{i}=0$ for $0\le i<r$ and $x_{i}>0$ for $r\le i \le n$. Our aim is to show that, for $r\le i\le n$,
\begin{equation}\label{derivativezero}\partial_{i} f_{\bc}(\bx)=0.\end{equation}
Let $\bh=(h_{0},\dots,h_{n})\in\mathcal{B}_{n+1}(1)$ such that $\langle \bh,\nabla f_{\bc}(\bx)\rangle=0$. Then we claim that there exists $\delta>0$ (independent of $\bh$) such that for all $|t|\le \delta$, there exists $\by=\by(t)\in\bR^{n+1}$ such that $f_{\bc}(\by)=0$ and $\by(t) =\bx+t\bh+O(|t|^{2})$ (the implied constant here may depend on $\bc$). Assuming this claim, we suppose for a contradiction that $\partial_{i} f_{\bc}(\bx)\ne 0$ for some $i\ge r$. Then if we let $h_{0}=\cdots=h_{r-1}=\tau$, for some $\tau>0$, and $h_{j}=0$ for $j\ge r$, $j\ne i$, there will be some unique $h_{i}$ such that $\langle \bh,\nabla f_{\bc}(\bx)\rangle=0$. We may choose $\tau>0$ such that $\| \bh\|=1$. Then for $t>0$ sufficiently small, we have $y_{j}>0$ for all $j$, where $\by=(y_{0},\dots,y_{n})$ is as in the claim. So considering $\by/\Vert \by\Vert$, we get a contradiction with our assumption that $f_{\bc}$ has no solutions in $\bS^{n}\cap(\bR^{+})^{n+1}$. Hence \eqref{derivativezero} holds for all $r\le i\le n$.

To prove the claim, let $\bd=\nabla f_{\bc}(\bx)$. We have the Taylor expansion
\begin{align*}f_{\bc}(\bx+t\bh+s\bd)&=f_{\bc}(\bx)+\langle \bh,\bd\rangle t+\Vert \bd\Vert^{2}s+O\left(|t|^{2}+|s|^{2}\right)\\&=\Vert \bd\Vert^{2}s+O\left(|t|^{2}+|s|^{2}\right).\end{align*}
Now if we choose $c>0$ large enough, $\delta>0$ small enough (in terms of $c$) and $|t|\le \delta$, then when $s=c|t|^{2}$, we will have $f_{\bc}(\bx+t\bh+s\bd)>0$ and when $s=-c|t|^{2}$, we will have $f_{\bc}(\bx+t\bh+s\bd)<0$. Hence by the intermediate value theorem, there will be some $|s|\le c|t|^{2}$ such that $f_{\bc}(\bx+t\bh+s\bd)=0$. This proves the claim, and hence \eqref{derivativezero}.

To summarise, if we define, for $0\le r\le n$,
\[\mathscr{U}_{d,n}^{(r)}(A)=\left\{\ba\in\bZ^{N_{d,n}}\cap\mathcal{B}_{N_{d,n}}(A):\exists \bc\in\mathcal{N}(\ba)\:\exists \bx\in\bS^{n}\begin{array}{l l}x_{i}=0\;\text{for $0\le i< r$}\\\partial_{i} f_{\bc}(\bx)=0\:\text{for $i\ge r$}\end{array}\right\},\]
then we have shown that
\[\#\mathscr{U}'_{d,n}(A)\ll \sum_{r=0}^{n}\#\mathscr{U}_{d,n}^{(r)}.\]

Now, given $\ba\in\mathcal{B}_{N_{d,n}}(A)$, if there are $\bc\in\mathcal{N}(\ba)$ and $\bx\in\bS^{n}$ such that $x_{i}=0$ for $0\le i<r$ and $\partial_{i} f_{\bc}(\bx)=0$ for $i\ge r$, then for any $\by$ in the ball $\Vert \by-\bx\Vert\le 1/A$, we have $|y_{i}|\le 1/A$ for $0\le i< r$ and $|\partial_{i} f_{\ba}(\by)|\ll 1$ for $i\ge r$, because
\begin{align*}|\partial_{i}f_{\ba}(\by)|&\le |\partial_{i}f_{\ba-\bc}(\by)|+|\partial_{i}f_{\bc}(\by)-\partial_{i}f_{\bc}(\bx)|+|\partial_{i} f_{\bc}(\bx)|\\&\ll \Vert \ba-\bc\Vert\cdot\Vert \by\Vert^{d-1}+\Vert \bc\Vert\cdot\Vert \by-\bx\Vert\max\left\{\Vert\bx\Vert,\Vert\by\Vert\right\}^{d-2}.\end{align*}
So for all $\ba\in \mathscr{U}_{d,n}^{(r)}(A)$, there exists $\bx\in \bS^{n}$ such that the ball $\|\by-\bx\| \le 1/A$ is contained in the set
\[\left\{\by\in\bR^{n+1}:\begin{array}{l l}1-\frac{1}{A}\le \Vert \by\Vert\le 1+\frac{1}{A}\\|y_{i}|\le \frac{1}{A}\:\text{for $0\le i< r$}\\|\partial_{i}f_{\ba}(\by)|\ll 1\:\text{for $i\ge r$}\end{array}\right\}\]
and hence the volume of this set is at least $\gg 1/A^{n+1}$. We get
\begin{align*}\#\mathscr{U}_{d,n}^{(r)}(A)&\ll A^{n+1}\sum_{\ba\in\bZ^{N_{d,n}}\cap\mathcal{B}_{N_{d,n}}(A)}\mathrm{vol}\left(\left\{\by\in\bR^{n+1}:\begin{array}{l l}1-\frac{1}{A}\le \Vert \by\Vert\le 1+\frac{1}{A}\\|y_{i}|\le \frac{1}{A}\:\text{for $0\le i< r$}\\|\partial_{i}f_{\ba}(\by)|\ll 1\:\text{for $i\ge r$}\end{array}\right\}\right)\\&\ll A^{n+1}\int_{\mathcal{H}_{n+1}^{(r)}(A)}\#\left\{\ba\in\bZ^{N_{d,n}}\cap\mathcal{B}_{N_{d,n}}(A):\|\partial_{i}f_{\ba}(\by)\|\ll 1\:\text{for $i\ge r$}\right\}\d\by,\end{align*}
where
\[\mathcal{H}_{n+1}^{(r)}(A)=\left\{\by\in\bR^{n+1}:\begin{array}{l l}|y_{i}|\le \frac{1}{A}\:\text{for $0\le i< r$}\\1-\frac{1}{A}\le \Vert \by\Vert\le 1+\frac{1}{A}\end{array}\right\}.\]
Now, for each $\by\in\mathcal{H}_{n+1}^{(r)}(A)$, we have, by a similar argument to the last part of the proof of \cite[Lemma 5.8]{bbs2023} that
\[\#\left\{\ba\in\bZ^{N_{d,n}}\cap\mathcal{B}_{N_{d,n}}(A):|\partial_{i}f_{\ba}(\by)|\ll 1\:\text{for $i\ge r$}\right\}\ll A^{N_{d,n}-(n+1-r)}.\]
If we combine this with the fact that
\[\mathrm{vol}\left(\mathcal{H}_{n+1}^{(r)}(A)\right)\ll \frac{1}{A^{r+1}},\]
we get
\[\#\mathscr{U}_{d,n}^{(r)}(A)\ll A^{N_{d,n}-1},\]
which completes the proof.
\end{proof}
\end{lemma}

Let, for $d,n\ge 1$ integers and $\delta\ge 0$,
\begin{equation}\label{scrBdef}\mathscr{B}_{d,n,\delta}=\left\{\ba\in\mathcal{B}_{N_{d,n}}(1):\begin{array}{l l}\exists \bx\in\bS^{n}\cap(\bR_{\ge 0})^{n+1}\: f_{\ba}(\bx)=0\\\nexists \bx\in\bS^{n}\cap(\bR_{\ge 0})^{n+1}\: f_{\ba}(\bx)=0\:\text{and $x_{i}>\delta$ for all $i$}\end{array}\right\}.\end{equation}
\begin{lemma}\label{badvolbound}
We have, for all $d,n\ge 1$, $\delta\ge 0$, $\mathrm{vol}\left(\mathscr{B}_{d,n,\delta}\right)\ll \delta$.
\begin{proof}
The case $\delta=0$ follows from the $\delta>0$ case by letting $\delta\rightarrow 0$, so suppose $\delta>0$. Since trivially $\vol(\mathscr{B}_{d,n,\delta})\ll 1$, we may also suppose without loss of generality that $\delta<1/(2n+2)$.

Let $\ba\in\mathscr{B}_{d,n,\delta}$. The set of $\bx\in\bS^{n}$ such that $f_{\ba}(\bx)=0$ is compact, so it contains some point $\bx$ such that $\min_{0\le i\le n} x_{i}$ is maximised. Note that, by definition of $\mathscr{B}_{d,n,\delta}$, we have
\begin{equation}\label{minxi}\min_{0\le i\le n} x_{i}\le \delta.\end{equation}
Without loss of generality, assume that for some $r\ge 1$, $x_{i}\le \delta$ for $0\le i< r$ and $x_{i}>\delta$ for all $r\le i\le n$. Since $\delta<1/(n+1)$ and $\bx\in\bS^{n}$, we must have $x_{i}>\delta$ for some $i\le n$, so $r\le n$.

We claim that $\partial_{i}f_{\ba}(\bx)\ll \delta$ for all $i\ge r$. It will then follow that
\[\mathrm{vol}\left(\mathscr{B}_{d,n,\delta}\right)\ll \sum_{r=1}^{n}\mathrm{vol}\left(\mathscr{B}_{d,n,\delta}^{(r)}\right),\]
where
\[\mathscr{B}_{d,n,\delta}^{(r)}=\left\{\ba\in\mathcal{B}_{N_{d,n}}(1):\exists \bx\in\bS^{n}\cap(\bR_{\ge 0})^{n+1}\begin{array}{l l}f_{\ba}(\bx)=0\\ x_{i}\le \delta\:\text{for $0\le i<r$}\\|\partial_{i}f_{\ba}(\bx)|\ll \delta\:\text{for $r\le i \le n$}\end{array}\right\}.\]

To prove the claim, suppose that $\bx$ is as described. Let $\bd=\nabla f_{\ba}(\bx)$. As in the proof of Lemma \ref{mathscrUbound}, we have, for $\bh\in\mathcal{B}_{n+1}(1)$ such that $\langle \bh,\bd\rangle=0$ and sufficiently small $t\ge 0$, there exists $\by(t)\in\bR^{n+1}$ such that $f_{\ba}(\by)=0$ and
\[\by(t)=\bx+t\bh+O\left(t^{2}\right).\]
Let $L>0$ and suppose $|\partial_{i}f_{\ba}(\bx)|>L\delta$ for some $i\ge r$. Choose $h_{0}=\cdots=h_{r-1}=\delta$ and $h_{j}=0$ for $j\ne i$, $j\ge r$, and then choose $h_{i}$ such that $\langle \bh,\bd\rangle=0$. Because $\Vert\ba\Vert\le 1$, we have $\Vert \bd\Vert\ll 1$ and so $h_{i}\ll 1/L$ and so we can fix $L$ sufficiently large so that $\Vert \bh\Vert\le 1/2$. Now consider $\bz(t)=\by(t)/\Vert \by(t)\Vert$. If we can show that $\min_{j\le n}z_{j}>\min_{j\le n}x_{j}$ then we will have a contradiction, and so $|\partial_{i}f_{\ba}(\bx)|\ll \delta$. For $j<r$ we have, when $x_{j}\ne 0$:
\[z_{j}=\frac{x_{j}+\delta t+O\left(t^{2}\right)}{\Vert \by(t)\Vert}\ge \frac{x_{j}+\delta t+O\left(t^{2}\right)}{1+t/2+O\left(t^{2}\right)}=x_{j}\frac{1+(\delta/x_{j})t+O\left(t^{2}/x_{j}\right)}{1+t/2+O\left(t^{2}\right)}.\]
Since $\delta/x_{j}\ge 1$, we see that this is strictly greater than $x_{j}$ when $t>0$ is sufficiently small. When $x_{j}=0$, we have $z_{j}>0$ for sufficiently small $t$. For $j\ge r$, we have $x_{j}>\delta$ and so when $t$ is sufficiently small, $z_{j}>\delta\ge \min_{j\le n}x_{j}$ (recalling \eqref{minxi}). Hence $\min_{j\le n}z_{j}>\min_{j\le n}x_{j}$ as required. This proves the claim.

Now, for $\ba\in\mathscr{B}_{d,n,\delta}^{(r)}$, if $\bx\in\bS^{n}$ with $f_{\ba}(\bx)=0$, $|x_{i}|\le \delta$ for $i<r$ and $|\partial_{i}f_{\ba}(\bx)|\ll \delta$ for $i\ge r$, then for all $\by\in\bS^{n}$, if $\Vert \bx-\by\Vert\le \delta$ then $|f_{\ba}(\by)|\ll \delta$, $|y_{i}|\le 2\delta$ for $i<r$ and $|\partial_{i}f_{\ba}(\by)|\ll \delta$ for $i\ge r$. Hence, if we let
\[\mathcal{S}^{(r)}=\{\bx\in\bS^{n}:|x_{i}|\le 2\delta\:\text{for $i<r$}\},\]
then
\[\mathrm{vol}\left(\mathscr{B}_{d,n,\delta}^{(r)}\right)\ll \frac{1}{\delta^{n}}\int_{\mathcal{S}^{(r)}}\mathrm{vol}\left(\left\{\ba\in\mathcal{B}_{N_{d,n}}(1):\begin{array}{l l}|f_{\ba}(\bx)|\ll \delta\\|\partial_{i}f_{\ba}(\bx)|\ll \delta\:\text{for $i\ge r$}\end{array}\right\}\right)\d\bx,\]
because the volume in the integrand is at least $\delta^{n}$ whenever $\bx\in \mathscr{B}_{d,n,\delta}^{(r)}$.

Now we claim that
\[\mathrm{vol}\left(\left\{\ba\in\mathcal{B}_{N_{d,n}}(1):\begin{array}{l l}|f_{\ba}(\bx)|\ll \delta\\|\partial_{i}f_{\ba}(\bx)|\ll \delta\:\text{for $i\ge r$}\end{array}\right\}\right)\ll \delta^{n+1-r}.\]
To see this, recall that we supposed $\delta<1/(2n+2)$, so if $\bx\in\mathcal{S}^{(r)}$ then $x_{j}\ge 1/(n+1)$ for some $r\le j \le n$ and we may suppose without loss of generality that $j=r$. Write $\ba=(a_{0},\dots,a_{N_{d,n}-1})$. For $r\le i\le n$, let $c_{i}=a_{j_{i}}$ be the entry of $\ba$ corresponding to the monomial $x_{r}^{d-1}x_{i}$. By a similar argument to the last part of the proof of \cite[Lemma 5.8]{bbs2023}, if we fix a choice of all the $a_{j}$ where $j\notin \{j_{r},\dots,j_{n}\}$ then the conditions $|\partial_{i}f_{\ba}(\bx)|\ll \delta$ imply that each $c_{i}$ must lie in an interval of length $O(\delta)$, which proves the claim.

We also have,
\[\mathrm{vol}\left(\mathcal{S}^{(r)}\right)\ll \delta^{r},\]
where we take the $n$-dimensional volume as a subset of $\bS^{n}$.

Hence
\[\mathrm{vol}\left(\mathscr{B}_{d,n,\delta}^{(r)}\right)\ll \delta,\]
which completes the proof.
\end{proof}
\end{lemma}

Now we let, as in \cite{bbs2023},
\[M_{d,n}=\max\left\{\Vert\nabla f_{\ba}(\bx)\Vert :(\ba,\bx)\in\bS^{N_{d,n}-1}\times\bS^{n}\right\}.\]
We also let, for $\lambda>0$,
\begin{align*}\hat{\mathcal{B}}_{d,n,\delta}^{(\lambda)}=&\left\{\ba\in\mathcal{B}_{N_{d,n}}(1):\exists\bx\in\bS^{n}\cap(\bR^{+})^{n+1}\begin{array}{l l}f_{\ba}(\bx)=0\\\lambda\Vert\ba\Vert/2<\Vert\nabla f_{\ba}(\bx)\Vert\le \lambda\Vert\ba\Vert\end{array}\right\}\\ &\setminus \mathscr{B}_{d,n,\delta},\end{align*}
\[\hat{\mathcal{B}}_{d,n,\delta}^{(0)}=\left\{\ba\in\mathcal{B}_{N_{d,n}}(1):\exists\bx\in\bS^{n}\cap(\bR^{+})^{n+1}\begin{array}{l l}f_{\ba}(\bx)=0\\\nabla f_{\ba}(\bx)=\mathbf{0}\end{array}\right\}\setminus \mathscr{B}_{d,n,\delta}\]
and $\hat{\mathcal{B}}_{d,n,\delta}=\mathcal{B}_{N_{d,n}}(1)\setminus \mathscr{B}_{d,n,\delta}$. Note that
\begin{equation}\label{calBdyadicunion}\hat{\mathcal{B}}_{d,n,\delta}\cap\bI_{d,n}^{\mathrm{ploc}}=\left(\bigcup_{\ell=0}^{\infty}\hat{\mathcal{B}}_{d,n,\delta}^{(M_{d,n}/2^{\ell})}\right)\cup\hat{\mathcal{B}}_{d,n,\delta}^{(0)}.\end{equation}
\begin{lemma}\label{volBbound}
Let $d\ge 2$ and $n\ge 3$. For $\lambda\in(0,M_{d,n})$ and $\delta>0$, we have $\mathrm{vol}\left(\hat{\mathcal{B}}_{d,n,\delta}^{(\lambda)}\right)\ll \lambda^{2}$.
We also have $\mathrm{vol}\left(\hat{\mathcal{B}}_{d,n,\delta}^{(0)}\right)=0$.
\begin{proof}
This follows immediately from \cite[Lemma 5.9]{bbs2023}.
\end{proof}
\end{lemma}
\begin{lemma}\label{taulowerbound}
Let $d\ge 2$ and $n\ge 3$. Let also $\gamma>0$ and $\tilde{\lambda}=\min\{\lambda,\delta\}$. For $\lambda\in(0,M_{d,n})$, and $\ba\in\hat{\mathcal{B}}_{d,n,\delta}^{(\lambda)}$, we have,
\[\tau'(\ba;\gamma)\gg\tilde{\lambda}^{n+1}\cdot\mathrm{min}\left\{\gamma,\frac{1}{\lambda\tilde{\lambda}}\right\}.\]
\begin{proof}
We may assume without loss of generality that $\Vert \ba\Vert=1$. Also, since $\ba\in\hat{\mathcal{B}}_{d,n,\delta}^{(\lambda)}$, we may take some $\bx\in\bR^{n+1}$ with $\Vert\bx\Vert=1/2$, satisfying the conditions $f_{\ba}(\bx)=0$, $x_{i}\ge \delta/2$ for all $i$ and
\[\frac{\lambda}{2^{d}}<\Vert\nabla f_{\ba}(\bx)\Vert\le \frac{\lambda}{2^{d-1}}.\]
Recall that
\[\tau'(\ba;\gamma)=\gamma \cdot \mathrm{vol}\left(\left\{\bu\in\mathcal{B}_{n+1}(1)\cap(\bR^{+})^{n+1}:|f_{\ba}(\bu)|\le \frac{\Vert\nu_{d,n}(\bu)\Vert}{2\gamma}\right\}\right).\]
Similarly to the proof of \cite[Lemma 5.10]{bbs2023}, it follows from $\Vert\nu_{d,n}(\bu)\Vert\ge\Vert\bu\Vert^{d}$ that
\[\tau'(\ba;\gamma)\ge \gamma \cdot \mathrm{vol}\left(\left\{\bu\in\left(\mathcal{B}_{n+1}(1)\setminus\mathcal{B}_{n+1}\left(\frac{1}{4}\right)\right)\cap(\bR^{+})^{n+1}:|f_{\ba}(\bu)|\le \frac{1}{2^{2d+1}\gamma}\right\}\right).\]
We also have that if $\Vert \bu-\bx\Vert\le \lambda/(4M_{d,n})$, then $1/4<\Vert \bu\Vert<3/4$. If we also have $\lambda/(4M_{d,n})<\delta/2$ then $u_{j}>0$ for all $j$ and so
\begin{equation}\tau'(\ba;\gamma)\ge \gamma \cdot\mathrm{vol}\left(\left\{\bv\in\mathcal{B}_{n+1}\left(\frac{\lambda}{4M_{d,n}}\right):|f_{\ba}(\bx+\bv)|\le \frac{1}{2^{2d+1}\gamma}\right\}\right).\end{equation}
In the proof of \cite[Lemma 5.10]{bbs2023} it was shown that the right hand side is $\gg \lambda^{n+1}\cdot \min\{\gamma,1/\lambda^{2}\}$, which suffices in this case.

When $\lambda/(4M_{d,n})\ge \delta/2$, we instead have
\[\tau'(\ba;\gamma)\ge \gamma \cdot\mathrm{vol}\left(\left\{\bv\in\mathcal{B}_{n+1}\left(\frac{\delta}{2}\right):|f_{\ba}(\bx+\bv)|\le \frac{1}{2^{2d+1}\gamma}\right\}\right).\]
This leads to
\[\tau'(\ba;\gamma)\ge \gamma \int_{\mathcal{B}_{n}(\delta/4)}\mathcal{W}'_{\ba,\bx}(\bw;\delta,\gamma)\d\bw,\]
where $\bw=(v_{1},\dots,v_{n})$ and
\[\mathcal{W}'_{\ba,\bx}(\bw;\delta,\gamma)=\mathrm{vol}\left(\left\{v_{0}\in[-\delta/4,\delta/4]:|f_{\ba}(\bx+(v_{0},\bw))|\le \frac{1}{2^{2d+1}\gamma}\right\}\right).\]
Without loss of generality, assume that
\[\frac{\lambda}{2^{d}n}<\left| \partial_{0}f_{\ba}(\bx)\right|\le \frac{\lambda}{2^{d-1}}.\]
We make the change of variables
\[v_{0}=-\left(\partial_{0}f_{\ba}(\bx)\right)^{-1}\left(\partial_{1}f_{\ba}(\bx)v_{1}+\cdots+\partial_{n}f_{\ba}(\bx)v_{n}-w_{0}\right),\]
so that $w_{0}=\langle \nabla f_{\ba}(\bx),\bv\rangle$.
Under this transformation, $f_{\ba}(\bx+(v_{0},\bw))$ becomes $w_{0}+P_{\ba,\bx}(w_{0},\bw)$, for some polynomial $P_{\ba,\bx}$. Looking at the first order Taylor expansion, we see that $P_{\ba,\bx}(w_{0},\bw)$ is free of constant and linear terms in $(w_{0},\bw)$. 

There exists $L>0$ such that if $\bw\in\mathcal{B}_{n}(L\delta)$ and $|w_{0}|\le L\lambda\delta$, then $|v_{0}|\le \delta/4$. We then get
\[\mathcal{W}'_{\ba,\bx}(\bw;\delta,\gamma)\gg \frac{1}{\lambda}\cdot \mathrm{vol}\left(\left\{w_{0}\in[-L\lambda\delta,L\lambda\delta]:|w_{0}+P_{\ba,\bx}(w_{0},\bw)|\le \frac{1}{2^{2d+1}\gamma}\right\}\right).\]
By the same argument as in \cite[Lemma 5.10]{bbs2023}, if $L$ is sufficiently small then for $\bw\in\mathcal{B}_{n}(L\delta)$ there exists $\omega_{\ba,\bx}(\bw)\in(-L\lambda\delta/2,L\lambda\delta/2)$ such that $\omega_{\ba,\bx}(\bw)+P_{\ba,\bx}(\omega_{\ba,\bx}(\bw),\bw)=0$. Following the argument from \cite[Lemma 5.10]{bbs2023} further, we find that after possibly making $L$ smaller, if $|w_{0}|\le L\lambda \delta$ then
\[|w_{0}+P_{\ba,\bx}(w_{0},\bw)|\le 2|w_{0}-\omega_{\ba,\bx}(\bw)|,\]
so
\begin{align*}\mathcal{W}'_{\ba,\bx}(\bw;\delta,\gamma)&\gg \frac{1}{\lambda}\cdot\mathrm{vol}\left(\left\{w_{0}\in[-L\lambda\delta/2,L\lambda\delta/2]:|w_{0}-\omega_{\ba,\bx}(\bw)|\le \frac{1}{2^{2d+2}\gamma}\right\}\right)\\ &\gg \frac{1}{\lambda}\cdot\mathrm{min}\left\{\lambda\delta,\frac{1}{\gamma}\right\}.\end{align*}
Hence we get the lower bound
\[\tau'(\ba;\gamma)\gg \frac{\gamma}{\lambda}\cdot \mathrm{min}\left\{\lambda\delta,\frac{1}{\gamma}\right\}\cdot\mathrm{vol}(\mathcal{B}_{n}(L\delta)),\]
which gives the result.
\end{proof}
\end{lemma}
The reason for needing to exclude points in $\mathscr{B}_{d,n,\delta}$ here is because otherwise, the point $\bx$ may have $x_{i}$ arbitrarily small for some $i$, and the $\omega_{\ba,\bx}(\bw)$ we construct might also conspire to be negative for most $\bw$. If these both happen, then most of the solutions that we find to the inequality $|f_{\ba}(\bx+\bv)|\le 1/(2^{2d+1}\gamma)$ will not lie in $(\bR^{+})^{n+1}$. 
\begin{proof}[Proof of Lemma \ref{archnotsmall}]
Recall the definitions \eqref{defarchfactor} and \eqref{defIprime} of $\fJ'_{V}$ and $\bI_{d,n}^{\mathrm{ploc}}$ respectively. Let $B=\psi(A)A$ and recall that $\alpha=\log B$, noting that $\log B \asymp \log A$. Define
\[\hat{\mathscr{I}}_{\beta}(A)=\frac{1}{\#\bV_{d,n}^{\mathrm{ploc}}(A)}\cdot\#\left\{V\in\bV_{d,n}^{\mathrm{ploc}}(A):\fJ'_{V}(A\psi(A))<\frac{C}{\beta(A)}\right\}.\]
By Theorems \ref{positivelocaldensity} and \ref{mathscrUbound}, and definitions \eqref{defarchfactor} and \eqref{defIprime},
\[\hat{\mathscr{I}}_{\beta}(A)\ll \frac{1}{A^{N_{d,n}}}\cdot\#\left\{\ba\in\bZ^{N_{d,n}}\cap\mathcal{B}_{N_{d,n}}(A):\begin{array}{l l}\mathcal{N}(\ba)\subset\bI_{d,n}^{\mathrm{ploc}}\\ \tau'(\ba;\alpha)<\frac{C}{\beta(A)}\end{array}\right\}+\frac{1}{A}.\]
We have that if $\ba\in\bR^{N_{d,n}}$ and $\Vert \ba\Vert \ge 8\alpha$, then for any $\by\in\mathcal{N}(\ba)$,
we have
\[\tau'(\by;2\alpha)\le 2\tau'(\ba;\alpha).\]
The proof of this is identical to the proof of the analogous result for $\tau$ seen in the proof of \cite[Proposition 5.3]{bbs2023}.
We get, following the proof of \cite[Proposition 5.3]{bbs2023} further:
\[\hat{\mathscr{I}}_{\beta}(A)\ll \mathrm{vol}\left(\left\{\by\in\mathcal{B}_{N_{d,n}}(1)\cap\bI_{d,n}^{\mathrm{ploc}}:\tau'(\by;2\alpha)<\frac{2C}{\beta(A)}\right\}\right)+\frac{1}{A}\]
and then, after removing the points of $\mathscr{B}_{d,n,\delta}$,
\begin{equation}\label{mathscrIbound}\hat{\mathscr{I}}_{\beta}(A)\ll \frac{1}{\beta(A)^{\kappa}}\int_{\hat{\mathcal{B}}_{d,n,\delta}\cap \bI_{d,n}^{\mathrm{ploc}}}\frac{\d\by}{\tau'(\by;2\alpha)^{\kappa}}+\frac{1}{A}+\mathrm{vol}(\mathscr{B}_{d,n,\delta}),\end{equation}
where $\kappa\in(0,2/n)$ to be chosen later and we let $\delta=1/\beta(A)^{\eta}$ for some $\eta>0$ also to be chosen.

Then, by \eqref{calBdyadicunion} and the fact that $\mathrm{vol}(\hat{\mathcal{B}}_{d,n,\delta}^{(0)})=0$ from Lemma \ref{volBbound},
\[\int_{\hat{\mathcal{B}}_{d,n,\delta}\cap \bI_{d,n}^{\mathrm{ploc}}}\frac{\d\by}{\tau'(\by;2\alpha)^{\kappa}}\le \sum_{\ell=0}^{\infty}\int_{\hat{\mathcal{B}}_{d,n,\delta}^{(M_{d,n}/2^{\ell})}}\frac{\d\by}{\tau'(\by;2\alpha)^{\kappa}}.\]
Now for each of these integrals, we bound the integrand using Lemma \ref{taulowerbound} with $\lambda=M_{d,n}/2^{\ell}$ and bound the volume of the domain using Lemma \ref{volBbound}. Letting $N(\delta)$ be the smallest integer such that $M_{d,n}/2^{N(\delta)}\le \delta$, we have
\begin{align*}\int_{\hat{\mathcal{B}}_{d,n,\delta}\cap \bI_{d,n}^{\mathrm{ploc}}}\frac{\d\by}{\tau'(\by;2\alpha)^{\kappa}}&\ll \sum_{\ell=N(\delta)}^{\infty}\left(\frac{2^{\ell(n+1)}}{\alpha}+2^{\ell(n-1)}\right)^{\kappa}\left(\frac{M_{d,n}}{2^{\ell}}\right)^{2}\\&+\sum_{\ell=0}^{N(\delta)-1}\left(\frac{1}{2^{\ell}\delta^{n}}+\frac{1}{\alpha\delta^{n+1}}\right)^{\kappa}\left(\frac{M_{d,n}}{2^{\ell}}\right)^{2}.\end{align*}
Now the first term can be shown to be at most $O_{\kappa}(1)$ as long as $\kappa<2/(n+1)$. We choose $\kappa=2/(n+1)-\eps$ for some $\eps>0$. For the second term, we take $\eta=\kappa/(1+n\kappa)$ so that by the assumption $\beta(A)\le (\log A)^{(3n+1)/2}$, we have $1/\delta\le \log A\ll \alpha$, so $\alpha\delta^{n+1}\gg \delta^{n}$. We then obtain a bound of
\[\sum_{\ell=0}^{N(\delta)-1}\left(\frac{1}{2^{\ell}\delta^{n}}+\frac{1}{\alpha\delta^{(n+1)}}\right)^{\kappa}\left(\frac{M_{d,n}}{2^{\ell}}\right)^{2}\ll \delta^{-n\kappa}\sum_{\ell=0}^{N(\delta)-1} 4^{-\ell}\ll \delta^{-n\kappa}=\beta(A)^{\eta n\kappa}.\]
By our choice of $\eta$, we have $\eta n \kappa=\kappa-\eta$. Recalling \eqref{mathscrIbound}, we then have
\[\hat{\mathscr{I}}_{\beta}(A)\ll \frac{1}{\beta(A)^{\eta}}+\frac{1}{A}+\mathrm{vol}(\mathscr{B}_{d,n,\delta}).\]
We then apply Lemma \ref{badvolbound} and see that $\mathrm{vol}(\mathscr{B}_{d,n,\delta})\ll \beta(A)^{-\eta}$. As $\kappa$ increases, so does $\eta$, so from our choice of $\kappa$, we find that
\[\eta=\frac{2/(n+1)}{1+2n/(n+1)}-\eps'=\frac{2}{3n+1}-\eps'\]
for some $\eps'>0$ with $\eps'=o(1)$ as $\eps\rightarrow 0^{+}$. This completes the proof.
\end{proof}
This completes the proof of Theorem \ref{localcountnotsmall}.

\providecommand{\bysame}{\leavevmode\hbox to3em{\hrulefill}\thinspace}

\end{document}